\documentclass[journal]{new-aiaa} 

\usepackage[utf8]{inputenc}
\usepackage{ulem}
\usepackage{amsthm} 
\newtheorem{theorem}{Theorem}[section] 
\newtheorem{lemma}{Lemma}
\newtheorem{Example}{Example}
\newtheorem{definition}{Definition}
\newtheorem{assumption}{Assumption}[section]
\newtheorem{remark}{Remark}[section]

\usepackage{amssymb}
\usepackage{xcolor}
\usepackage{hyperref}
\definecolor{citecolor}{HTML}{005A45}   
\definecolor{linkcolor}{HTML}{0A2F8C}   
\definecolor{urlcolor}{HTML}{7A0F4A}    
\hypersetup{
    colorlinks=true,
    citecolor=citecolor,
    linkcolor=linkcolor,
    urlcolor=urlcolor,
    linktocpage=true
}

\RequirePackage{titlesec}
\titlespacing*{\section}{0pt}{0.85\baselineskip}{0.5\baselineskip}
\titlespacing*{\subsection}{0pt}{0.65\baselineskip}{0.45\baselineskip}
\titlespacing*{\subsubsection}{0pt}{0.55\baselineskip}{0.4\baselineskip}

\usepackage{graphicx}
\usepackage{subcaption}

\usepackage{amsmath}
\usepackage[version=4]{mhchem}
\usepackage{siunitx}
\usepackage{longtable,tabularx}
\renewcommand{\u}{\mathbf{u}} 
\newcommand{\B}{\mathbf{B}} 
\renewcommand{\v}{\mathbf{v}} 
\newcommand{\C}{\mathbf{w}} 
\renewcommand{\L}{\mathbf{L}} 
\newcommand{\R}{\mathbb{R}}
 \definecolor{ForestGreen}{HTML}{228B22} 
\numberwithin{equation}{section}
\numberwithin{lemma}{section}
\numberwithin{definition}{section}

\title{Stability, Convergence, and Error Analysis of Finite Element Methods for 3D Magnetohydrodynamics with $p$-Laplacian Viscosity}

\author{Kim Ngan Le\footnote{School of Mathematics, Monash University, Clayton, Victoria 3800, Melbourne. Email: ngan.le@monash.edu}}
\affil{Monash University, Clayton, Victoria 3800, Australia}

\author{Akash Ashirbad Panda\footnote{Department of Mathematics, IIT Indore, Simrol, India-452020. Email: akashpanda@iiti.ac.in} \ and \ Evana Islam Sarkar\footnote{Department of Mathematics, IIT Bhubaneswar, Khorda-752050. Email: s24ma09008@iitbbs.ac.in}}
\affil{Indian Institute of Technology Bhubaneswar, Khorda, India-752050}

\usepackage{fancyhdr}

\begin{document}

\maketitle

\begin{abstract}
{\small{
This paper develops a fully discrete finite element method for three-dimensional incompressible magnetohydrodynamic (MHD) flows with nonlinear $p$--Laplace viscosity. The scheme combines spatial finite elements with a semi-implicit Euler time discretisation. 
Convergence to a weak solution is proved using a time-translation compactness argument together with Minty's monotonicity method. Under additional regularity assumptions, we derive unconditional 
error estimates for both the velocity and magnetic field, with no coupling restriction between the time-step and mesh-size. The framework extends finite element analysis of incompressible MHD systems to non-Newtonian shear-thickening fluids ($p>2$), while recovering the classical Newtonian case ($p=2$). Finally, 
numerical simulations are provided to validate the theoretical convergence rates and demonstrate the robustness of the proposed method.
}}
\end{abstract}

\vspace{0.3cm}

\noindent
{\bf Keywords and phrases:} \ 3D Magnetohydrodynamics, $p$-Laplace, Finite Element Method, Convergence and Error Analysis, Non-Newtonian Fluid.

\vspace{0.2cm}

\noindent
{\bf AMS subject classification (2020):} \ 65N30, 76W05, 65N12, 65N15, 35J92  



\section{Introduction}

\vspace{0.2cm}

Magnetohydrodynamics (MHD) describes the macroscopic interaction between electrically conducting fluids and electromagnetic fields. It concerns viscous, incompressible fluids capable of conducting electric current, which interact dynamically with induced electromagnetic fields. This theory plays a central role in diverse scientific and engineering applications, including astrophysics, geophysics, nuclear fusion, metallurgical processes, electromagnetic pumping, and the stirring of liquid metals (see, e.g., \cite{Davidson2001}, \cite{Moreau1990} and references therein).

In the classical MHD framework, viscosity is modeled through the standard Laplacian, corresponding to Newtonian fluids with a linear stress–strain relation. However, many realistic media exhibit non-Newtonian behavior, where the viscosity depends nonlinearly on the shear rate. Such phenomena occur, for example, in polymeric melts, plasmas with anomalous diffusion, and astrophysical or geophysical flows where turbulent transport is significant. To capture these effects, we adopt a nonlinear viscosity model based on the $p$-Laplacian operator $\operatorname{div}\!\left(|\nabla \u|^{p-2}\nabla \u\right),$ which generalizes the classical Laplacian. The choice $p=2$ recovers the Newtonian case, while $p>2$ corresponds to shear-thickening and $1<p<2$ to shear-thinning behavior. Incorporating this modification into the MHD equations allows for a more realistic description of complex flows under strong shear and magnetic interactions, extending the applicability of the theory beyond the Newtonian setting. We consider the following problem.

Let $\mathbb{D}\subset\mathbb{R}^3$ be a bounded Lipschitz polyhedral domain. For simplicity, we assume that $\mathbb{D}$ is simply connected and that its boundary $\partial\mathbb{D}$ is connected. We consider the nonstationary incompressible MHD flow with constant density, where the viscous stress tensor is modeled by the $p$-Laplacian. The governing system reads
\begin{align}
    \label{eq-MHD-1} \u_t-\frac{\nu}{\rho}\operatorname{div} (\left|\nabla \u\right|^{p-2}\nabla \u)+\operatorname{div}(\u \otimes \u)+\frac{\nabla P^{\ast}}{\rho}+\frac{1}{\mu\rho}\B\times(\nabla \times \B) &=\mathbf{g} \quad \mbox{in} \hspace{0.1 cm}\mathbb{D}_T,\\ \
    \label{eq-MHD-2}  \B_t+\frac{1}{\mu\sigma} \nabla \times(\nabla \times \B)-\operatorname{\nabla \times(\u\times \B)} & = \mathbf{f} \quad \mbox{in} \hspace{0.1 cm}  \mathbb{D}_T,\\ \
    \label{eq-MHD-3}  \operatorname{div}(\u)&= 0 \quad \mbox{in}\hspace{0.1 cm}  \mathbb{D}_T,\\
     \label{eq-MHD-4}    \operatorname{div}(\B)&= 0 \quad \mbox{in}\hspace{0.1 cm}  \mathbb{D}_T\,,
    \end{align}
where, $\mathbb{D}_T=\mathbb{D}\times(0,T)$ with a given finite final time $T>0$. Here, $\u$ denotes the velocity field, $P^{\ast}$ the pressure, $\B$ the magnetic induction, $\rho$ the (constant) fluid density, $\nu$ the kinematic viscosity, $\mu$ the magnetic permeability, $\sigma$ the electrical conductivity, $\mathbf{g}$ a prescribed body force per unit mass, and $\mathbf{f}$ a divergence-free applied current source. The system is supplemented with the initial and boundary conditions
 \begin{align}
 \u(x, 0)=\u^{0}, \quad \B(x, 0)= \B^{0} \quad \forall x \in \mathbb{D}, \\
 \u = 0, \quad \mathbf{n} \times \operatorname{curl} \B=0, \quad \mathbf{n} \cdot \B=0 \quad \mbox{on} \hspace{0.1 cm} \partial \mathbb{D},
 \end{align}
where, $\mathbf{n}$ denotes the outward unit normal on the boundary $\partial\mathbb{D}$. The initial magnetic field $\B^0$ is assumed to be divergence-free.

\subsection{Related Works}
\paragraph{Existence, Uniqueness and Regularity of MHD.}
 The question of global-in-time regularity for the 3D (incompressible) MHD equations remain a major open problem in mathematical fluid dynamics. For the 2D case, global classical solutions exist for arbitrary initial data $\u_0,\B_0 \in \mathbf{H}^m$ with $m>2$ \cite{CaoWu2011}. In the 3D setting, Sermange and Temam \cite{SermangeTemam1983} established local existence, but global regularity is still unresolved. Subsequent research has focused on identifying regularity criteria that ensure global existence, formulated in terms of velocity, magnetic field, or pressure derivatives (see \cite{ChenMiaoZhang2008, 
 Wu2004}). The pioneering works of He and Xin \cite{HeWang2008} revealed that the velocity field plays a dominant role, proving global regularity criteria based solely on it. Building on their ideas, later studies developed refined component-wise regularity conditions, with sufficient assumptions on either $\u, \B$ or the pressure independently \cite{Zhou2006}.

\paragraph{FEM for MHD.}

Ding et al. \cite{ding2022convergence} analyzed the convergence of a fully discrete variable-density MHD system. Earlier works on the density equation, discretized via the discontinuous Galerkin method, used the compactness framework of Lions and Magenes \cite{Lions1978} 
 as in Walkington \cite{Walkington2005}.
 Finite element formulations employing continuous elements for velocity and Nédélec edge elements for the magnetic field were developed by Dauge et al. 
 \cite{ CostabelDauge1998, CostabelDauge2000}, the latter being essential for nonsmooth magnetic solutions in Lipschitz domains.

The first convergence result for incompressible MHD FEM was established by Prohl \cite{Prohl2008}. For thermally coupled MHD systems, Meir \cite{Meir1995} 
considered stationary problems with constant coefficients, while temperature-dependent models were analyzed in 
\cite{TabataTagami2005}. Error estimates for such systems were derived by Ravindran \cite{Ravindran2019} using a Crank–Nicolson scheme and by Qiu \cite{Qiu2020} via a semi-implicit Euler method. Owing to nonsmooth magnetic fields, $\mathbf{H}(\mathrm{\mathbf{curl}})$-conforming Nédélec elements remain the natural choice \cite{CostabelDauge2000}. 

For variable-density Navier–Stokes and MHD systems, stabilized and convergent FEM schemes have been developed %
\cite{CaiLiLi2021, GuermondQuartapelle2000,  LiQiuYang2022,  BanasProhl2010}, with DG-based approaches providing compactness results \cite{LiuWalkington2007, Walkington2005}. Recent second-order analyses for variable-density MHD using continuous finite elements were given by Li et al. 
\cite{ LiAn2021}.

\paragraph{FEM for \texorpdfstring{$p$}{p}-Laplace.}

Finite element analysis of the parabolic $p$-Laplacian has been extensively developed. Barrett and Liu \cite{BarrettLiu1994} established error estimates for the backward Euler–finite element scheme, later extended by Diening et al. \cite{DieningEbmeyerRuzicka2007} and Berselli and Ruzicka \cite{BerselliRuzicka2022} to discontinuous-in-time methods with optimal quasi-norm bounds. Breit et al. \cite{BreitDieningStornWichmann2021} removed earlier discretization constraints using fractional differentiability in Nikolski spaces. For $p$-Navier–Stokes systems, Berselli et al. \cite{BerselliDieningRuzicka2015} and Eckstein and Ruzicka \cite{EcksteinRuzicka2018} derived similar optimal estimates. Higher-order and Runge–Kutta schemes have also been analyzed, yielding weak or suboptimal error results for nonlinear parabolic problems \cite{EmmrichWroblewska2013}.

\subsection*{Novelties and Difficulties of the Work}

This work develops a rigorous finite element framework for the incompressible MHD system with nonlinear $p$-Laplace viscosity, addressing several analytical challenges arising from the strong nonlinearity in the momentum equation. The nonlinear diffusion term $\operatorname{div} (\left|\nabla \u\right|^{p-2}\nabla \u)$ prevents the direct application of classical 
linear elliptic theory
and significantly complicates both existence and error analysis. To overcome these challenges, we design a semi-implicit Euler–finite element scheme and employ a fixed-point approach to establish the existence of discrete solutions.

A central analytical difficulty in this work arises from the nonlinear \(p\)-Laplace term in the momentum equation. In contrast to the Newtonian case \(p=2\), this nonlinearity prevents a direct control of the discrete time derivative of the velocity, which is essential for compactness and strong convergence arguments. To overcome this obstacle, we employ a time-translation technique, replacing unavailable time-derivative estimates with bounds on temporal shifts of the discrete velocity. This yields uniform stability estimates and allows the use of compactness tools such as the Aubin--Lions lemma. A second key challenge is the identification of the nonlinear limit in the passage from the discrete scheme to the continuous problem. Here, Minty's method is used to rigorously show that the numerical solution converges to the weak solution of the continuous \(p\)-MHD system. Together, these arguments form the core technical contributions of the analysis and distinguish the present study from existing work on standard Newtonian MHD models.


The three-dimensional setting further amplifies the mathematical difficulties, as several key Sobolev embeddings and interpolation inequalities become weaker or more delicate, requiring careful treatment of nonlinear terms. The coupling between velocity and magnetic fields, combined with the $p$-Laplace operator in the momentum equation, introduces additional analytical obstacles. Despite these complexities, we 
show the stability for the discrete problem. Finally, we derive unconditionally
error estimates for the velocity and magnetic fields, where the appearance of $L^p$-norms on the right-hand side requires refined manipulations to suitably apply the discrete Gronwall lemma. Our error analysis requires weaker regularity assumptions on solutions than those commonly imposed in the existing literature.

To the best of our knowledge, this is among the first finite element analyses for a fully discrete three-dimensional incompressible MHD system with $p$-Laplacian viscosity.
The study extends finite element analysis of incompressible MHD flows to the non-Newtonian shear-thickening regime ($p>2$), providing new insights into the numerical treatment of nonlinear 3D $p$-MHD systems.

To validate the theoretical findings, a series of numerical experiments was performed using the finite element software FEniCS implemented in Python. The discretization errors were evaluated in appropriate norms under successive refinements of both the spatial mesh and the time step. The computed convergence rates are in excellent agreement with the theoretical error estimates derived in the preceding sections, thereby confirming the accuracy and effectiveness of the proposed numerical scheme.

\subsection*{Organization Of The Paper}

   The remainder of the paper is organized as follows. Section 2 introduces the notation, recalls several preliminary results, and presents the weak formulation of the governing equations, forming the mathematical framework of our study. Section 3 develops a fully discrete, energy-preserving numerical scheme. The well-posedness and stability properties of this scheme are established in Section 4. Section 5 provides a rigorous convergence analysis under suitable assumptions, while Section 6 concludes with a detailed error analysis, yielding optimal-order error bounds for the proposed fully discrete approximation. Section 7 provides the numerical experiments.

\vspace{0.2cm}

\section{Mathematical Formulation and Functional Setting of the Problem}\label{sec-setting}

This section fixes the functional framework used throughout the paper and derives the weak formulation of the $p$-MHD system \eqref{eq-MHD-1}--\eqref{eq-MHD-4} on which the whole analysis rests. Section \ref{subsec-notation} collects the notation, Section \ref{subsec-spaces} introduces the function spaces, Section \ref{subsec-data} states the standing hypotheses on the data, and Section \ref{subsec-weakform} defines a weak solution and records the equivalent form \eqref{1st weak form}--\eqref{2nd weeak equation} that is used from Section \ref{sec-fem} onwards.

\subsection{Notation}\label{subsec-notation}

Throughout, $\mathbb{D}\subset\R^{3}$ is a bounded polyhedral domain with boundary $\partial\mathbb{D}$ and outward unit normal $\mathbf{n}$, $T>0$ is a fixed final time, and $\mathbb{D}_{T}=\mathbb{D}\times(0,T)$. The flow index is fixed once and for all:
\begin{equation}\label{eq-p-range}
p\ge2,
\end{equation}
which is the shear-thickening regime; the Newtonian case is recovered for $p=2$. We write $p'=\frac{p}{p-1}\in(1,2]$ for the conjugate exponent.\\
\noindent
For an integer $m\ge1$ and $1\le q\le\infty$, $W^{m,q}(\mathbb{D})$ denotes the usual Sobolev space with norm $\|\cdot\|_{m,q}$, and $H^{m}(\mathbb{D})=W^{m,2}(\mathbb{D})$. For a real Banach space $X$ and $1\le q\le\infty$, $L^{q}(0,T;X)$ is the Bochner space with
\[
\|v\|_{L^{q}(0,T;X)}=\left(\int_{0}^{T}\|v(t)\|_{X}^{q}\,dt\right)^{1/q},
\qquad
\|v\|_{L^{\infty}(0,T;X)}=\operatorname*{ess\,sup}_{0\le t\le T}\|v(t)\|_{X}.
\]
We also use the spaces $\mathcal{C}^{m}(\mathbb{D})$ of $m$-times continuously differentiable functions; for further properties of all these spaces we refer to \cite{GiraultRaviart1986}.\\
\noindent
Vector-valued quantities and spaces of vector-valued functions are written in boldface, so that $\u=(u_{1},u_{2},u_{3})$ and $\mathbf{L}^{2}(\mathbb{D}):=\left(L^{2}(\mathbb{D})\right)^{3}$, and similarly for $\mathbf{W}^{m,q}(\mathbb{D})$ and $\mathbf{H}^{m}(\mathbb{D})$. The $L^{2}(\mathbb{D})$ inner product of scalar fields is $(\phi,\psi)=\int_{\mathbb{D}}\phi\,\psi\,dx$, with norm $\|\cdot\|_{L^{2}}$; the same notation is used for the componentwise inner product of vector fields, with norm $\|\cdot\|_{\L^{2}}$. For $\mathbf{v}\in\mathbf{W}^{l,2}(\mathbb{D})$ we abbreviate $\|\mathbf{v}\|_{l,2}:=\|\mathbf{v}\|_{\mathbf{W}^{l,2}}$, and similarly $\|\mathbf{v}\|_{l,p}$ in the $\mathbf{W}^{l,p}$ scale. Finally, $\nabla\times$ and $\operatorname{div}$ denote the curl and divergence operators, and $C$ denotes a generic positive constant whose value may change from occurrence to occurrence.

\subsection{Function spaces}\label{subsec-spaces}

The spaces below fall into three groups: those carrying the velocity and pressure, those carrying the magnetic field, and a family of solenoidal spaces needed for the compactness arguments of Section \ref{sec-convergence}. We list them by role.

\medskip
\noindent\textbf{(a) Velocity and pressure.} The natural energy space for the $p$-Laplace viscosity is $\mathbf{W}^{1,p}_{0}$, and the pressure is normalised to have zero mean:
\begin{equation}\label{eq-spaces-velocity}
\mathcal{X}=\mathbf{W}^{1,p}_{0}(\mathbb{D}),
\qquad
\mathcal{X}_{0}=\left\{\mathbf{v}\in\mathcal{X}:\operatorname{div}\mathbf{v}=0\ \text{a.e. in }\mathbb{D}\right\},
\qquad
Q=L^{2}_{0}(\mathbb{D})=\left\{q\in L^{2}(\mathbb{D}):\textstyle\int_{\mathbb{D}}q\,dx=0\right\}.
\end{equation}

\medskip
\noindent\textbf{(b) Magnetic field.} Since the magnetic diffusion acts through the curl operator only, the magnetic field lives in
\begin{equation}\label{eq-spaces-magnetic}
\mathcal{Y}=\mathbf{H}(\operatorname{\mathbf{curl}};\mathbb{D})=\left\{\mathbf{w}\in\mathbf{L}^{2}(\mathbb{D}):\nabla\times\mathbf{w}\in\mathbf{L}^{2}(\mathbb{D})\right\},
\qquad
\mathcal{Y}_{0}=\left\{\mathbf{w}\in\mathcal{Y}:\mathbf{w}\times\mathbf{n}=0\ \text{on }\partial\mathbb{D}\right\},
\end{equation}
equipped with the graph norm
\[
\|\mathbf{w}\|_{\mathcal{Y}}=\left(\|\mathbf{w}\|^{2}_{\L^{2}}+\|\nabla\times\mathbf{w}\|^{2}_{\L^{2}}\right)^{1/2}.
\]

\medskip
\noindent\textbf{(c) Solenoidal and auxiliary spaces.} The remaining spaces encode the two divergence constraints \eqref{eq-MHD-3}--\eqref{eq-MHD-4} at different levels of regularity, and are used only in the convergence analysis of Section \ref{sec-convergence}; a reader may postpone them until then. We first recall the divergence spaces
\[
\mathbf{H}(\operatorname{div};\mathbb{D})=\left\{\mathbf{w}\in\mathbf{L}^{2}(\mathbb{D}):\operatorname{div}\mathbf{w}\in L^{2}(\mathbb{D})\right\},
\quad
\mathbf{H}_{0}(\operatorname{div};\mathbb{D})=\left\{\mathbf{w}\in\mathbf{H}(\operatorname{div};\mathbb{D}):\mathbf{w}\cdot\mathbf{n}=0\ \text{on }\partial\mathbb{D}\right\},
\]
in terms of which the magnetic field lies, in the limit, in the solenoidal space
\[
\mathcal{K}=\left\{\mathbf{b}\in\mathcal{Y}\cap\mathbf{H}_{0}(\operatorname{div};\mathbb{D}):\operatorname{div}\mathbf{b}=0\ \text{a.e. in }\mathbb{D}\right\},
\]
the range of the Hodge-type operator $\mathcal{L}$ introduced in Section \ref{sec-convergence}. The compactness arguments of that section use, in addition, the spaces of smooth solenoidal test fields
\[
\mathcal{J}=\left\{\mathbf{w}\in\mathcal{C}^{\infty}_{0}(\mathbb{D}):\operatorname{div}\mathbf{w}=0\right\},
\qquad
\mathcal{G}=\left\{\mathbf{w}\in\mathcal{C}^{\infty}(\overline{\mathbb{D}}):\operatorname{div}\mathbf{w}=0\ \text{in }\mathbb{D},\ \mathbf{w}\cdot\mathbf{n}=0\ \text{on }\partial\mathbb{D}\right\},
\]
which serve as test functions for the velocity and the magnetic limit respectively, together with the space
\[
\mathbb{H}=\left\{\mathbf{v}\in\mathbf{L}^{2}(\mathbb{D}):\operatorname{div}\mathbf{v}=0\ \text{weakly},\ \mathbf{v}\cdot\mathbf{n}=0\ \text{on }\partial\mathbb{D}\right\},
\]
which carries the $L^{\infty}$-in-time regularity of $(\u,\B)$ in Definition \ref{def-weak-solution}.

\subsection{Standing hypotheses on the data}\label{subsec-data}

Throughout the paper we assume that the initial data and the source terms satisfy
\begin{equation}\label{eq-data-assumption}
\u^{0},\B^{0}\in\mathbf{L}^{2}(\mathbb{D}),
\qquad
\mathbf{g},\mathbf{f}\in L^{2}\!\left(0,T;\mathbf{L}^{2}(\mathbb{D})\right),
\qquad
\operatorname{div}\B^{0}=0,
\qquad
\operatorname{div}\mathbf{f}=0\ \text{ in }\mathbb{D}_{T}.
\end{equation}
The solenoidality of $\mathbf{f}$ is not a modelling convenience but a consistency requirement. Taking the divergence of \eqref{eq-MHD-2} and using $\operatorname{div}\nabla\times=0$ gives
\[
\partial_{t}\left(\operatorname{div}\B\right)=\operatorname{div}\mathbf{f}\quad\text{in }\mathbb{D}_{T},
\]
so the constraint \eqref{eq-MHD-4} propagates from the divergence-free initial datum $\B^{0}$ if and only if $\operatorname{div}\mathbf{f}=0$. No analogous condition is needed for $\mathbf{g}$, since the incompressibility constraint \eqref{eq-MHD-3} is enforced by the pressure rather than transported by the equation.

\subsection{Weak formulation}\label{subsec-weakform}

We first introduce the two forms through which the nonlinearities are expressed. For $\u,\mathbf{w},\mathbf{v}\in\mathcal{X}$ we set
\begin{align}
\left(a_{0}(\u),\nabla\mathbf{v}\right)
&:=\frac{\nu}{\rho}\int_{\mathbb{D}}\left|\nabla\u\right|^{p-2}\nabla\u:\nabla\mathbf{v}\,dx,
\label{eq-a0-def}\\
a_{1}(\u,\mathbf{w},\mathbf{v})
&:=\frac12\left\{\int_{\mathbb{D}}\left((\u\cdot\nabla)\mathbf{w}\right)\cdot\mathbf{v}\,dx
-\int_{\mathbb{D}}\left((\u\cdot\nabla)\mathbf{v}\right)\cdot\mathbf{w}\,dx\right\}.
\label{eq-a1-def}
\end{align}
The operator $a_{0}$ is the $p$-Laplace viscosity; it is monotone for $p\ge2$, which is the source of all dissipation in the velocity equation. The form $a_{1}$ is the skew-symmetric, or Temam, form of the convective term. Its decisive property, immediate from \eqref{eq-a1-def} and used repeatedly in Sections \ref{sec-wellposed}--\ref{ lemma-error bounds}, is
\begin{equation}\label{eq-a1-skew}
a_{1}(\u,\mathbf{v},\mathbf{v})=0
\qquad\text{for all }\u,\mathbf{v}\in\mathcal{X},
\end{equation}
which holds without any divergence constraint on the first argument. This is why $a_{1}$, rather than the convective term in its divergence form, is used both in the weak formulation and in the discrete scheme.
\\ \noindent
We can now define the notion of solution used throughout.

\begin{definition}[Weak solution]\label{def-weak-solution}
Assume \eqref{eq-data-assumption}. A pair $(\u,\B)$ is a weak solution of \eqref{eq-MHD-1}--\eqref{eq-MHD-4} if
\begin{itemize}
\item[(i)] $\displaystyle \u\in L^{\infty}(0,T;\mathbb{H})\cap L^{2}(0,T;\mathcal{X}_{0})$ \ and \ $\displaystyle \B\in L^{\infty}(0,T;\mathbb{H})\cap L^{2}(0,T;\mathcal{K})$; and
\item[(ii)] for every $\mathbf{v}\in\left\{\mathcal{C}^{\infty}_{0}(\mathbb{D}\times[0,T)):\operatorname{div}\mathbf{v}=0\right\}$ and every $\mathbf{C}\in\mathcal{C}^{\infty}_{0}(\mathbb{D}\times[0,T))$ there holds
\begin{align}
&\int_{0}^{T}\!\!\int_{\mathbb{D}}\left(-\u\cdot\partial_{t}\mathbf{v}
+\frac{\nu}{\rho}\left|\nabla\u\right|^{p-2}\nabla\u:\nabla\mathbf{v}
-\left(\u\otimes\u\right):\nabla\mathbf{v}
+\frac{1}{\mu\rho}\left(\B\times(\nabla\times\B)\right)\cdot\mathbf{v}\right)dx\,dt
\nonumber\\
&\hspace{4.2cm}=\int_{0}^{T}\!\!\int_{\mathbb{D}}\mathbf{g}\cdot\mathbf{v}\,dx\,dt
+\int_{\mathbb{D}}\u^{0}\cdot\mathbf{v}(\mathbf{x},0)\,dx,
\label{eq-weak-distributional-1}
\end{align}
\begin{align}
&\int_{0}^{T}\!\!\int_{\mathbb{D}}\left(-\B\cdot\partial_{t}\mathbf{C}
+\frac{1}{\mu\sigma}\left(\nabla\times\B\right)\cdot\left(\nabla\times\mathbf{C}\right)
-\left(\u\times\B\right)\cdot\left(\nabla\times\mathbf{C}\right)\right)dx\,dt
\nonumber\\
&\hspace{4.2cm}=\int_{0}^{T}\!\!\int_{\mathbb{D}}\mathbf{f}\cdot\mathbf{C}\,dx\,dt
+\int_{\mathbb{D}}\B^{0}\cdot\mathbf{C}(\mathbf{x},0)\,dx.
\label{eq-weak-distributional-2}
\end{align}
\end{itemize}
\end{definition}
\noindent
Formulation \eqref{eq-weak-distributional-1}--\eqref{eq-weak-distributional-2} is stated in space--time distributional form, which is convenient for passing to the limit in Section \ref{sec-convergence} because it requires no time regularity of $(\u,\B)$ beyond \textup{(i)}. For the stability and error analysis, however, it is more convenient to work with a formulation holding at almost every time. The two are related as follows.

\begin{remark}[Equivalent pointwise-in-time formulation]\label{rem-equivalence}
Let $(\u,\B)$ be a weak solution in the sense of Definition \ref{def-weak-solution} which in addition satisfies $\u_{t}\in L^{p'}(0,T;\mathcal{X}_{0}')$ and $\B_{t}\in L^{2}(0,T;\mathcal{K}')$. Then \eqref{eq-weak-distributional-1}--\eqref{eq-weak-distributional-2} hold if and only if, for almost every $t\in(0,T)$,
\begin{align}
\left(\u_{t},\mathbf{v}\right)+\left(a_{0}(\u),\nabla\mathbf{v}\right)+a_{1}(\u,\u,\mathbf{v})
+\frac{1}{\mu\rho}\left(\B\times(\nabla\times\B),\mathbf{v}\right)&=\left(\mathbf{g},\mathbf{v}\right),
\label{1st weak form}\\
\left(\B_{t},\mathbf{C}\right)+\frac{1}{\mu\sigma}\left(\nabla\times\B,\nabla\times\mathbf{C}\right)
-\left(\u\times\B,\nabla\times\mathbf{C}\right)&=\left(\mathbf{f},\mathbf{C}\right),
\label{2nd weeak equation}
\end{align}
for all $\mathbf{v}\in\mathcal{X}_{0}$ and all $\mathbf{C}\in\mathcal{K}$, together with $\u(0)=\u^{0}$ and $\B(0)=\B^{0}$.
\\ \noindent
Only the convective term requires comment. Integrating by parts and using $\operatorname{div}\u=0$ together with $\mathbf{v}|_{\partial\mathbb{D}}=0$ gives
\[
\int_{\mathbb{D}}\left((\u\cdot\nabla)\mathbf{v}\right)\cdot\u\,dx
=\int_{\mathbb{D}}\left(\u\otimes\u\right):\nabla\mathbf{v}\,dx,
\qquad
\int_{\mathbb{D}}\left((\u\cdot\nabla)\u\right)\cdot\mathbf{v}\,dx
=-\int_{\mathbb{D}}\left(\u\otimes\u\right):\nabla\mathbf{v}\,dx .
\]
The two integrals in \eqref{eq-a1-def} therefore differ only by sign, and
\[
a_{1}(\u,\u,\mathbf{v})=-\int_{\mathbb{D}}\left(\u\otimes\u\right):\nabla\mathbf{v}\,dx,
\]
which is precisely the convective term appearing in \eqref{eq-weak-distributional-1}. The remaining terms transform by the divergence theorem and the boundary conditions, and the initial conditions are recovered in the usual way by integrating the time derivative by parts.
\end{remark}
\noindent
Equations \eqref{1st weak form}--\eqref{2nd weeak equation} are the form of the continuous problem that will be discretised in Section \ref{sec-fem} and used as the reference solution in the error analysis of Section \ref{ lemma-error bounds}.

\section{A Mixed Finite Element Method for the \texorpdfstring{$p$}{p}-MHD system \texorpdfstring{\eqref{1st weak form}--\eqref{2nd weeak equation}}{}}
\label{sec-fem}

\vspace{0.2cm}
\noindent
In this section, we introduce the mixed finite element framework proposed by Gerbeau \cite[Section 3.4.2]{GerbeauLeBrisLelievre2006} for the spatial discretization of the problem \eqref{1st weak form}–\eqref{2nd weeak equation} using a semi-implicit time-stepping strategy.\\
\noindent
Let $\mathbb{D}$ be a Lipschitz polyhedral domain and let $\tau_h$ denote a regular, quasi-uniform triangulation of $\mathbb{D}$ with mesh parameter $h$ (maximum diameter of the elements). For a nonnegative integer $k$, we denote by $P_k(K)$ the space of polynomials of degree at most $k$ on an element $K$, and by $\tilde{P}_k(K)$ the space of homogeneous polynomials of degree $k$.\\
\noindent
The velocity field is approximated using continuous finite elements and evolved in time via an Euler semi-implicit scheme. The chosen velocity–pressure pair satisfies the classical inf–sup stability condition. For the magnetic field, we employ Nédélec edge elements to properly capture the $\operatorname{curl}$-conformity required by the structure of Maxwell-type terms.
Let us denote
\begin{align*}
\mathcal{X}_{h} \subset\left\{\v_{h} \in \mathcal{X},\left.\v_{h}\right|_{K} \in P_{k}(K), \forall K \in \mathcal{T}_{h}\right\}, \quad \mbox{ and } \quad Q_{h} \subset\left\{q_{h} \in Q,\left.q_{h}\right|_{K} \in P_{k}(K), \forall K \in \mathcal{T}_{h}\right\}\,
\end{align*}
be a pair of spaces satisfying the Babuska-Brezzi condition \cite{LiuWalkington2007}.
\noindent
The discrete kernel space of the divergence operator defined by
$$\mathcal{X}_{0 h}=\left\{\v_{h} \in \mathcal{X}_{h},\left( \operatorname{div} \v_{h}, q_{h}\right)=0 \quad \forall q_{h} \in Q_{h}\right\}.$$
 For the magnetic induction, the discretization is based on Nédelec finite elements, and it is defined as
$$
\mathcal{Y}_{h}=\left\{\mathbf{A}_{h} \in \mathcal{Y},\left.\mathbf{A}_{h}\right|_{K} \in \mathcal{N}_{k}(K) \quad \forall K \in \mathcal{T}_{h}\right\},
$$
where, $\mathcal{N}_{k}(K)={P}_{k-1}(K)\oplus\mathcal{O}_{k}(K)$ and $\mathcal{O}_{k}(K)=\left\{\mathbf{q}\in \tilde{{P}}_{k}(K),\mathbf{q(x)}\cdot \mathbf{x}=0 \quad \mbox{on} \quad K\right\}$ (see \cite[Section 3.4.2]{GiraultRaviart1986} for details).
 Setting, $ S_{h}=\left\{C\in H^{1}(\mathbb{D}) \cap L_{0}^{2}(\mathbb{D}), C \in P_{k}(K), \forall K \in \mathcal{T}_{h}\right\}$, see \cite{Monk2003}, we introduce
the discretely solenoidal function space
    $$\mathcal{Y}_{0 h}=\left\{\mathbf{c} \in \mathcal{Y}_{h},(\mathbf{c}, \nabla S)=0 \quad \forall S \in S_{h}\right\}.$$ 
    Another link between the spaces $\mathcal{Y}_{0 h}$ and $\mathcal{K}$ is accomplished by
    the Hodge mapping (see e.g., \cite{Hiptmair2002})
    $$
    \mathcal{L}: \mathbf{H}_{0} (\operatorname{\mathbf{curl}} ; \mathbb{D}) \rightarrow \mathcal{K},
    $$
    such that
    $$
    \nabla \times \mathcal{L}(\mathbf{b})=\nabla \times \mathbf{b} \quad \forall \mathbf{b} \in \mathcal{Y}.
    $$
    The Hodge Mapping satisfies the following approximation property,
\begin{equation}\label{ine-zbh}
    \| \mathbf{b}_{h}-\mathcal{L}(\mathbf{b}_{h})\|_{\mathbf{L}^2} \leq c h^{\frac{1}{2}+s}\| \nabla \times\mathbf{b}_{h} \|_{\mathbf{L}^2} \quad \forall \mathbf{b}_{h} \in \mathcal{Y}_{0 h},
\end{equation} 
for some $s \equiv s(\mathbb{D})>0.$

\subsection{The Mixed Finite Element Scheme}

\vspace{0.1cm}
\noindent
Now, we are ready to introduce the fully discrete scheme, which combines spatial finite elements with a semi-implicit Euler time discretisation. Let $N>0$ be a fixed integer number and let $\Delta t=T / N$ be the time-step size. Thus, $t_{m}=m \Delta t$ denote the discrete time levels with $m=1,2, \ldots, N$. For the sake of convenience, $d_{t}\mathbf{u}^{m}=\frac{\mathbf{u}^{m}-\mathbf{u}^{m-1}}{\Delta t}$. 

\noindent
Let   $\mathcal{P}_{h}$ be the discrete Stokes projection onto the discrete divergence-free space $\mathcal{X}_{0h}$. The projection  $\mathcal{P}_{h}$ satisfies the approximation rate, for $0<s\leq k$
\begin{align*}
\| \mathbf{v}-\mathcal{P}_{h}(\mathbf{v}_{h})\|_{\mathbf{W}^{1,p}} \leq c h^{s}\| \mathbf{v} \|_{\mathbf{W}^{1+s,p}} \quad \forall \,  \mathbf{v} \in \mathbf{W}^{1+s,p},
\end{align*}
For the initial velocity $\mathbf{u}^{0}\in \mathbf{W}^{1+s,p}$, let $\mathbf{u}_{h}^{0}:=\mathcal{P}_{h} \mathbf{u}^{0}$.
For the initial magnetic field $\B^{0}\in \mathbf{W}^{s,2}$, let $\B_{h}^{0}:= \mathcal{F}_{h} \B^{0}$.
Let
\[
{\mathbf g^m = \mathcal{P}_{h}\,\mathbf g(t_m), \qquad
\mathbf f^m = \mathcal{F}_{h}\,\mathbf f(t_m),}
\]
be the finite element approximations of the source terms
$\mathbf f$ and $\mathbf g$ at time $t_m$. Then,
\[
\Delta t \sum_{m=1}^{N}
\left(
\|\mathbf g^m\|_{\L^2}^2
+\frac{1}{\mu\rho}\|\mathbf f^m\|_{\L^2}^2
\right)
\le C(\mathbf f,\mathbf g),
\]
where the constant $C(\mathbf f,\mathbf g)$ depends only on the norms of $\mathbf f$ and $\mathbf g$.
For any $0<m\leq N$, we aim to find that $\left(\mathbf{u}_{h}^{m},\mathbf{B}_{h}^{m}\right) \in \left(\mathcal{X}_{0h},\mathcal{Y}_{h} \right)$ such that for any $\left(\mathbf{v}_{h},\mathbf{w}_{h}\right) \in \left(\mathcal{X}_{0h},\mathcal{Y}_{h}\right)$


 %
 %
 \begin{align}
   \nonumber  \label{eq-FEM-1} &\left(d_{t} \u_{h}^{m}, \mathbf{v}_{h}\right)+a_{1}\left(\u_{h}^{m-1}, \u_{h}^{m}, \mathbf{v}_{h}\right)+{\frac{\nu}{\rho}\left(a_{0}(\u^m_h),\nabla\v_h\right) }\\  
\qquad &
+\frac{1}{\mu\rho}\left(\mathbf{B}_{h}^{m-1} \times (\nabla \times \mathbf{B}_{h}^{m}), \mathbf{v}_{h}\right)=\left(\mathbf{g}^{m}, \mathbf{v}_{h}\right),  \\ 
&\label{eq-FEM-2}  \left(d_{t} \B_{h}^{m}, \mathbf{w}_{h}\right)+\frac{1}{\mu\sigma}\left(\nabla \times \B_{h}^{m},\nabla \times \mathbf{w}_{h}\right)-\left(\u_{h}^{m} \times \B_{h}^{m-1}, \nabla \times\mathbf{w}_{h}\right)=\left(\mathbf{f}^{m}, \mathbf{w}_{h}\right).
     \end{align}
 %
 %
By choosing $\C_h=\nabla s_h$ in \eqref{eq-FEM-2}, where $s_h\in S_h$, it is easy deduce $(d_t\B^{m}_{h}, \nabla s_h)=0, m=1,2,3,\cdots, N$. Making use of $(\B^{0}_{h},\nabla s_h)=0$, we are able to obtain that $(\B^{m}_{h}, \nabla s_h)=0, m=1,2,3,\cdots, N$. Then, the fully discrete scheme \eqref{eq-FEM-2} satisfy the weakly divergence free property.


\section{Well-Posedness and Stability of the Fully Discrete Scheme
\texorpdfstring{\eqref{eq-FEM-1}--\eqref{eq-FEM-2}}{}}
\label{sec-wellposed}

We first establish the existence and uniqueness of the fully discrete solution.
Since the discrete problem contains the nonlinear $p$-Laplacian operator, its
solvability is obtained through a finite-dimensional fixed-point argument.
We then derive the stability estimates required for the subsequent convergence
analysis.

\subsection{Existence and uniqueness of the discrete solution}

\begin{theorem}[Well-posedness of the fully discrete scheme]
\label{thm-existence}
For every $m=1,\ldots,N$, the scheme
\eqref{eq-FEM-1}--\eqref{eq-FEM-2} admits a unique solution
\[
(\u_h^m,\B_h^m)\in\mathcal X_{0h}\times\mathcal Y_h.
\]
The result holds for arbitrary $\Delta t>0$ and $h>0$; in particular, no
coupling condition between the time-step and the mesh-size is required.
\end{theorem}

\begin{proof}
Recall the equations \eqref{eq-FEM-1} and \eqref{eq-FEM-2}
    \begin{align}
   &\left(d_{t} \u_{h}^{m}, \mathbf{v}_{h}\right)+a_{1}\left(\u_{h}^{m-1}, \u_{h}^{m}, \mathbf{v}_{h}\right)+{\frac{\nu}{\rho}\left(a_{0}(\u^m_h),\nabla\v_h\right) }
  +\frac{1}{\mu\rho}\left(\mathbf{B}_{h}^{m-1} \times (\nabla \times \mathbf{B}_{h}^{m}), \mathbf{v}_{h}\right)=\left(\mathbf{g}^{m}, \mathbf{v}_{h}\right), \\ 
&\left(d_{t} \B_{h}^{m}, \mathbf{w}_{h}\right)+\frac{1}{\mu\sigma}\left(\nabla \times \B_{h}^{m},\nabla \times \mathbf{w}_{h}\right)-\left(\u_{h}^{m} \times \B_{h}^{m-1}, \nabla \times\mathbf{w}_{h}\right)=\left(\mathbf{f}^{m}, \mathbf{w}_{h}\right).
 \end{align}
 Let us denote 
 $$
 H^{m}_{h} :=\left(\u^{m}_{h},\frac{1}{\mu\rho}\B^{m}_{h}\right) \quad \mbox{and} \quad \phi_{h} :=\left(\v_{h},\C_{h}\right)
 $$
 Putting, $d_{t}\mathbf{\u}^{m}=\frac{\mathbf{\u}^{m}-\mathbf{\u}^{m-1}}{\Delta t}$ and 
 $d_{t}\mathbf{B}^{m}=\frac{\mathbf{B}^{m}-\mathbf{B}^{m-1}}{\Delta t}$ we obtain,
 \begin{align}
    & \underbrace{\frac{1}{\Delta t}(\u^{m}_{h},\v_{h})+a_{1}\left(\u^{m-1}_{h},\u^{m}_{h},\v_{h}\right)+{\frac{\nu}{\rho}\left(a_{0}(\u^m_h),\nabla\v_h\right)} +\frac{1}{\mu\rho}\left(\B^{m-1}_{h}\times(\nabla \times\B^{m}_{h}), \v_{h}\right)}_{{\mbox{LHS}_{1}} (H^{m}_h,\v_{h})} \nonumber
    \\ &\qquad = \underbrace{\frac{1}{\Delta t}(\u^{m-1}_{h},\v_{h})+(\mathbf{g}^{m},\v_{h}),}_{{\mbox{RHS}_{1}} (H^{m-1}_h,\v_{h})} \label{eq-existence diff-1}
\end{align}
\begin{align}\label{eq-existence diff-2} 
     & \underbrace{\frac{1}{\Delta t}\left(\B^{m}_{h}, \C_{h}\right)+\frac{1}{\mu\sigma}\left(\nabla \times\B^{m}_{h},\nabla \times\C_{h}\right)-\left(\u_{h}^{m} \times \B_{h}^{m-1}, \nabla \times\mathbf{w}_{h}\right)}_{{\mbox{LHS}_{2}} (H^{m}_h,\C_{h})} = \underbrace{\frac{1}{\Delta t}\left(\B^{m-1}_{h}, \C_{h}\right)+\left(\mathbf{f}^{m}, \mathbf{w}_{h}\right)}_{{\mbox{RHS}_{2}} (H^{m-1}_h,\C_{h})}.
\end{align}
Adding the above two equations yields
{\small{
    \begin{align}
    \nonumber   & \frac{1}{\Delta t}\left(\u^{m}_{h},\v_{h}\right)+a_{1}\left(\u^{m-1}_{h},\u^{m}_{h},\v_{h}\right)+{\frac{\nu}{\rho}\left(a_{0}(\u^m_h),\nabla\v_h\right)}+\frac{1}{\mu\rho}\left(\B^{m-1}_{h}\times(\nabla \times\B^{m}_{h}), \v_{h}\right)+\frac{1}{\Delta t}(\B^{m}_{h}, \C_{h})\\&+\frac{1}{\mu\sigma}\left(\nabla \times\B^{m}_{h},\nabla \times\C_{h}\right)-\left(\u_{h}^{m} \times \B_{h}^{m-1}, \nabla \times\mathbf{w}_{h}\right)=\frac{1}{\Delta t}\left(\u^{m-1}_{h},\v_{h}\right)+(\mathbf{g}^{m},\v_{h})+\frac{1}{\Delta t}\left(\B^{m-1}_{h}, \C_{h}\right)+\left(\mathbf{f}^{m}, \mathbf{w}_{h}\right).
    \end{align}
    }}
   We define the following:
     \begin{align*}
     \left(F(H^{m}_{h}),\phi_{h}\right) &:=(\mbox{LHS}_{1}(H^{m}_{h}),\v_{h})+(\mbox{LHS}_{2}(H^{m}_{h}),\C_{h})\\ &=\frac{1}{\Delta t}(\u^{m}_{h},\v_{h})+a_{1}\left(\u^{m-1}_{h},\u^{m}_{h},\v_{h}\right)+{\frac{\nu}{\rho}\left(a_{0}(\u^m_h),\nabla\v_h\right)}+\frac{1}{\mu\rho}\left(\B^{m-1}\times(\nabla \times\B^{m}_{h}), \v_{h}\right)\\ &\qquad+\frac{1}{\Delta t}(\B^{m}_{h}, \C_{h})+\frac{1}{\mu\sigma}\left(\nabla \times\B^{m}_{h},\nabla \times\C_{h}\right)-\left(\u_{h}^{m} \times \B_{h}^{m-1}, \nabla \times\mathbf{w}_{h}\right)\,,\\
     (L(H^{m-1}_{h}),\phi_{h}) &:=(\mbox{RHS}_{1}(H^{m-1}_{h}),\v_{h})+(\mbox{RHS}_{2}(H^{m-1}_{h}),\C_{h})
    \\ &=\frac{1}{\Delta t}(\u^{m-1}_{h},\v_{h})+(\mathbf{g}^{m},\v_{h})+\frac{1}{\Delta t}(\B^{m-1}_{h}, \C_{h})+\left(\mathbf{f}^{m}, \mathbf{w}_{h}\right).
 \end{align*}

\noindent
 One can observe $\frac{(F(H^{m}_{h}),\phi_{h})}{\|H^{m}_{h}\|_{\L^2}} =\frac{(L(H^{m-1}_{h}),\phi_{h})}{\|H^{m}_{h}\|_{\L^2}}$. Thus,
 \begin{align}
       (\mbox{LHS}_{1}(H^{m}_{h}),\v_{h})+(\mbox{LHS}_{2}(H^{m}_{h}),\C_{h})&=(\mbox{RHS}_{1}(H^{m-1}_{h}),\v_{h})+(\mbox{RHS}_{2}(H^{m-1}_{h}),\C_{h}).\label{eq-existence 1}
       \end{align}
       Now, take $\phi_{h}=H^{m}_{h}$ and observe
 $$
 \frac{(F(H^{m}_{h}),H^{m}_{h})}{\|H^{m}_{h}\|_{\L^2}}\ge\frac{\|\u^{m}_{h}\|^{2}_{\L^2}+\Delta t\|\nabla \u^{m}_{h}\|^{p}_{\L^p}+\frac{1}{\mu\rho}\|\B^{m}_{h}\|^{2}_{\L^2}+k\Delta t\|\nabla \times\B^{m}_{h}\|^{2}_{\L^2}}{\|H^{m}_{h}\|_{\L^2}}\ge\frac{\|H^{m}_{h}\|^{2}_{\L^2}}{\|H^{m}_{h}\|_{\L^2}}\rightarrow \infty\,,
 $$
 as $\|H^{m}_{h}\|_{\L^2}\rightarrow \infty$,
 where, $\|H^{m}_{h}\|^{2}_{\L^2}=\|\u^{m}_{h}\|^{2}_{\L^2}+\frac{1}{\mu\rho}\|\B^{m}_{h}\|^{2}_{\L^2}$ and $k=\frac{1}{\mu^2\rho\sigma}$.\\
By taking $\phi_{h}=(\v_{h},0)$ in \eqref{eq-existence 1}, we infer
$$(\mbox{LHS}_{1}(H^{m}_{h}),\v_{h})=(\mbox{RHS}_{1}(H^{m-1}_{h}),\v_{h}).$$
 If  $\phi_{h}=(0,\C_{h})$ in \eqref{eq-existence 1}, then we have
 $$(\mbox{LHS}_{1}(H^{m}_{h}),\C_{h})=(\mbox{RHS}_{1}(H^{m-1}_{h}),\C_{h}).$$
Finally, by the help of Lemma \ref{lemma-fixed point}, we can show that there exists a solution for the scheme \eqref{eq-FEM-1}--\eqref{eq-FEM-2}.
\end{proof}

\vspace{0.2cm}
\noindent
The semi-implicit structure of
\eqref{eq-FEM-1}--\eqref{eq-FEM-2} also yields unconditional uniqueness.
Indeed, the convective term is evaluated at the previous velocity
$\mathbf{u}_h^{m-1}$, while both electromagnetic coupling terms involve the
previous magnetic field $\mathbf{B}_h^{m-1}$. Consequently, after subtracting
two discrete solutions, the convective term vanishes by skew-symmetry, the
electromagnetic terms cancel, and the remaining nonlinear term is controlled
by the monotonicity of the $p$-Laplace operator.

\begin{lemma}[Uniqueness of the discrete solution]
\label{lem-uniqueness}
Let $m\in\{1,\ldots,N\}$ and $p\ge2$. For given
$(\mathbf{u}_h^{m-1},\mathbf{B}_h^{m-1})$ and data
$\mathbf{g}^m,\mathbf{f}^m$, the scheme
\eqref{eq-FEM-1}--\eqref{eq-FEM-2} has at most one solution in
$\mathcal X_{0h}\times\mathcal Y_h$, for arbitrary $\Delta t>0$ and $h>0$.
\end{lemma}

\begin{proof}
Let $(\mathbf{u}_1,\mathbf{B}_1)$ and
$(\mathbf{u}_2,\mathbf{B}_2)$ be two solutions corresponding to the same
data at the time level $m$, and define
\[
\mathbf{e}_{u}:=\mathbf{u}_1-\mathbf{u}_2,
\qquad
\mathbf{e}_{B}:=\mathbf{B}_1-\mathbf{B}_2.
\]
Then $(\mathbf{e}_{u},\mathbf{e}_{B})
\in\mathcal X_{0h}\times\mathcal Y_h.$ Subtracting the two momentum equations and testing the resulting identity
with $\mathbf{e}_{u}$ gives
\begin{align*}
&\frac{1}{\Delta t}\|\mathbf{e}_{u}\|_{\L^2}^{2}
+a_1\bigl(\mathbf{u}_h^{m-1},\mathbf{e}_{u},\mathbf{e}_{u}\bigr)
+\frac{\nu}{\rho}
 \bigl(a_0(\mathbf{u}_1)-a_0(\mathbf{u}_2),
       \nabla\mathbf{e}_{u}\bigr) + \frac{1}{\mu\rho}
 \bigl(
 \mathbf{B}_h^{m-1}\times(\nabla\times\mathbf{e}_{B}),
 \mathbf{e}_{u}
 \bigr)
=0.
\end{align*}
Similarly, subtracting the two induction equations and testing with
$(\mu\rho)^{-1}\mathbf{e}_{B}$ yields
\begin{align*}
\frac{1}{\mu\rho\,\Delta t}\|\mathbf{e}_{B}\|_{\L^2}^{2}
+\frac{1}{\mu^2\rho\sigma}
 \|\nabla\times\mathbf{e}_{B}\|_{\L^2}^{2}
-\frac{1}{\mu\rho}
 \bigl(
 \mathbf{e}_{u}\times\mathbf{B}_h^{m-1},
 \nabla\times\mathbf{e}_{B}
 \bigr)
=0.
\end{align*}
\noindent
By the skew-symmetry of $a_1$ in its last two arguments,
\[
a_1\bigl(\mathbf{u}_h^{m-1},
         \mathbf{e}_{u},\mathbf{e}_{u}\bigr)=0.
\]
Moreover, the scalar triple-product identity gives
\[
\bigl(
\mathbf{B}_h^{m-1}\times(\nabla\times\mathbf{e}_{B}),
\mathbf{e}_{u}
\bigr)
=
\bigl(
\mathbf{e}_{u}\times\mathbf{B}_h^{m-1},
\nabla\times\mathbf{e}_{B}
\bigr).
\]
Hence the electromagnetic coupling terms cancel when the two tested
equations are added. We therefore obtain
\begin{align}
\label{eq-discrete-uniqueness-energy}
&\frac{1}{\Delta t}\|\mathbf{e}_{u}\|_{\L^2}^{2}
+\frac{1}{\mu\rho\,\Delta t}\|\mathbf{e}_{B}\|_{\L^2}^{2}
+\frac{\nu}{\rho}
 \bigl(
 a_0(\mathbf{u}_1)-a_0(\mathbf{u}_2),
 \nabla\mathbf{e}_{u}
 \bigr) +\frac{1}{\mu^2\rho\sigma}
 \|\nabla\times\mathbf{e}_{B}\|_{\L^2}^{2}
=0.
\end{align}
\noindent
For $p\ge2$, the mapping $\boldsymbol{\xi}\longmapsto
|\boldsymbol{\xi}|^{p-2}\boldsymbol{\xi}$ is strongly monotone. In particular,
\[
\bigl(
a_0(\mathbf{u}_1)-a_0(\mathbf{u}_2),
\nabla\mathbf{e}_{u}
\bigr)
\ge
2^{2-p}\|\nabla\mathbf{e}_{u}\|_{\L^p}^{p}.
\]
It follows from \eqref{eq-discrete-uniqueness-energy} that
\begin{align*}
&\frac{1}{\Delta t}\|\mathbf{e}_{u}\|_{\L^2}^{2}
+\frac{1}{\mu\rho\,\Delta t}\|\mathbf{e}_{B}\|_{\L^2}^{2}
+\frac{\nu\,2^{2-p}}{\rho}
 \|\nabla\mathbf{e}_{u}\|_{\L^p}^{p} +\frac{1}{\mu^2\rho\sigma}
 \|\nabla\times\mathbf{e}_{B}\|_{\L^2}^{2}
\le0.
\end{align*}
All terms on the left-hand side are non-negative. Therefore,
\[
\|\mathbf{e}_{u}\|_{\L^2}
=
\|\mathbf{e}_{B}\|_{\L^2}
=0,
\]
and hence
\[
\mathbf{u}_1=\mathbf{u}_2,
\qquad
\mathbf{B}_1=\mathbf{B}_2.
\]
This proves uniqueness for arbitrary $\Delta t>0$ and $h>0$.
\end{proof}

Combining Theorem~\ref{thm-existence} with
Lemma~\ref{lem-uniqueness}, the fully discrete scheme
\eqref{eq-FEM-1}--\eqref{eq-FEM-2} is well posed at every time level, without
any restriction coupling $\Delta t$ and $h$.

\vspace{0.2cm}

\subsection{Stability Analysis of the Finite Element Scheme \texorpdfstring{\eqref{eq-FEM-1}--\eqref{eq-FEM-2}}{}}

\vspace{0.2cm}

\noindent
We now show that the solution of the scheme \eqref{eq-FEM-1}--\eqref{eq-FEM-2} enjoys a stability property, which holds regardless of the sizes of $h$ and $\Delta t$, will play an important role in proving the convergence of the fully discrete solution.
\begin{lemma}\label{bounds}
    Let $\left(\u_{h}^{m}, \mathbf{B}_{h}^{m}\right)$ be the solutions of fully discrete scheme \eqref{eq-FEM-1}--\eqref{eq-FEM-2}. Then, there exists a constant $C\equiv C\left(\mathbf{g}^m,\mathbf{f}^m,\|\u^{0}_{h}\|_{\L^2}^2,\|\B^{0}_{h}\|^2_{\L^2}, T, \nu,\sigma,\mu \right)$, which is independent of $h,\Delta t$ 
    such that 
\begin{align}
\max \limits_{0 \leq m \leq N}\left[\left\|\u_{h}^{m}\right\|_{\mathbf{L}^2}^{2}+\left\|\B_{h}^{m}\right\|_{\mathbf{L}^2}^{2}\right] \leq C, \label{eq-bounds-1}\\
\Delta t \sum_{m=1}^{N}\left({\frac{\nu}{\rho}\left\|\nabla \u_{h}^{m} \right\|_{\mathbf{L}^p}^{p}}+\frac{1}{\mu^{2}\rho\sigma}\left\| \nabla \times \B_{h}^{m}\right\|_{\mathbf{L}^2}^{2}\right) \leq C, \label{eq-bounds-2}\\
\sum_{m=1}^{N}\left(\left\|\u_{h}^{m}-\u_{h}^{m-1}\right\|_{\mathbf{L}^2}^{2}+\frac{1}{\mu\rho}\left\|\B_{h}^{m}-\B_{h}^{m-1}\right\|_{\mathbf{L}^2}^{2}\right) \leq C.\label{eq-bounds-3}
\end{align}
\end{lemma}
\begin{proof}
    Placing  $\v_{h}=2\Delta t \u_{h}^{m}\in \mathcal{X}_{0h}$, 
    and $\mathbf{w}_{h}=2\Delta t(\mu\rho)^{-1}\B_{h}^{m} \in \mathcal{Y}_{0h}$ in (\ref{eq-FEM-1}) and (\ref{eq-FEM-2}) respectively, we get
\begin{align}
  \label{eq-FEM-stability-u}  &\left(d_{t} \u_{h}^{m}, 2\Delta t \u_{h}^{m}\right)+a_{1}\left(\u_{h}^{m-1}, \u_{h}^{m}, 2\Delta t \u_{h}^{m}\right)+{\frac{\nu}{\rho}\left(a_{0}(\u^m_h),\nabla\left(2\Delta t \u_{h}^{m}\right)\right)} \nonumber
\\ &\qquad 
+\frac{1}{\mu\rho}\left(\mathbf{B}_{h}^{m-1} \times (\nabla \times \mathbf{B}_{h}^{m}), 2\Delta t \u_{h}^{m}\right) = \left(\mathbf{g}^{m}, 2\Delta t\u^{m}_{h}\right), \\
\label{eq-FEM-stability-B}&\left(d_{t} \B_{h}^{m}, 2\Delta t(\mu\rho)^{-1}\B_{h}^{m}\right)+\frac{1}{\mu\sigma}\left(\nabla \times \B_{h}^{m},\nabla \times( 2\Delta t(\mu\rho)^{-1}\B_{h}^{m})\right) \nonumber
 \\ &\qquad -\left(\u_{h}^{m} \times \B_{h}^{m-1}, \nabla \times(2\Delta t(\mu\rho)^{-1}\B_{h}^{m})\right)=\left(\mathbf{f}^{m}, 2\Delta t(\mu\rho)^{-1}\B_{h}^{m}\right).
 \end{align}
Dealing with the first term on the left-hand side of \eqref{eq-FEM-stability-u} we get
 {\small{
\begin{align}
    2\Delta t \left(d_{t} \u_{h}^{m},  \u_{h}^{m}\right)=2\Delta t \left(\frac{ \u_{h}^{m}-\u_{h}^{m-1}}{\Delta t},  \u_{h}^{m}\right) =2\left(( \u_{h}^{m}-\u_{h}^{m-1}),  \u_{h}^{m}\right)=\left\|\u_{h}^{m}\right\|_{\mathbf{L}^2}^{2}+\left\|\u_{h}^{m}-\u_{h}^{m-1}\right\|_{\mathbf{L}^2}^{2}-\left\|\u_{h}^{m-1}\right\|_{\mathbf{L}^2}^{2}\,.
\end{align}
}}
The second term on the left-hand side of \eqref{eq-FEM-stability-u} becomes
\begin{align}
 &a_{1}\left(\u_{h}^{m-1}, \u_{h}^{m}, \u_{h}^{m}\right)=\frac{1}{2}\left\{\int_{\mathbb{D}}((\u^{m-1}_{h}\cdot\nabla)\u^{m}_{h})\cdot\u^{m}_{h}-\int_{\mathbb{D}}((\u^{m-1}_{h}\cdot\nabla)\u^{m}_{h})\cdot\u^{m}_{h}\right\}=0. 
 \end{align}
Using the above two estimates in \eqref{eq-FEM-stability-u} we get
 \begin{align}
     \label{eq-FEM-stability-u-2}
     &\left\|\u_{h}^{m}\right\|_{\mathbf{L}^2}^{2}+\left\|\u_{h}^{m}-\u_{h}^{m-1}\right\|_{\mathbf{L}^2}^{2}-\left\|\u_{h}^{m-1}\right\|_{\mathbf{L}^2}^{2}+\frac{2\Delta t \nu}{\rho}\|\nabla \u^{m}_{h}\|^{p}_{\L^p}\nonumber\\
     &\qquad+\frac{2\Delta t}{\mu\rho}\left(\mathbf{B}_{h}^{m-1} \times (\nabla \times \mathbf{B}_{h}^{m}), \u_{h}^{m}\right) =2\Delta t \left(\mathbf{g}^{m}, \mathbf{u}^{m}_{h}\right).
 \end{align}
The magnetic equation \eqref{eq-FEM-stability-B} gives 
 \begin{align}\label{eq-B-in-stability-FEm}
  &2\Delta t(\mu\rho)^{-1}\left(d_{t} \B_{h}^{m},\B_{h}^{m}\right)+ 2\Delta t(\mu\rho)^{-1}\frac{1}{\mu\sigma}\left(\nabla \times \B_{h}^{m}, \nabla \times\B_{h}^{m}\right) \nonumber
  \\ &\qquad- 2\Delta t(\mu\rho)^{-1}\left(\u_{h}^{m} \times \B_{h}^{m-1}, \nabla \times\B_{h}^{m}\right)= 2\Delta t(\mu\rho)^{-1}\left(\mathbf{f}^{m},\B_{h}^{m}\right)\,.
\end{align}
By the definition of $d_{t} \B_{h}^{m}$, the first term on the left-hand side of \eqref{eq-B-in-stability-FEm} is written as
\begin{align}
    \frac{2\Delta t}{\mu\rho}\left(d_{t} \B_{h}^{m},\B_{h}^{m}\right)=
\frac{1}{\mu\rho}\left\{\left\|\B_{h}^{m}\right\|_{\mathbf{L}^2}^{2}+\left\|\B_{h}^{m}-\B_{h}^{m-1}\right\|_{\mathbf{L}^2}^{2}-\left\|\B_{h}^{m-1}\right\|_{\mathbf{L}^2}^{2}\right\}.
\end{align}
Placing this term in \eqref{eq-B-in-stability-FEm}, we get
\begin{align}
    \label{eq-FEM-stability-B-2}
 \nonumber  & \frac{1}{\mu\rho}\left\{\left\|\B_{h}^{m}\right\|_{\mathbf{L}^2}^{2}+\left\|\B_{h}^{m}-\B_{h}^{m-1}\right\|_{\mathbf{L}^2}^{2}-\left\|\B_{h}^{m-1}\right\|_{\mathbf{L}^2}^{2}\right\}+\frac{2\Delta t}{\mu^2\sigma\rho}\|\nabla\times\B^{m}_{h}\|_{\L^2}^2\\&\qquad- \frac{2\Delta t}{\mu\rho}\left(\u_{h}^{m} \times \B_{h}^{m-1}, \nabla \times\B_{h}^{m}\right)=\frac{2\Delta t}{\mu\rho} \left(\mathbf{f}^{m},\B_{h}^{m}\right)\,.
\end{align}
Adding \eqref{eq-FEM-stability-u-2} and \eqref{eq-FEM-stability-B-2} yields
\begin{align}
    &\left\|\u_{h}^{m}\right\|_{\mathbf{L}^2}^{2}+\left\|\u_{h}^{m}-\u_{h}^{m-1}\right\|_{\mathbf{L}^2}^{2}-\left\|\u_{h}^{m-1}\right\|_{\mathbf{L}^2}^{2}+\frac{1}{\mu\rho}\left(\left\|\B_{h}^{m}\right\|_{\mathbf{L}^2}^{2}+\left\|\B_{h}^{m}-\B_{h}^{m-1}\right\|_{\mathbf{L}^2}^{2}-\left\|\B_{h}^{m-1}\right\|_{\mathbf{L}^2}^{2}\right) \nonumber
    \\ & +{2\Delta t\frac{\nu}{\rho} \left(a_{0}(\u^m_h),\nabla\u^m_h\right)}  + 2\Delta t\frac{1}{\mu^2\sigma\rho}\left(\nabla \times \B_{h}^{m},\nabla \times\B_{h}^{m}\right)=2\Delta t \left(\left(\mathbf{g}^{m}, \u_{h}^{m}\right)+\frac{1}{\mu\rho}\left(\mathbf{f}^{m},\B_{h}^{m}\right)\right).
\end{align}
Now, summing it up from $m=1$ to $N$ we get 
\begin{align}
&\left\| \u_{h}^{N}\right\|_{\mathbf{L}^2}^{2}+\frac{1}{\mu\rho}\left\|\B^N_h\right\|_{\mathbf{L}^2}^{2}+\sum_{m=1}^{N}\left(\left\|\u_{h}^{m}-\u_{h}^{m-1}\right\|_{\mathbf{L}^2}^{2}+\frac{1}{\mu\rho}\left\|\B_{h}^{m}-\B_{h}^{m-1}\right\|_{\mathbf{L}^2}^{2}\right) \nonumber
\\ &\quad+2 \Delta t \sum_{m=1}^{N}\left({\frac{\nu}{\rho} \left\| \nabla\u_{h}^{m} \right\|_{\mathbf{L}^p}^{p}}+\frac{1}{\mu^{2}\rho\sigma}\left\|\nabla \times \B_{h}^{m}\right\|_{\mathbf{L}^2}^{2}\right) \nonumber \\
&=\left\|\u_{h}^{0}\right\|_{\mathbf{L}^2}^{2}+\frac{1}{\mu\rho}\left\|\B_{h}^{0}\right\|_{\mathbf{L}^2}^{2}+2 \Delta t \sum_{m=1}^{N}\left(\left(\mathbf{g}^{m}, \u_{h}^{m}\right)+\frac{1}{\mu\rho}\left(\mathbf{f}^{m}, \B_{h}^{m}\right)\right).
\end{align}
Then, we use Young's inequality to get 
\begin{align}{\label{eq-for bounds}}
&\left\| \u_{h}^{N}\right\|_{\mathbf{L}^2}^{2}+\frac{1}{\mu\rho}\left\|\B^N_h\right\|_{\mathbf{L}^2}^{2}+\sum_{m=1}^{N}\left(\left\|\u_{h}^{m}-\u_{h}^{m-1}\right\|_{\mathbf{L}^2}^{2}+\frac{1}{\mu\rho}\left\|\B_{h}^{m}-\B_{h}^{m-1}\right\|_{\mathbf{L}^2}^{2}\right) \nonumber
\\ &\quad+2 \Delta t \sum_{m=1}^{N}\left({\frac{\nu}{\rho} \left\| \nabla\u_{h}^{m} \right\|_{\mathbf{L}^{p}}^{p}}+\frac{1}{\mu^{2}\rho\sigma}\left\|\nabla \times \B_{h}^{m}\right\|_{\mathbf{L}^2}^{2}\right) \nonumber \\
&\le\left\|\u_{h}^{0}\right\|_{\mathbf{L}^2}^{2}+\frac{1}{\mu\rho}\left\|\B_{h}^{0}\right\|_{\mathbf{L}^2}^{2}+\Delta t \sum_{m=1}^{N}\left(\|\mathbf{g}^{m}\|^2_{\L^2}+\frac{1}{\mu\rho}\|\mathbf{f}^{m}\|_{\L^2}^2 \right)+\Delta t \sum_{m=1}^{N}\|\u^{m}_{h}\|_{\L^2}^2+\Delta t \sum_{m=1}^{N}\frac{1}{\mu\rho}\|\B^{m}_{h}\|_{\L^2}^2.
\end{align}
   By the help of the \eqref{eq-for bounds}  and using the discrete Grownwall Lemma we have the following inequality 
   \begin{align}
    \nonumber &\left\| \u_{h}^{N}\right\|_{\mathbf{L}^2}^{2}+\frac{1}{\mu\rho}\left\|\B^N_h\right\|_{\mathbf{L}^2}^{2}
    +2 \Delta t \sum_{m=1}^{N}\left({\frac{\nu}{\rho} \left\| \nabla\u_{h}^{m} \right\|_{\mathbf{L}^p}^{p}}+\frac{1}{\mu^{2}\rho\sigma}\left\|\nabla \times \B_{h}^{m}\right\|_{\mathbf{L}^2}^{2}\right)\\ &\le C \left\{\left\|\u_{h}^{0}\right\|_{\mathbf{L}^2}^{2}+\frac{1}{\mu\rho}\left\|\B_{h}^{0}\right\|_{\mathbf{L}^2}^{2}+\Delta t \sum_{m=1}^{N}\left(\|\mathbf{g}^{m}\|^2_{\L^2}+\frac{1}{\mu\rho}\|\mathbf{f}^{m}\|_{\L^2}^2 \right)\right\}.
   \end{align}
%
This yields the required stability bounds.
\end{proof}

\vspace{0.3cm}

\section{Convergence Analysis of the Fully Discrete Scheme
\texorpdfstring{\eqref{eq-FEM-1}--\eqref{eq-FEM-2}}{}}
\label{sec-convergence}

In this section, we prove that, as $h,\Delta t\to0$, the fully discrete
solutions of \eqref{eq-FEM-1}--\eqref{eq-FEM-2} converge, up to a subsequence,
to a weak solution of
\eqref{1st weak form}--\eqref{2nd weeak equation} in the sense of
Definition~\ref{def-weak-solution}.

The stability estimates of Lemma~\ref{bounds} provide the required weak
compactness. To pass to the nonlinear convective term, however, strong
compactness of the velocity is needed. This is obtained from the time-translate
estimate of Lemma~\ref{time translate}, together with the Aubin--Lions
compactness theorem. For the magnetic field, Lemma~\ref{time derivative}
provides a bound on the discrete time derivative in a negative-order space.
Finally, the nonlinear $p$-Laplace flux is identified by Minty's monotonicity
argument.

\subsection{Interpolants and the main convergence result}
\label{subsec-conv-interp}
For $t\in(t_{m-1},t_m]$, $m=1,\ldots,N$, define the piecewise linear
interpolants
\begin{align*}
\mathbf{u}_{h,\Delta t}(t) =
\frac{t-t_{m-1}}{\Delta t}\mathbf{u}_h^m
+\frac{t_m-t}{\Delta t}\mathbf{u}_h^{m-1}, \qquad \mathbf{B}_{h,\Delta t}(t) =
\frac{t-t_{m-1}}{\Delta t}\mathbf{B}_h^m
+\frac{t_m-t}{\Delta t}\mathbf{B}_h^{m-1}.
\end{align*}
On the same interval, the piecewise constant extension will be as follows
\[
\mathring{\mathbf{u}}_{h,\Delta t}
=\mathbf{u}_h^m,
\qquad
\mathring{\mathbf{B}}_{h,\Delta t}
=\mathbf{B}_h^m,
\qquad
\mathring{\mathbf{f}}_{\Delta t}
=\mathbf{f}^m,
\qquad
\mathring{\mathbf{g}}_{\Delta t}
=\mathbf{g}^m, \quad \mbox{and} \quad \breve{\mathbf{u}}_{h,\Delta t}
=\mathbf{u}_h^{m-1},
\qquad
\breve{\mathbf{B}}_{h,\Delta t}
=\mathbf{B}_h^{m-1}.
\]
Thus,
\[
\mathring{\mathbf{u}}_{h,\Delta t}
-\breve{\mathbf{u}}_{h,\Delta t}
=
\mathbf{u}_h^m-\mathbf{u}_h^{m-1},
\]
and analogously for the magnetic field.

\begin{theorem}[Convergence of the fully discrete scheme]
\label{Convergence of the Fully Discrete Scheme}
Assume that
\[
\mathring{\mathbf{f}}_{\Delta t}\to\mathbf{f},
\qquad
\mathring{\mathbf{g}}_{\Delta t}\to\mathbf{g}
\quad\text{in}\quad
L^2(0,T;\mathbf{L}^2(\mathbb{D})).
\]
Then, as $h,\Delta t\to0$, the sequence $\bigl\{
(\mathring{\mathbf{u}}_{h,\Delta t},
 \mathring{\mathbf{B}}_{h,\Delta t})
\bigr\}$ admits a subsequence converging to a pair $(\mathbf{u},\mathbf{B})$.
Moreover, $(\mathbf{u},\mathbf{B})$ is a weak solution of the continuous
$p$-MHD system in the sense of
\eqref{1st weak form}--\eqref{2nd weeak equation}.
\end{theorem}

\subsection{Compactness in time}
\label{subsec-conv-compactness}

The stability estimates of Lemma~\ref{bounds} provide uniform spatial bounds
but do not directly control temporal oscillations. For the velocity, the
nonlinear $p$-Laplace term makes a direct estimate of the discrete time
derivative inconvenient; we therefore establish a time-translate estimate.
For the magnetic field, a direct bound on the discrete time derivative is
available in a suitable negative-order space.

\begin{lemma}[Time-translate estimate for the velocity]
\label{time translate}
Let $t_l=l\Delta t$, where $1\le l\le N-1$. Then there exists a constant
$C=C(T)>0$, independent of $h$, $\Delta t$, and $l$, such that
\[
\Delta t\sum_{n=1}^{N-l}
\|\mathbf{u}_h^{n+l}-\mathbf{u}_h^n\|_{\mathbf{L}^{2}}^{2}
\le C(T)t_l^\alpha,
\]
where
\[
\alpha
=
\min\left\{
\frac14,\,
s,\,
\frac{2\delta_1}{3(1+\delta_1)},\,
\frac1p,\,
\left(\frac12+s\right)
\frac{2\delta_1}{3(1+\delta_1)}
\right\},
\]
and $\delta_1=\delta_1(\mathbb{D})>0$ is given by
Lemma~\ref{Embedding lemma}.
\end{lemma}

\begin{proof}
Multiplying \eqref{eq-FEM-1} by $\Delta t$, we obtain
\begin{align*}
&\left(\mathbf{u}_h^m-\mathbf{u}_h^{m-1},\mathbf{v}_h\right)
+\Delta t\,a_1
 \left(\mathbf{u}_h^{m-1},\mathbf{u}_h^m,\mathbf{v}_h\right)
+\frac{\nu}{\rho}\Delta t
 \left(a_0(\mathbf{u}_h^m),\nabla\mathbf{v}_h\right)
\\
&\qquad
+\frac{\Delta t}{\mu\rho}
 \left(
 \mathbf{B}_h^{m-1}\times(\nabla\times\mathbf{B}_h^m),
 \mathbf{v}_h
 \right)
=
\Delta t\,(\mathbf{g}^m,\mathbf{v}_h).
\end{align*}
Summing this identity from $m=n+1$ to $m=n+l$ gives
\begin{align}
\label{eq-u^n+l}
&\left(\mathbf{u}_h^{n+l}-\mathbf{u}_h^n,\mathbf{v}_h\right)
+\Delta t\sum_{i=1}^l
a_1\left(
\mathbf{u}_h^{n-1+i},
\mathbf{u}_h^{n+i},
\mathbf{v}_h
\right) +\frac{\nu}{\rho}\Delta t\sum_{i=1}^l
\left(
|\nabla\mathbf{u}_h^{n+i}|^{p-2}
\nabla\mathbf{u}_h^{n+i},
\nabla\mathbf{v}_h
\right)
\nonumber\\
&\quad
+\frac{\Delta t}{\mu\rho}\sum_{i=1}^l
\left(
\mathbf{B}_h^{n-1+i}
\times(\nabla\times\mathbf{B}_h^{n+i}),
\mathbf{v}_h
\right)
=
\Delta t\sum_{i=1}^l
(\mathbf{g}^{n+i},\mathbf{v}_h).
\end{align}
\noindent
For convenience, set
\[
\delta_l\mathbf{u}_h^n
:=
\mathbf{u}_h^{n+l}-\mathbf{u}_h^n.
\]
Taking $\mathbf{v}_h=\Delta t\,\delta_l\mathbf{u}_h^n$ in
\eqref{eq-u^n+l} and summing over $n=1,\ldots,N-l$, we obtain
\begin{equation}
\label{eq-u^n+2}
\Delta t\sum_{n=1}^{N-l}
\|\delta_l\mathbf{u}_h^n\|_{\mathbf{L}^2}^{2}
+I_{\mathrm{c}}+I_p+I_{\mathrm{M}}
=I_{\mathrm{g}},
\end{equation}
where $I_{\mathrm{c}}$, $I_p$, $I_{\mathrm{M}}$, and $I_{\mathrm{g}}$
denote, respectively, the convective, $p$-Laplace, magnetic, and forcing
terms. We estimate these contributions separately.

\smallskip
\noindent
{Estimate of the convective term.}
By the continuity of $a_1$, followed by H\"older's and Young's inequalities,
\begin{align}
\label{eq-Lem8}
|I_{\mathrm{c}}|
&\le
C(\Delta t)^2
\sum_{n=1}^{N-l}\sum_{i=1}^l
\|\mathbf{u}_h^{n-1+i}\|_{\mathbf{L}^4}
\|\nabla\mathbf{u}_h^{n+i}\|_{\mathbf{L}^2}
\|\delta_l\mathbf{u}_h^n\|_{\mathbf{L}^4}
\nonumber\\
&\le
C(\Delta t)^2
\sum_{n=1}^{N-l}\sum_{i=1}^l
\|\nabla\mathbf{u}_h^{n+i}\|_{\mathbf{L}^2}^{2} +C(\Delta t)^2
\sum_{n=1}^{N-l}\sum_{i=1}^l
\|\mathbf{u}_h^{n-1+i}\|_{\mathbf{L}^4}^{2}
\|\delta_l\mathbf{u}_h^n\|_{\mathbf{L}^4}^{2}
\nonumber\\
&\le
Ct_l+
C(\Delta t)^2
\sum_{n=1}^{N-l}\sum_{i=1}^l
\|\mathbf{u}_h^{n-1+i}\|_{\mathbf{L}^4}^{2}
\|\delta_l\mathbf{u}_h^n\|_{\mathbf{L}^4}^{2},
\end{align}
where \eqref{eq-bounds-2} and $t_l=l\Delta t$ were used in the last step.\\
\noindent
The three-dimensional Ladyzhenskaya inequality
\[
\|\mathbf{v}\|_{\mathbf{L}^4}
\le
C\|\mathbf{v}\|_{\mathbf{L}^2}^{1/4}
 \|\nabla\mathbf{v}\|_{\mathbf{L}^2}^{3/4}
\]
and \eqref{eq-bounds-1} imply
\begin{align*}
&(\Delta t)^2
\sum_{n=1}^{N-l}\sum_{i=1}^l
\|\mathbf{u}_h^{n-1+i}\|_{\mathbf{L}^4}^{2}
\|\delta_l\mathbf{u}_h^n\|_{\mathbf{L}^4}^{2} \le
C(\Delta t)^2
\sum_{n=1}^{N-l}\sum_{i=1}^l
\|\nabla\mathbf{u}_h^{n-1+i}\|_{\mathbf{L}^2}^{3/2}
\|\nabla\delta_l\mathbf{u}_h^n\|_{\mathbf{L}^2}^{3/2}.
\end{align*}
Applying H\"older's inequality first in $i$ and then in $n$, and using
\eqref{eq-bounds-2}, we obtain
\begin{align*}
&(\Delta t)^2
\sum_{n=1}^{N-l}\sum_{i=1}^l
\|\nabla\mathbf{u}_h^{n-1+i}\|_{\mathbf{L}^2}^{3/2}
\|\nabla\delta_l\mathbf{u}_h^n\|_{\mathbf{L}^2}^{3/2} \le
Ct_l^{1/4}\Delta t
\sum_{n=1}^{N-l}
\|\nabla\delta_l\mathbf{u}_h^n\|_{\mathbf{L}^2}^{3/2}
\\
&\qquad\le
Ct_l^{1/4}
\left(
\Delta t\sum_{n=1}^{N-l}
\|\nabla\delta_l\mathbf{u}_h^n\|_{\mathbf{L}^2}^{2}
\right)^{3/4}
(T-t_l)^{1/4} \le C(T)t_l^{1/4}.
\end{align*}
Consequently,
\begin{equation}
\label{eq-convection-translate}
|I_{\mathrm{c}}|
\le C(T)\left(t_l+t_l^{1/4}\right).
\end{equation}

\smallskip
\noindent
{Estimate of the magnetic term.}
Adding and subtracting the Hodge mapping 
$\mathcal{L}(\mathbf{B}_h^{n-1+i})$, we have
\[
|I_{\mathrm{M}}|
\le I_{\mathrm{M},1}+I_{\mathrm{M},2},
\]
where
\begin{align}
I_{\mathrm{M},1}
&:=
C(\Delta t)^2
\sum_{n=1}^{N-l}\sum_{i=1}^l
\|\mathbf{B}_h^{n-1+i}
-\mathcal{L}(\mathbf{B}_h^{n-1+i})\|_{\mathbf{L}^2}\times
\|\nabla\times\mathbf{B}_h^{n+i}\|_{\mathbf{L}^2}
\|\delta_l\mathbf{u}_h^n\|_{\mathbf{L}^{\infty}},
\label{eq-magnetic-I1}
\\
I_{\mathrm{M},2}
&:=
C(\Delta t)^2
\sum_{n=1}^{N-l}\sum_{i=1}^l
\|\mathcal{L}(\mathbf{B}_h^{n-1+i})\|_{\mathbf{L}^3}
\|\nabla\times\mathbf{B}_h^{n+i}\|_{\mathbf{L}^2}
\|\delta_l\mathbf{u}_h^n\|_{\mathbf{L}^6}.
\label{eq-magnetic-I2}
\end{align}
\noindent
For $I_{\mathrm{M},1}$, the lifting estimate \eqref{ine-zbh} and the inverse
estimate \eqref{eq-inverse estimate} give
\begin{align*}
I_{\mathrm{M},1}
&\le
C(\Delta t)^2
\sum_{n=1}^{N-l}\sum_{i=1}^l
h^{s+\frac12}
\|\nabla\times\mathbf{B}_h^{n-1+i}\|_{\mathbf{L}^2}
\|\nabla\times\mathbf{B}_h^{n+i}\|_{\mathbf{L}^2}
\|\delta_l\mathbf{u}_h^n\|_{\mathbf{L}^{\infty}}
\\
&\le
C(\Delta t)^2
\sum_{n=1}^{N-l}\sum_{i=1}^l
h^s
\|\nabla\times\mathbf{B}_h^{n-1+i}\|_{\mathbf{L}^2}
\|\nabla\times\mathbf{B}_h^{n+i}\|_{\mathbf{L}^2}
\|\nabla\delta_l\mathbf{u}_h^n\|_{\mathbf{L}^2}.
\end{align*}
Using the Sobolev interpolation estimate, the uniform
$\mathbf{L}^2$-bound for $\mathbf{B}_h^m$, and Young's inequality, this gives
\begin{align*}
I_{\mathrm{M},1}
&\le
C(\Delta t)^2
\sum_{n=1}^{N-l}\sum_{i=1}^l
\|\nabla\times\mathbf{B}_h^{n-1+i}\|_{\mathbf{L}^2}^{2(1-s)}
\|\nabla\times\mathbf{B}_h^{n+i}\|_{\mathbf{L}^2}^{2} +C(\Delta t)^2
\sum_{n=1}^{N-l}\sum_{i=1}^l
\|\nabla\delta_l\mathbf{u}_h^n\|_{\mathbf{L}^2}^{2}.
\end{align*}
The second term is bounded by $Ct_l$. For the first one, H\"older's
inequality in the index $i$, followed by \eqref{eq-bounds-2}, yields
\[
C(\Delta t)^2
\sum_{n=1}^{N-l}\sum_{i=1}^l
\|\nabla\times\mathbf{B}_h^{n-1+i}\|_{\mathbf{L}^2}^{2(1-s)}
\|\nabla\times\mathbf{B}_h^{n+i}\|_{\mathbf{L}^2}^{2}
\le C(T)t_l^s.
\]
Hence,
\begin{equation}
\label{eq-magnetic-I1-estimate}
I_{\mathrm{M},1}
\le C(T)t_l^s+Ct_l.
\end{equation}
\noindent
We next estimate $I_{\mathrm{M},2}$. Set $\theta:=\frac{2\delta_1}{3(1+\delta_1)}.$ The interpolation inequality
\[
\|\mathcal{L}(\mathbf{B}_h)\|_{\mathbf{L}^3}
\le
C
\|\mathcal{L}(\mathbf{B}_h)\|_{\mathbf{L}^2}^{\theta}
\|\mathcal{L}(\mathbf{B}_h)\|_{0,3+\delta_1}^{1-\theta}
\]
and the decomposition
\[
\mathcal{L}(\mathbf{B}_h)
=
\bigl(\mathcal{L}(\mathbf{B}_h)-\mathbf{B}_h\bigr)
+\mathbf{B}_h
\]
lead to
\[
I_{\mathrm{M},2}
\le I_{\mathrm{M},2,1}+I_{\mathrm{M},2,2}.
\]
\noindent
Let $\lambda:=\left(s+\frac12\right)\theta.$ Using \eqref{ine-zbh}, \eqref{eq-inverse estimate}, the Sobolev embedding,
and the bounds of Lemma~\ref{bounds}, we obtain
\begin{align*}
I_{\mathrm{M},2,1}
&\le
C(\Delta t)^2
\sum_{n=1}^{N-l}\sum_{i=1}^l
h^\lambda
\|\nabla\times\mathbf{B}_h^{n-1+i}\|_{\mathbf{L}^2}
\|\nabla\times\mathbf{B}_h^{n+i}\|_{\mathbf{L}^2}
\|\nabla\delta_l\mathbf{u}_h^n\|_{\mathbf{L}^2} \le C(T)\left(t_l^\lambda+t_l\right).
\end{align*}
For the remaining term, the uniform $\mathbf{L}^2$-bound of
$\mathbf{B}_h^{n-1+i}$ and the lifting estimate give
\begin{align*}
I_{\mathrm{M},2,2}
&\le
C(\Delta t)^2
\sum_{n=1}^{N-l}\sum_{i=1}^l
\|\nabla\times\mathbf{B}_h^{n-1+i}\|_{\mathbf{L}^2}^{1-\theta}
\|\nabla\times\mathbf{B}_h^{n+i}\|_{\mathbf{L}^2}
\|\nabla\delta_l\mathbf{u}_h^n\|_{\mathbf{L}^2} \le C(T)\left(t_l^\theta+t_l\right),
\end{align*}
where the last step follows from Young's inequality, H\"older's inequality
in $i$, and \eqref{eq-bounds-2}. Therefore,
\begin{equation}
\label{eq-magnetic-I2-estimate}
I_{\mathrm{M},2}
\le
C(T)\left(t_l^\lambda+t_l^\theta+t_l\right).
\end{equation}
Combining \eqref{eq-magnetic-I1-estimate} and
\eqref{eq-magnetic-I2-estimate}, we conclude that
\begin{equation}
\label{eq-magnetic-translate}
|I_{\mathrm{M}}|
\le
C(T)\left(
t_l^s+t_l^\lambda+t_l^\theta+t_l
\right).
\end{equation}

\smallskip
\noindent
{Estimate of the \(p\)-Laplace term.}
The nonlinear diffusion contribution is
\[
I_p
=
\frac{\nu}{\rho}(\Delta t)^2
\sum_{n=1}^{N-l}\sum_{i=1}^l
\left(
|\nabla\mathbf{u}_h^{n+i}|^{p-2}
\nabla\mathbf{u}_h^{n+i},
\nabla\delta_l\mathbf{u}_h^n
\right).
\]
By H\"older's inequality in space and then in the index $i$,
\begin{align*}
|I_p|
&\le
C(\Delta t)^2
\sum_{n=1}^{N-l}
\|\nabla\delta_l\mathbf{u}_h^n\|_{\mathbf{L}^p}
\sum_{i=1}^l
\|\nabla\mathbf{u}_h^{n+i}\|_{\mathbf{L}^p}^{p-1}
\\
&\le
Ct_l^{1/p}\Delta t
\sum_{n=1}^{N-l}
\|\nabla\delta_l\mathbf{u}_h^n\|_{\mathbf{L}^p}
\left(
\Delta t\sum_{i=1}^l
\|\nabla\mathbf{u}_h^{n+i}\|_{\mathbf{L}^p}^{p}
\right)^{\frac{p-1}{p}}.
\end{align*}
The last factor is uniformly bounded by \eqref{eq-bounds-3}. Hence, writing
$p'=p/(p-1)$ and applying H\"older's inequality in the index $n$, we find
\begin{align*}
|I_p|
&\le
Ct_l^{1/p}
\left(
\Delta t\sum_{n=1}^{N-l}
\|\nabla\delta_l\mathbf{u}_h^n\|_{\mathbf{L}^p}^{p}
\right)^{1/p}
(T-t_l)^{1/p'} \le C(T)t_l^{1/p}.
\end{align*}
Here we used
\[
\|\nabla\delta_l\mathbf{u}_h^n\|_{\mathbf{L}^p}^{p}
\le
2^{p-1}
\left(
\|\nabla\mathbf{u}_h^{n+l}\|_{\mathbf{L}^p}^{p}
+
\|\nabla\mathbf{u}_h^n\|_{\mathbf{L}^p}^{p}
\right)
\]
together with \eqref{eq-bounds-3}. Thus,
\begin{equation}
\label{eq-p-laplace-translate}
|I_p|\le C(T)t_l^{1/p}.
\end{equation}

\smallskip
\noindent
{Estimate of the forcing term.}
By the Cauchy--Schwarz and Young inequalities, followed by the assumed
uniform bound on the discrete forcing,
\begin{align*}
|I_{\mathrm{g}}|
&\le
(\Delta t)^2
\sum_{n=1}^{N-l}\sum_{i=1}^l
\|\mathbf{g}^{n+i}\|_{\mathbf{L}^2}
\|\delta_l\mathbf{u}_h^n\|_{\mathbf{L}^2}
\\
&\le
\frac12\Delta t
\sum_{n=1}^{N-l}
\|\delta_l\mathbf{u}_h^n\|_{\mathbf{L}^2}^{2}
+C(T)t_l.
\end{align*}
\noindent
Finally, inserting
\eqref{eq-convection-translate},
\eqref{eq-magnetic-translate}, and
\eqref{eq-p-laplace-translate} into \eqref{eq-u^n+2}, and absorbing the
first term in the forcing estimate into the left-hand side, we obtain
\begin{align*}
\Delta t\sum_{n=1}^{N-l}
\|\mathbf{u}_h^{n+l}-\mathbf{u}_h^n\|_{\mathbf{L}^2}^{2}
\le
C(T)\left(
t_l^{1/4}
+t_l^s
+t_l^\theta
+t_l^\lambda
+t_l^{1/p}
+t_l
\right).
\end{align*}
Since
\[
\theta=\frac{2\delta_1}{3(1+\delta_1)},
\qquad
\lambda=
\left(\frac12+s\right)
\frac{2\delta_1}{3(1+\delta_1)},
\]
the definition of $\alpha$ gives
\[
\Delta t\sum_{n=1}^{N-l}
\|\mathbf{u}_h^{n+l}-\mathbf{u}_h^n\|_{\mathbf{L}^2}^{2}
\le C(T)t_l^\alpha.
\]
\end{proof}
\noindent
As a consequence of Lemma~\ref{time translate} and
\cite[Remark~3.4]{MR4410739}, the family
$\{\mathbf{u}_{h,\Delta t}\}$ is uniformly bounded in
\[
H^\beta(0,T;\mathbf{L}^2(\mathbb{D}))
\qquad\text{for every}\qquad
0<\beta<\frac{\alpha}{2}.
\]

\begin{lemma}[Time derivative of the magnetic field]
\label{time derivative}
There exists a constant $C=C(T)>0$, independent of $h$ and $\Delta t$, such
that
\[
\|\partial_t\mathbf{B}_{h,\Delta t}\|_{
L^r\left(
0,T;
\left(\mathcal{Y}\cap\mathbf{W}^{2,2}(\mathbb{D})\right)'
\right)}
\le C,
\]
where
\[
r=\frac{2}{2-\zeta} \quad \mbox{and} \quad \zeta
=
\min\left\{
s,\,
\left(\frac12+s\right)\frac{2\delta_1}{3(1+\delta_1)},\,
\frac{2\delta_1}{3(1+\delta_1)},\,
\frac14,\,
\frac1p
\right\}.
\]
\end{lemma}

               \begin{proof}
            
               We define the projection to $\mathcal{Y}_{h}$ via $\mathcal{R}_{h}$ such that $\left(\mathbf{w}-\mathcal{R}_{h} \mathbf{w}, \partial_{t} \B_{h, \Delta t}\right)=0$, for any $\mathbf{w}_{h} \in \mathcal{Y}_{h}$ by \eqref{eq-FEM-2}, we know that $\left(\u_{h,\Delta t}, \B_{h,\Delta t}\right)$ satisfy
\begin{align}
    &\int_{0}^{T}\left(\partial_{t} \B_{h,\Delta t}, \mathcal{R}_{h} \mathbf{w}\right) dt+\frac{1}{\mu \sigma }\int_{0}^{T}\left( \nabla \times \mathring{\B}_{h, \Delta t}, \nabla \times \mathcal{R}_{h} \mathbf{w}\right) dt \nonumber
    \\ &\qquad +\int_{0}^{T}\left((\nabla \times \mathcal{R}_{h} \mathbf{w} )\times \breve{\B}_{h, \Delta t}, \mathring{\u}_{h, \Delta t}\right) dt=\int_{0}^{T}\left(\mathring{\mathbf{f}}_{\Delta t}, \mathcal{R}_{h} \mathbf{w}\right) dt\,.
\end{align}
By applying Cauchy-Schwarz's inequality and the interpolation inequality for any $\C\in L^{r'}(0,T;(\mathcal{Y}\cap\mathbf{W}^{2, 2}(\mathbb{D})))$ we see
with the help of \cite[Lemma 12]{DING2025116470} that
\begin{align*}
   &\int_{0}^{T}\left | (\mathring{\u}_{h,\Delta t} \times \breve{\B}_{h, \Delta t}, \nabla \times\mathcal{R}_{h}\C)\right| dt 
\le C\|\nabla \times\mathcal{R}_{h}\C\|_{L^{r'}(0,T;\mathbf{L}^{2}(\mathbb{D}))}.
\end{align*} 
Since $1<r<2$, we derive 
$$
\int_{0}^{T}|\nabla \times\mathring{\B}_{h,\Delta t},\nabla \times\mathcal{R}_{h}\C|dt \le C \|\nabla \times\mathcal{R}_{h}\C\|_{L^{r'}(0,T;\mathbf{L}^{2}(\mathbb{D}))}.
$$
Using the properties of $\mathcal{R}_{h}$ we get,
$$
\|\nabla \times\mathcal{R}_{h}\C\|_{\mathbf{L}^{2}(\mathbb{D})}\le c \|\C\|_{W^{2,2}(\mathbb{D})}+c\|\nabla \times\C\|_{\L^{2}(\mathbb{D})}.
$$
$$\int_{0}^{T}\left|\partial_{t}\B_{h,\Delta t}, \C\right|dt\le C\|\C\|_{L^{r'}(0,T,\mathcal{Y}\cap\mathbf{W}^{2, 2}(\mathbb{D}))}.$$
This implies that
$$
\|\partial_{t}\B_{h,\Delta t}\|_{L^{r}\left(0,T;\left(\mathcal{Y}\cap \mathbf{W}^{2, 2}(\mathbb{D})\right)'\right)}\le C.
$$
\end{proof}

\subsection{Proof of Theorem \ref{Convergence of the Fully Discrete Scheme}}
     \begin{proof}
      We divide the proof into three steps: Step 1 extracts a strongly convergent subsequence, Step 2 passes to the limit in the velocity equation, and Step 3 in the magnetic equation.\\
  \noindent      
{\bf Step 1:}\, For $h, \Delta t\rightarrow 0$, there exists a convergent subsequence of $\left(\mathring{\u}_{h,\Delta t},\mathring{\B}_{h,\Delta t}\right)$ and the functions
        $$\u \in L^{\infty}\left(0, T ; \mathbf{L}^{2}\right) \cap L^{2}\left(0, T ; \mathbf{W}^{1, p}_{0}\right)\,, \quad \mbox{and} \quad \B \in L^{\infty}(0, T ; \mathbb{H}) \cap L^{2}(0, T; \mathcal{K}),$$
such that
\begin{align*}
    &\mathring{\u}_{h, \Delta t}, \breve{\u}_{h, \Delta t}, \u_{h, \Delta t} \xrightarrow{\text {weakly }} \u \quad \mbox{in } \ L^{2}\left(0, T ; \mathcal{X}\right),\\
&\mathring{\u}_{h, \Delta t}, \check{\u}_{h, \Delta t}, \u_{h, \Delta t} \xrightarrow{\text {weakly* }} \u \quad \mbox{in } \ L^{\infty}\left(0, T ;\mathbf {L}^{2}(\mathbb{D})\right),\\
&\mathring{\B}_{h, \Delta t, } \breve{\mathbf{B}}_{h, \Delta t}, \B_{h, \Delta t} \xrightarrow{\text { weakly }} \mathbf{B} \quad \mbox{in } \ L^{2}(0, T ; \mathcal{Y}),\\
&\mathring{\B}_{h, \Delta t, } \breve{\B}_{h, \Delta t}, \B_{h, \Delta t} \xrightarrow{\text { weakly* }} \mathbf{B} \quad \mbox{in } \ L^{\infty}(0, T ; \mathbf{L}^{2}(\mathbb{D})),\\
&\partial_{t} \B_{h,\Delta t} \xrightarrow{\text { weakly }} \B_{t} \quad \mbox{in } \ L^{r}\left(0, T ;\left(\mathcal{Y} \cap \mathbf{W}^{2, 2}(\mathbb{D})\right)^{\prime}\right),\\
&\partial_{t} \u_{h, \Delta t} \xrightarrow{\text { weakly } }\u_{t} \quad \mbox{in } {H}^{\beta}\left(0,T;\mathbf{L}^{2}(\mathbb{D})\right)\,.
\end{align*}
Using the definition of $\u_{h,\Delta t}$ and $\mathring{\u}_{h,\Delta t}$, we have
\begin{align*}
   \left\|\u_{h,\Delta t}-\mathring{\u}_{h,\Delta t}\right\|_{L^{2}(0,T;\mathbf{L}^{p}(\mathbb{D}))}^{2} &=\left\|\frac{t-t_{m-1}}{\Delta t}\u_{h}^{m}
    +\frac{t_{m}-t}{\Delta t}\u_{h}^{m-1}-{\u}_{h}^{m}\right\|_{L^{2}(0,T;\mathbf{L}^{p}(\mathbb{D}))}^{2}\\
    &=\left\|\frac{t-t_{m}}{\Delta t}\u_{h}^{m}
    -\frac{t-t_{m}}{\Delta t}\u_{h}^{m-1}\right\|_{L^{2}(0,T;\mathbf{L}^{p}(\mathbb{D}))}^{2}\\
    &=\sum_{m=1}^{N}\left\|\u_{h}^{m}-\u_{h}^{m-1}\right\|_{\mathbf{L}^{p}}^{2}\int_{t_{m-1}}^{t_{m}}\left(\frac{t-t_{m}}{\Delta t}\right)^2 dt =\frac{\Delta t}{3}\sum_{m=1}^{N}\left\|\u_{h}^{m}-\u_{h}^{m-1}\right\|_{\mathbf{L}^{p}}^{2}.
\end{align*}
Using the Poincar\'e inequality, the H\"older inequality and \eqref{eq-bounds-2} of the Lemma \ref{bounds}, and as $h,\Delta t\rightarrow 0$, we infer
\begin{align}
       \nonumber \frac{\Delta t}{3}\sum_{m=1}^{N}\left\|\u_{h}^{m}-\u_{h}^{m-1}\right\|_{\mathbf{L}^{p}}^{2} &\nonumber \le\frac{\Delta t}{3}\sum_{m=1}^{N}\|\nabla(\u_{h}^{m}-\u_{h}^{m-1})\|_{\mathbf{L}^{p}}^{2}\\
       &\le \frac{\Delta t}{3} \left(\sum_{m=1}^{N}\|\nabla(\u_{h}^{m}-\u_{h}^{m-1})\|_{\mathbf{L}^{p}}^{p}\right)^{\frac{2}{p}}\left(\sum_{m=1}^{N} 1 \right)^{1-\frac{2}{p}}\nonumber\\
       &\le \frac{(\Delta t)^{1-\frac{2}{p}}}{3}\left(\sum_{m=1}^{N} 1 \right)^{1-\frac{2}{p}} \longrightarrow 0\,.
\end{align}
Similarly, we can obtain that 
\begin{equation}
    \|\u_{h,\Delta t}-\breve{\u}_{h,\Delta t}\|_{L^{2}(0,T;\mathbf{L}^{p}(\mathbb{D}))}\longrightarrow 0\qquad  \mbox{as}\qquad h, \Delta t\rightarrow 0.
\end{equation}
So, we get by the help of Lemma \ref{Aubin Lions} and Lemma \ref{time translate} that $(\u_{h,\Delta t}),(\mathring{\u}_{h,\Delta t}),(\breve{\u}_{h,\Delta t})$ converges strongly to the same limit $\u$, as $h,\Delta t\rightarrow 0$, in the space $L^{2}(0,T;\mathbf{L}^{p}(\mathbb{D}))$ for $p\ge 2$.\\
By the help of \eqref{eq-bounds-3} of the Lemma \ref{bounds}, we have the following
\begin{align}
 \nonumber \|\B_{h,\Delta t}-\mathring{\B}_{h,\Delta t}\|^{2}_{L^{2}(0,T;\mathbf{L}^{2}(\mathbb{D}))}
 &=\left\|\frac{t-t_{m-1}}{\Delta t}\B^{m}_{h}+\frac{t_{m}-t}{\Delta t}\B^{m-1}_{h}-\B^{m}_{h} \right\|^{2}_{L^{2}(0,T;\mathbf{L}^{2}(\mathbb{D}))}\\
\nonumber =&\sum_{m=1}^{N}\|\B^{m}_{h}-\B^{m-1}_{h}\|^{2}_{\mathbf{L}^{2}}\int_{t_{m-1}}^{t_{m}}\left(\frac{t-t_{m-1}}{\Delta t}\right)^{2} dt\\
 =&\frac{\Delta t}{3}\sum_{m=1}^{N}\|\B^{m}_{h}-\B^{m-1}_{h}\|^{2}_{\mathbf{L}^{2}} \longrightarrow 0 \quad \mbox{as} \quad h, \Delta t\rightarrow 0.
  \end{align}  

\noindent
Similarly, we have 
\begin{equation}
    \|\B_{h,\Delta t}-\breve{\B}_{h,\Delta t}\|^{2}_{L^{2}(0,T;\mathbf{L}^{2}(\mathbb{D}))}\longrightarrow 0\quad \mbox{as} \quad h, \Delta t\rightarrow 0.
\end{equation}
Then by the help of Lemma \ref{Aubin Lions} and \ref{time derivative} we have that $(\B_{h,\Delta t}),(\breve{\B}_{h,\Delta t}), (\mathring{\B}_{h,\Delta t})$ converges strongly to $\B$ as $h, \Delta t\rightarrow 0$ in the space $L^{2}(0,T;\mathbf{L}^{2}(\mathbb{D}))$.
\\

{\bf{Step 2}:} We consider the equation
\begin{align}\label{eq-convergence}
   \nonumber &\int_{0}^{T}\int_{\mathbb{D}}\partial_{t}\u_{h, \Delta t}\cdot\mathbf{v}_{h} dx dt-\int_{0}^{T}\int_{\mathbb{D}}(\breve{\u}_{h, \Delta t}\otimes\mathring{\u}_{h, \Delta t})\cdot\nabla \mathbf{v}_{h} dx dt+\frac{\nu}{\rho} \int_{0}^{T}\int_{\mathbb{D}}|\nabla \mathring{\u}_{h, \Delta t}|^{p-2}\nabla\mathring{\u}_{h, \Delta t}\cdot \nabla \mathbf{v}_{h} dx dt\\
    &+\frac{1}{\mu\rho}\int_{0}^{T}\int_{\mathbb{D}}(\breve{\B}_{h, \Delta t}\times(\nabla \times\mathring{\B}_{h, \Delta t}))\cdot \mathbf{v}_{h} dx dt=\int_{0}^{T}\int_{\mathbb{D}}\mathring{\mathbf{g}}_{\Delta t}\cdot\mathbf{v}_{h} dx dt.
\end{align}
We pass to the limit term by term. Fix $\mathbf{v}\in\mathcal{J}$ and choose $\v_{h}\in\mathcal{X}_{0h}$ with $\v_{h}\rightarrow\v$ in $\mathbf{W}^{1,p}(\mathbb{D})$ for $2\le p<6$, and fix $\xi\in C^{\infty}([0,T])$. Using the strong convergence of the initial data as $h,\Delta t\rightarrow0$, the first term on the left-hand side of \eqref{eq-convergence} is estimated by
\begin{align}
   \nonumber &\left|\int_{0}^{T}\int_{\mathbb{D}}(\partial_{t}\u_{h, \Delta t}\cdot\mathbf{v}_{h}-\partial_{t}\u\cdot\mathbf{v})\xi(t) dx dt\right|\\
  \nonumber  &=\left|\int_{0}^{T}\left(\u_{h, \Delta t},\partial_{t}(\mathbf{v}_{h}\xi(t)\right)+\left(\u,\partial_{t}(\mathbf{v}\xi(t)\right)-\left(\u^{0}_{h},\v_{h}\xi(0)\right)+\left(\u^{0},\v\xi(0)\right)dt\right|\\
     &\le C \|\u_{h, \Delta t}-\u\|_{L^2(0,T;\mathbf{L}^2(\mathbb{D}))}\|\partial_{t}(\mathbf{v}_{h}\xi(t))\|_{L^2(0,T;\mathbf{L}^2(\mathbb{D}))}\nonumber\\
     &\qquad+\|\u\|_{L^2(0,T;\mathbf{L}^2(\mathbb{D}))}\|\partial_{t}[(\v_{h}-\v)\xi(t)]\|_{L^2(0,T;\mathbf{L}^2(\mathbb{D}))}\longrightarrow 0.
\end{align}
Now, we estimate the second term on the left-hand side of \eqref{eq-convergence} with the help of  \cite[Theorem 3]{DING2025116470}
\begin{align}
   \nonumber& \left|\int_{0}^{T}\int_{\mathbb{D}}[(\breve{\u}_{h, \Delta t}\otimes\mathring{\u}_{h, \Delta t})\cdot\nabla \mathbf{v}_{h}-(\u\otimes \u)\cdot \nabla\v]\xi(t)  dx dt\right|
\longrightarrow 0 \quad \mbox{as} \quad h,\Delta t\rightarrow 0.
\end{align}
Next, we estimate the last term of the left hand side of \eqref{eq-convergence} with the help of  \cite[Theorem 3]{DING2025116470}
\begin{align}\label{eq-F}
    \nonumber&\frac{1}{\mu\rho}\left|\int_{0}^{T}\int_{\mathbb{D}}\left[(\breve{\B}_{h, \Delta t}\times(\nabla \times\mathring{\B}_{h, \Delta t}))\cdot \mathbf{v}_{h}-(\B\times(\nabla \times\B))\cdot\v\right]\xi(t) dx dt\right|\\
   \nonumber& =\frac{1}{\mu\rho}\bigg|
   \int_{0}^{T}\int_{\mathbb{D}}\Big[(\breve{\B}_{h, \Delta t}-\mathcal{L}(\breve{\B}_{h,\Delta t}))\times(\nabla \times\mathring{\B}_{h, \Delta t}),\mathbf{v}_{h}) 
   \\ &\qquad \nonumber +(\mathcal{L}(\breve{\B}_{h,\Delta t})\times(\nabla \times\mathring{\B}_{h,\Delta t}),\mathbf{v}_{h}) -(\B\times(\nabla \times\B),\v)\xi(t)\Big] dx dt\bigg|\\ 
   \nonumber &\le C h^{l+\frac{1}{2}}\|\nabla \times\breve{\B}_{h,\Delta t}\|_{L^{2}(0,T;\mathbf{L}^2(\mathbb{D}))}\|\nabla \times\mathring{\B}_{h,\Delta t}\|_{L^{2}(0,T;\mathbf{L}^3(\mathbb{D}))}\|\v_{h}\xi(t)\|_{L^{\infty}(0,T;\mathbf{L}^{6}(\mathbb{D}))}\\
\qquad&+\left|\int_{0}^{T}\int_{\mathbb{D}}\left[(\mathcal{L}(\breve{\B}_{h,\Delta t})\times(\nabla \times\mathring{\B}_{h,\Delta t}),\mathbf{v}_{h})-(\B\times(\nabla \times\B),\v)\right]\xi(t) dx dt\right| =: J_{h, 1}+|J_{h, 2}|.
\end{align}

\noindent
Recall that $\mathcal{L}:H_{0}(\operatorname{\mathbf{curl}};\mathbb{D})\rightarrow \mathcal{K}$ and since $\operatorname{div}(\mathcal{L}(\B))=0$, we infer
\begin{align*}
    \|\mathcal{L}(\breve{\B}_{h,\Delta t})\|_{L^{2}(0,T,\mathcal{D}(\mathbb{D}))}=\|\nabla \times\mathcal{L}(\breve{\B}_{h,\Delta t})\|_{L^{2}(0,T;\mathbf{L}^{2}(\mathbb{D}))}=\|\nabla \times\breve{\B}_{h,\Delta t}\|_{L^{2}(0,T;\mathbf{L}^{2}(\mathbb{D}))}\le C\,.
\end{align*}
Similarly, it can be shown that $\|\nabla \times\mathring{\B}_{h,\Delta t}\|_{L^{2}(0,T;\mathbf{L}^{2}(\mathbb{D}))}$ is bounded. Thus, $J_{h,1}\rightarrow0$ as $h,\Delta t\rightarrow 0$. 
\noindent
 Now we choose $\delta_{2}> 0,$ such that
\begin{align}
   & \frac{1}{3+\delta_{1}}+\frac{1}{6-\delta_{2}}=\frac{1}{2}\,.
\end{align}
Thus, we get the embedding $H^{1}(\mathbb{D})\hookrightarrow L^{6-\delta_2}(\mathbb{D}).$
Now, we calculate the second part of the \eqref{eq-F}
\begin{align}
\nonumber|J_{h,2}| &=\left|\int_{0}^{T}\int_{\mathbb{D}}\left[(\mathcal{L}(\breve{\B}_{h,\Delta t})\times\nabla \times\mathring{\B}_{h,\Delta t},\mathbf{v}_{h})-(\B\times\nabla \times\B,\v)\right]\xi(t) dx dt\right|\\
   \nonumber &\le C\|\mathcal{L}(\breve{\B}_{h,\Delta t})-\B\|_{L^2(0,T;\mathbf{L}^{3}(\mathbb{D}))}\|\nabla \times\breve{\B}_{h,\Delta t}\|_{L^2(0,T;\mathbf{L}^{2}(\mathbb{D}))}\|\v_{h}\xi(t)\|_{L^\infty(0,T;\mathbf{L}^{3}(\mathbb{D}))}\\ \nonumber&\quad+C\|\B_{h,\Delta t}\|_{L^2(0,T;\mathbf{L}^{2}(\mathbb{D}))}\|\nabla \times\B\|_{L^2(0,T;\mathbf{L}^{2}(\mathbb{D}))}\|(\v_{h}-\v)\xi(t)\|_{L^\infty(0,T;\mathbf{L}^{6-\delta_{2}}(\mathbb{D}))}
   \\  &\quad +C\left|\int_{0}^{T}\int_{\mathbb{D}}\left[(\mathcal{L}(\breve{\B}_{h,\Delta t})\times(\nabla \times\mathring{\B}_{h,\Delta t}-\nabla \times\B),\mathbf{v}_{h})\right]\xi(t) dx dt\right|\longrightarrow 0\qquad \mbox{as}\qquad h,\Delta t\rightarrow 0.
\end{align} 
By the Sobolev inequality, we infer
\begin{align}
    \nonumber\|\mathcal{L}(\breve{\B}_{h,\Delta t})-\B\|_{L^2(0,T;\mathbf{L}^{3}(\mathbb{D}))}&\le\|\mathcal{L}(\breve{\B}_{h,\Delta t})-\B\|_{L^2(0,T;\mathbf{L}^{2}(\mathbb{D}))}^{\frac{2\delta_{1}}{3(1+\delta_{1})}}\|\mathcal{L}(\breve{\B}_{h,\Delta t})-\B\|_{L^2(0,T;\mathbf{L}^{3+\delta_{1}}(\mathbb{D}))}^{\frac{3+\delta_{1}}{3(1+\delta_{1})}}\,.
\end{align}
\begin{align}
    \|\mathcal{L}(\breve{\B}_{h,\Delta t})-\B\|_{L^2(0,T;\mathbf{L}^{2}(\mathbb{D}))}^{\frac{2\delta_{1}}{3(1+\delta_{1})}}\nonumber&\le\left( \|\mathcal{L}(\breve{\B}_{h,\Delta t})-\breve{\B}\|_{L^2(0,T;\mathbf{L}^{2}(\mathbb{D}))}+\|\breve{\B}-\B\|_{L^2(0,T;\mathbf{L}^{2}(\mathbb{D}))}\right)^{\frac{2\delta_{1}}{3(1+\delta_{1})}}\\
    &\le \left(Ch^{\frac{1}{2}+l}\|\nabla \times\breve{\B}_{h,\Delta t}\|_{L^2(0,T;\mathbf{L}^{2}(\mathbb{D}))}+\|\breve{\B}-\B\|_{L^2(0,T;\mathbf{L}^{2}(\mathbb{D}))}\right)^{\frac{2\delta_{1}}{3(1+\delta_{1})}}.
\end{align}
It remains to treat the $p$-Laplace term of \eqref{eq-convergence}. Writing
\[
A_{h,\Delta t}(x,t)=|\nabla \mathring{\u}_{h, \Delta t}|^{p-2} \nabla\mathring{\u}_{h, \Delta t},
\]
the bound \eqref{eq-bounds-2} of Lemma \ref{bounds} shows that $\{A_{h,\Delta t}\}$ is bounded in $L^{p'}(0,T;\L^{p'}(\mathbb{D}))$, so that, along a subsequence, $A_{h,\Delta t}\rightharpoonup\bar{A}$ weakly in that space. Weak convergence of $\nabla\mathring{\u}_{h,\Delta t}$ alone does not identify $\bar{A}$ with $|\nabla\u|^{p-2}\nabla\u$, since the map is nonlinear; the identification is supplied by Minty's monotonicity argument, which applies because $A$ is monotone, hemicontinuous, coercive and bounded on the reflexive space $L^{p}(0,T;\mathcal{X}_{0})$. Passing to the limit in \eqref{eq-convergence} gives
\begin{align}\label{eq-p-laplace convergence}
   \nonumber &\int_{0}^{T}\int_{\mathbb{D}}\partial_{t}\u\cdot\mathbf{v} dx dt-\int_{0}^{T}\int_{\mathbb{D}}({\u}\otimes{\u})\cdot\nabla \mathbf{v} dx dt+\frac{\nu}{\rho} \int_{0}^{T}\int_{\mathbb{D}}\bar{A}(x,t)\nabla \v dx dt\\
    &+\frac{1}{\mu\rho}\int_{0}^{T}\int_{\mathbb{D}}({\B}\times(\nabla \times{\B}))\cdot \mathbf{v} dx dt=\int_{0}^{T}\int_{\mathbb{D}}{\mathbf{g}}\cdot\mathbf{v}dx dt.
\end{align}
Now, we have to show that 
$$
\bar{A}(x,t)=|\nabla \u|^{p-2}\nabla\u.
$$
We take the superior limit in \eqref{eq-convergence} and we already know that $\mathring{\u}_{h,\Delta t}\rightarrow \u$ converges strongly. By taking all other terms to the right hand side, we infer
\begin{align}
   \nonumber &\limsup_{t\rightarrow 0}\frac{\nu}{\rho} \int_{0}^{T}\int_{\mathbb{D}}A_{h,\Delta t}(x,t)\cdot \nabla \mathbf{v}_{h} dx dt
    \\\nonumber &=-\limsup_{t\rightarrow 0}\bigg(\int_{0}^{T}\int_{\mathbb{D}}\partial_{t}\u_{h, \Delta t}\cdot\mathbf{v}_{h} dx dt-\int_{0}^{T}\int_{\mathbb{D}}(\breve{\u}_{h, \Delta t}\otimes\mathring{\u}_{h, \Delta t})\cdot\nabla \mathbf{v}_{h} dx dt
    \\ & \quad+\frac{1}{\mu\rho}\int_{0}^{T}\int_{\mathbb{D}}(\breve{\B}_{h, \Delta t}\times(\nabla \times\mathring{\B}_{h, \Delta t}))\cdot \mathbf{v}_{h} dx dt -\int_{0}^{T}\int_{\mathbb{D}}\mathring{\mathbf{g}}_{\Delta t}\cdot\mathbf{v}_{h} dx dt\bigg).
    \end{align}
    Now, with the help of \eqref{eq-p-laplace convergence} we deduce
    \begin{align}
     &\limsup_{t\rightarrow 0}\frac{\nu}{\rho} \int_{0}^{T}\int_{\mathbb{D}}A_{h,\Delta t}(x,t)\cdot \nabla \mathbf{v}_{h} dx dt=\frac{\nu}{\rho} \int_{0}^{T}\int_{\mathbb{D}}\bar{A}(x,t)\nabla \v dx dt.
    \end{align}
We now invoke Minty's monotonicity argument. For $\mathbf{G}\in L^{p}(0,T;L^{p}(\mathbb{D}))^{d}$,
\begin{align}
    &\int_{0}^{T}\int_{\mathbb{D}}(A_{h,\Delta t}-A(\mathbf{G}))\cdot(\nabla\mathring{\u}_{h,\Delta t}-\mathbf{G}(x,t)) dx dt\ge 0.
\end{align}
This uses the monotonicity of the $p$-Laplace operator together with the strong convergence $\mathring{\u}_{h,\Delta t}\rightarrow\u$ in $\mathbf{L}^{p}$. Now, we pass the limit by the help of weak convergence to deduce
\begin{align}
    \int_{0}^{T}\int_{\mathbb{D}}[\bar{A}(x,t)-A(\mathbf{G}(x,t))][\nabla \u-\mathbf{G}(x,t)] dx dt \ge 0.
\end{align}
We take $\mathbf{G}=\nabla \u+\alpha\phi$, 
then we infer
\begin{align*}
  & \int_{0}^{T}\int_{\mathbb{D}}[\bar{A}(x,t)-A(\nabla \u+\alpha\phi)][\nabla \u-\nabla \u+\alpha\phi] dx dt \ge 0. \\
  &\alpha\int_{0}^{T}\int_{\mathbb{D}}[\bar{A}(x,t)-A(\nabla \u+\alpha\phi)]\phi\ge 0.
\end{align*}
 Dividing by $\alpha$ and by taking $\alpha\rightarrow 0$ we conclude
\begin{align}
  \int_{0}^{T}\int_{\mathbb{D}}[\bar{A}(x,t)-A(\nabla \u)]\phi=0\,,
\end{align}
which is true for all $\phi\in L^{p}(\mathbb{D})$.
This shows that $\bar{A}(x,t)=A(\nabla\u)$, which implies the convergence result for \eqref{eq-convergence}.\\

\noindent
{\bf{Step 3}:} In this step, we show that $\B$ is the weak solution of\eqref{2nd weeak equation}. By \eqref{eq-FEM-2} we know that $(\u_{h,\Delta t},\B_{h,\Delta t})$ satisfies
\begin{align}\label{eq-2-convergence}
  \nonumber  &\int_{0}^{T}\int_{\mathbb{D}}\partial_{t} \B_{h,\Delta t}\cdot\mathbf{w}_{h} dxdt+\frac{1}{\mu \sigma }\int_{0}^{T}\int_{\mathbb{D}}\nabla \times \mathring{\B}_{h, \Delta t}\cdot \nabla \times\mathbf{w}_{h}dx dt\\&\qquad-\int_{0}^{T}\int_{\mathbb{D}}\mathring{\u}_{h, \Delta t} \times \breve{\B}_{h, \Delta t}\cdot(\nabla \times\mathbf{w}_{h})dxdt=\int_{0}^{T}\int_{\mathbb{D}}\mathring{\mathbf{f}}_{\Delta t}\cdot\mathbf{w}_{h} dxdt.
\end{align}
For any $\C\in\mathcal{G}$ we may choose $\C_{h}\in \mathcal{Y}_{0h}$ with $\mathbf{w}_{h}\rightarrow \mathbf{w}$ as $h\rightarrow 0$ in the space
\[
\left\{\mathbf{w}\in \mathcal{Y}:\ \nabla \times\mathbf{w}\in \L^{p}(\mathbb{D})\ \ \forall\, 2\le p<6\right\}.
\]
We will show the convergence term by term.
For the first term of \eqref{eq-2-convergence}, the strong convergence of the initial data gives
\begin{align}
\nonumber&\left|\int_{0}^{T}\int_{\mathbb{D}}\left(\partial_{t} \B_{h,\Delta t}\cdot\mathbf{w}_{h}-\partial_{t}\B\cdot\mathbf{w}\right)\xi(t) dx dt\right|\\
  \nonumber &=\left|\int_{0}^{T}\left[(-\B_{h,\Delta t},\partial_{t}(\mathbf{w}_{h}\xi(t)))+(\B,\partial_{t}(\mathbf{w}\xi(t)))-(\B_{h}^{0},\mathbf{w}_{h}\xi(0))+(\B^{0},\mathbf{w}\xi(0))\right]dt\right|\\
 \nonumber &\le C\|\B_{h,\Delta t}-\B\|_{L^2(0,T;\mathbf{L}^{2}(\mathbb{D}))}\|\partial_{t}(\mathbf{w}_{h}\xi(t))\|_{L^2(0,T;\mathbf{L}^{2}(\mathbb{D}))}\\&\qquad+C\|\B\|_{L^2(0,T;\mathbf{L}^{2}(\mathbb{D}))}\|(\partial_{t}\mathbf{w}_{h}-\partial_t\mathbf{w})\xi(t)\|_{L^2(0,T;\mathbf{L}^{2}(\mathbb{D}))}\longrightarrow 0 \qquad \mbox{as}\qquad h,\Delta t\longrightarrow 0.
    \end{align}
    Now, we estimate the second term of the left-hand side of the \eqref{eq-2-convergence}
    \begin{align}
       \nonumber &\left|\frac{1}{\mu \sigma }\int_{0}^{T}\int_{\mathbb{D}}\left[(\nabla \times \mathring{\B}_{h, \Delta t}\cdot \nabla \times\mathbf{w}_{h})-(\nabla \times\B\cdot\nabla \times\mathbf{w})\right]\xi(t)dx dt\right|\\
       \nonumber& \le C \left|\int_{0}^{T}\!\!\int_{\mathbb{D}}\left[(\nabla \times \mathring{\B}_{h, \Delta t}\cdot \nabla \times\mathbf{w}_{h})-(\nabla \times\B\cdot \nabla \times\mathbf{w}_{h})\right.\right.\\
       \nonumber&\hspace{2.6cm}\left.\left.+(\nabla \times\B\cdot \nabla \times\mathbf{w}_{h})-(\nabla \times\B\cdot\nabla \times\mathbf{w})\right]\xi(t)dx dt\right|\\
       \nonumber&\le C \|\nabla \times \mathring{\B}_{h, \Delta t}-\nabla \times\B\|_{L^2(0,T;\mathbf{L}^{2}(\mathbb{D}))}\|(\nabla \times\mathbf{w}_{h})\xi(t)\|_{L^2(0,T;\mathbf{L}^{2}(\mathbb{D}))}\\
       \nonumber&\qquad+ C\|\nabla \times\B\|_{L^2(0,T;\mathbf{L}^{2}(\mathbb{D}))}\|(\nabla \times\mathbf{w}_{h}-\nabla \times\mathbf{w})\xi(t)\|_{L^2(0,T;\mathbf{L}^{2}(\mathbb{D}))}\\
       &\longrightarrow 0 \qquad \mbox{as}\qquad h, \Delta t\rightarrow 0.
    \end{align}
   In order to estimate the last term of the left-hand side, we use \eqref{eq-inverse estimate} and \eqref{ine-zbh} to deduce
    \begin{align}
&\left|\int_{0}^{T}\int_{\mathbb{D}}\left[(\mathring{\u}_{h, \Delta t} \times \breve{\B}_{h, \Delta t}\cdot\nabla \times\mathbf{w}_{h})-(\u\times\B\cdot\nabla \times\mathbf{w})\right] \xi(t)dxdt\right| \nonumber\\
&\le\left|\int_{0}^{T}\!\!\int_{\mathbb{D}}\left[(\mathring{\u}_{h, \Delta t} \times (\breve{\B}_{h, \Delta t}-\mathcal{L}(\breve{\B}_{h, \Delta t}))\cdot\nabla \times\mathbf{w}_{h})\right.\right.\nonumber\\
&\hspace{2.2cm}\left.\left.+(\mathring{\u}_{h, \Delta t} \times\mathcal{L}(\breve{\B}_{h, \Delta t}))\cdot\nabla \times\mathbf{w}_{h})-(\u\times\B\cdot\nabla \times\mathbf{w})\right]\xi(t) dxdt\right| \nonumber\\
\nonumber& \le Ch^{l+\frac{1}{2}}\|\nabla \times\breve{\B}_{h, \Delta t}\|_{L^2(0,T;\mathbf{L}^{3}(\mathbb{D}))}\|\mathring{\u}_{h, \Delta t}\|_{L^\infty(0,T;\mathbf{L}^{2}(\mathbb{D}))}\|(\nabla \times\mathbf{w}_{h})\xi(t)\|_{L^2(0,T;\mathbf{L}^{6}(\mathbb{D}))} \\
\nonumber&\qquad +\left|\int_{0}^{T}\int_{\mathbb{D}}\left[(\mathring{\u}_{h, \Delta t} \times\mathcal{L}(\breve{\B}_{h, \Delta t})\cdot\nabla \times\mathbf{w}_{h})-(\u\times\B\cdot\nabla \times\mathbf{w})\right]\xi(t) dxdt\right|\\
\nonumber& \le Ch^{l}\|\nabla \times\breve{\B}_{h, \Delta t}\|_{L^2(0,T;\mathbf{L}^{2}(\mathbb{D}))}\|\mathring{\u}_{h, \Delta t}\|_{L^\infty(0,T;\mathbf{L}^{2}(\mathbb{D}))}\|(\nabla \times\mathbf{w}_{h})\xi(t)\|_{L^2(0,T;\mathbf{L}^{6}(\mathbb{D}))}\\
&\qquad +\left|\int_{0}^{T}\int_{\mathbb{D}}\left[(\mathring{\u}_{h, \Delta t} \times\mathcal{L}(\breve{\B}_{h, \Delta t})\cdot\nabla \times\mathbf{w}_{h})-(\u\times\B\cdot\nabla \times\mathbf{w})\right]\xi(t) dxdt\right| =:E_{1}+|E_{2}|\rightarrow 0
\end{align}
as $h,\Delta t\rightarrow 0$  (see \cite[Theorem 3]{DING2025116470}).
Thus, the convergence result for \eqref{eq-2-convergence} is proved.
    \end{proof}

{
\section{Error Analysis for the System \texorpdfstring{\eqref{1st weak form}--\eqref{2nd weeak equation}}{}}
\label{ lemma-error bounds}

In this section we derive a priori error estimates for the fully discrete scheme \eqref{eq-FEM-1}--\eqref{eq-FEM-2}, whose well-posedness and convergence were established in Sections \ref{sec-wellposed} and \ref{sec-convergence}. The main result, Theorem \ref{Main theorem of error}, shows that the scheme is first-order accurate in time, with a spatial rate governed by the conjugate exponent $p'=\frac{p}{p-1}$ rather than by the exponent $2$ of the Newtonian theory.

This exponent has a single structural origin. For $p>2$ the $p$-Laplace dissipation controls only the $p$-th power of the velocity-error gradient, so a term can be absorbed into it only through Young's inequality with the pair $\left(p,p'\right)$, at the cost of the accompanying power of $\Delta t$ or $h$; the magnetic dissipation, by contrast, is quadratic. We therefore use the $\left(p,p'\right)$--Young inequality exactly once, for the one term that requires it, and route every other term through the magnetic dissipation or the discrete Gronwall lemma. The details are given in Section \ref{subsec-strategy}, after the notation, error splitting and projections are fixed in Section \ref{subsec-err-setting}, the two elementary tools in Section \ref{subsec-err-tools}, and the error equations in Section \ref{subsec-err-equations}; the individual estimates occupy Sections \ref{subsec-consistency}--\ref{subsec-coupling}, and the main theorem and a discussion of the sharpness of the exponents follow in Sections \ref{subsec-main-error} and \ref{subsec-err-discussion}.

\subsection{Standing hypotheses, projections and notation}\label{subsec-err-setting}

Throughout this section $C$ denotes a generic positive constant, independent of $\Delta t$ and $h$, whose value may change from line to line. For any arbitrary function $\varphi$, we write
\[
p'=\frac{p}{p-1}\in(1,2] \quad (p\ge 2),\qquad
I_{n}=(t_{n-1},t_{n}),\qquad
d_{t}\varphi^{n}=\frac{\varphi^{n}-\varphi^{n-1}}{\Delta t}.
\]
Since $\mathbb{D}$ is bounded and $p\ge2$, H\"older's inequality gives
\begin{equation}\label{eq-Lp-L2}
\|\mathbf{v}\|_{\L^{2}}\le|\mathbb{D}|^{\frac12-\frac1p}\|\mathbf{v}\|_{\L^{p}}\qquad\text{for all }\mathbf{v}\in\L^{p}(\mathbb{D}),
\end{equation}
which we use freely to pass between the $\mathbf{L}^{p}$ scale of the velocity and the $\mathbf{L}^{2}$ scale of the magnetic field.\\
\noindent
The analysis compares three objects at each time level $t_{n}$: the exact solution $(\u^{n},\B^{n})$, its projection onto the discrete spaces, and the computed solution $(\mathbf{U}^{n},\mathfrak{B}^{n})$. Denoting by $\mathcal{P}_{h}$ and $\mathcal{F}_{h}$ the velocity and magnetic projections defined in \eqref{eq-Ph-def}--\eqref{eq-Fh-def} below, we split the total error in the standard way as
\begin{equation}\label{eq-error-splitting}
\underbrace{\u^{n}-\mathbf{U}^{n}}_{\text{total velocity error}}
=\underbrace{\left(\u^{n}-\mathcal{P}_{h}\u^{n}\right)}_{=:\ \eta^{n}_{\u}\ \text{(projection part)}}
-\underbrace{\left(\mathbf{U}^{n}-\mathcal{P}_{h}\u^{n}\right)}_{=:\ E^{n}_{\u}\ \text{(discrete part)}},
\qquad
\B^{n}-\mathfrak{B}^{n}=\eta^{n}_{\B}-E^{n}_{\B},
\end{equation}
with $\eta^{n}_{\B}=\B^{n}-\mathcal{F}_{h}\B^{n}$ and $E^{n}_{\B}=\mathfrak{B}^{n}-\mathcal{F}_{h}\B^{n}$. The projection parts $\eta$ are estimated directly by approximation theory (see \eqref{eq-error approximation}--\eqref{eq-error approximation-B} below), while the discrete parts $E$ are the unknowns of the analysis; when the time level is immaterial we drop the superscript and write $E_{\u},E_{\B}$. Every further symbol is defined where it first appears.\\
\noindent
We next record the regularity required of the exact solution.

\begin{assumption}[Regularity]\label{Assumption}
Let $(\u,\B)$ be the solution of \eqref{1st weak form}--\eqref{2nd weeak equation} and suppose that
\begin{center}
$\u\in L^\infty(0,T;\mathbf{W}^{s+1,p}(\mathbb{D}))$,\qquad $\B\in L^\infty(0,T;\mathbf{W}^{s,2}(\mathbb{D}))$, \\
$\nabla\times\B\in L^\infty(0,T;\mathbf{W}^{s,2}(\mathbb{D}))$,\qquad $\u_{t}\in L^2(0,T;\mathbf{W}^{s+1,p}(\mathbb{D}))$,\qquad $\B_{t}\in L^2(0,T;\mathbf{W}^{s,2}(\mathbb{D}))$,\\
$\u_{tt}\in L^2(0,T;\L^{2}(\mathbb{D}))$,\qquad $\B_{tt}\in L^2(0,T;\L^{2}(\mathbb{D}))$,
\end{center}
where the exponent $s\ge\frac12$ depends on $\mathbb{D}$, and where $sp>3$, so that the Sobolev embedding $\mathbf{W}^{s+1,p}(\mathbb{D})\hookrightarrow\mathbf{W}^{1,\infty}(\mathbb{D})$ holds in three dimensions.
\end{assumption}
\noindent
The time derivatives are required in the same spatial scale as $\u,\B$ so that the consistency terms of Section \ref{subsec-consistency} can be estimated through the integral form of the projection-error difference quotient; this is what makes the estimates unconditional. The condition $sp>3$ gives $\nabla\u\in\L^{\infty}$, used in Lemma \ref{lem-trilinear-L2}, and holds already for $s\ge1$ in the regime $p\ge5$ of the experiments.\\
\noindent
The second hypothesis concerns the velocity element.

\begin{assumption}[Pointwise divergence-free velocity space]\label{Assumption-divfree}
The discrete velocity space satisfies $\operatorname{div}\v_{h}=0$ pointwise in $\mathbb{D}$ for every $\v_{h}\in\mathcal{X}_{0h}$.
\end{assumption}

\begin{remark}[On Assumption \ref{Assumption-divfree}]\label{rem-divfree}
Assumption \ref{Assumption-divfree} holds for pressure-robust, exactly divergence-free pairs (Scott--Vogelius on suitable meshes, $\mathbf{H}(\operatorname{div})$-conforming DG), but not for Taylor--Hood, where $\operatorname{div}\v_{h}$ is only $L^{2}$-orthogonal to the discrete pressure. It enters only through Lemma \ref{lem-trilinear-L2}, and it is genuinely needed: without it the identity \eqref{eq-trilinear-identity} retains the defect $\tfrac12\int_{\mathbb{D}}(\operatorname{div}\mathbf{w})(\mathbf{v}\cdot\boldsymbol\phi)$, contributing a term $\sim\|\nabla E^{n-1}_{\u}\|_{\L^{2}}\|E^{n}_{\u}\|_{\L^{2}}$ to Lemma \ref{lem-convective}. For $p>2$ this quadratic quantity in $\nabla E$ cannot be absorbed into the $p$-th power dissipation, and routing it through $\left(p,p'\right)$--Young leaves $\|E^{n}_{\u}\|^{p'}_{\L^{2}}$ with $p'<2$, on which the discrete Gronwall lemma gives only a bound with a non-vanishing floor. Removing the assumption thus seems to require the quasi-norm framework of Remark \ref{rem-quasinorm}, and we leave it as an open problem.
\end{remark}
\noindent
We next fix the two projections. Let $\mathcal{P}_{h}:=\mathcal{P}_{p,h}$ denote the discrete Stokes projection of Section \ref{sec-fem}, that is, given $\u\in\mathcal{X}_{0}$, $\mathcal{P}_{h}\u\in\mathcal{X}_{0h}$ is determined by
\begin{equation}\label{eq-Ph-def}
\left(\nabla\mathcal{P}_{h}\u,\nabla\v_{h}\right)=\left(\nabla\u,\nabla\v_{h}\right)\qquad\forall\,\v_{h}\in\mathcal{X}_{0h},
\end{equation}
and let $\mathcal{F}_{h}:\mathcal{Y}_{0}\to\mathcal{Y}_{h}$ be determined by
\begin{equation}\label{eq-Fh-def}
\frac{1}{\mu\sigma}\left(\nabla\times\mathcal{F}_{h}\B,\nabla\times\C_{h}\right)=\frac{1}{\mu\sigma}\left(\nabla\times\B,\nabla\times\C_{h}\right),\qquad
\left(\mathcal{F}_{h}\B,\nabla\psi_{h}\right)=\left(\B,\nabla\psi_{h}\right),
\end{equation}
for all $\C_{h}\in\mathcal{Y}_{h}$ and $\psi_{h}\in S_{h}$. We use \eqref{eq-Fh-def} through the orthogonality $\left(\nabla\times(\B-\mathcal{F}_{h}\B),\nabla\times\C_{h}\right)=0$, which cancels the magnetic diffusion defect. The nonlinear operator $a_{0}$ admits no such orthogonality; its defect is treated in Section \ref{subsec-plaplace-defect}.\\
\noindent
These projections satisfy the approximation properties
\begin{align}\label{eq-error approximation}
\|\u-\mathcal{P}_{h}\u\|_{\L^{p}}+h\|\nabla(\u-\mathcal{P}_{h}\u)\|_{\L^{p}}&\le C_{e}h^{1+l}\|\u\|_{1+l,p},\\
\label{eq-error approximation-B}
\|\B-\mathcal{F}_{h}\B\|_{\L^{2}}+\|\nabla\times(\B-\mathcal{F}_{h}\B)\|_{\L^{2}}&\le C_{e}h^{l}\left(\|\B\|_{l,2}+\|\nabla\times\B\|_{l,2}\right),
\end{align}
with $l=\min\{k,s\}$ and $k\ge1$ the velocity polynomial degree; the velocity estimate is in the $\mathbf{W}^{1,p}$ scale and the magnetic one in the $\mathbf{L}^{2}$ scale, reconciled by \eqref{eq-Lp-L2}. We also use the uniform bounds
\begin{equation}\label{eq-Cr-bounds}
\|\u-\mathcal{P}_{h}\u\|_{\L^{\infty}}+\|\nabla\mathcal{P}_{h}\u\|_{\L^{\infty}}
+\|\nabla\times(\B-\mathcal{F}_{h}\B)\|_{\L^{3}}+\|\mathcal{F}_{h}\B\|_{\L^{\infty}}\le C_{r},
\end{equation}
valid on quasi-uniform meshes under Assumption \ref{Assumption}; see Lemma \ref{eq-error-estimate bounds}.\\
\noindent
The projection parts $\eta^{n}_{\u},\eta^{n}_{\B}$ in the splitting \eqref{eq-error-splitting} are now controlled directly by \eqref{eq-error approximation}--\eqref{eq-error approximation-B}; the analysis below is devoted to the discrete parts $E^{n}_{\u},E^{n}_{\B}$.

\subsection{Two elementary tools}\label{subsec-err-tools}

The first lemma is what preserves the optimal temporal rate. In its skew-symmetric form $a_{1}(\mathbf{w},\mathbf{v},\boldsymbol\phi)$ carries a derivative of the test function $\boldsymbol\phi$; integrating by parts once moves it onto the middle argument (in our applications the exact solution or its projection), leaving only $\|\boldsymbol\phi\|_{\L^{2}}$.

\begin{lemma}[$\mathbf{L}^{2}$ bound for the convective form]\label{lem-trilinear-L2}
Let $\mathbf{w}\in\mathbf{W}^{1,p}_{0}(\mathbb{D})$, $\mathbf{v}\in\mathbf{W}^{1,\infty}(\mathbb{D})$ and $\boldsymbol\phi\in\mathbf{W}^{1,p}_{0}(\mathbb{D})$. Then
\begin{equation}\label{eq-trilinear-identity}
a_{1}(\mathbf{w},\mathbf{v},\boldsymbol\phi)
=\int_{\mathbb{D}}\left(\mathbf{w}\cdot\nabla\right)\mathbf{v}\cdot\boldsymbol\phi\,dx
+\frac12\int_{\mathbb{D}}\left(\operatorname{div}\mathbf{w}\right)\left(\mathbf{v}\cdot\boldsymbol\phi\right)dx .
\end{equation}
If moreover $\operatorname{div}\mathbf{w}=0$ pointwise, then
\begin{equation}\label{eq-trilinear-bound}
\left|a_{1}(\mathbf{w},\mathbf{v},\boldsymbol\phi)\right|
\le\|\mathbf{w}\|_{\L^{2}}\,\|\nabla\mathbf{v}\|_{\L^{\infty}}\,\|\boldsymbol\phi\|_{\L^{2}} .
\end{equation}
\end{lemma}

\begin{proof}
By definition,
$a_{1}(\mathbf{w},\mathbf{v},\boldsymbol\phi)=\tfrac12\int_{\mathbb{D}}(\mathbf{w}\cdot\nabla)\mathbf{v}\cdot\boldsymbol\phi\,dx-\tfrac12\int_{\mathbb{D}}(\mathbf{w}\cdot\nabla)\boldsymbol\phi\cdot\mathbf{v}\,dx$.
Writing the second integral in components and integrating by parts, the boundary term vanishing because $\boldsymbol\phi\in\mathbf{W}^{1,p}_{0}(\mathbb{D})$, we obtain
\[
\int_{\mathbb{D}}w_{j}\,\partial_{j}\phi_{i}\,v_{i}\,dx
=-\int_{\mathbb{D}}\partial_{j}\!\left(w_{j}v_{i}\right)\phi_{i}\,dx
=-\int_{\mathbb{D}}\left(\operatorname{div}\mathbf{w}\right)\left(\mathbf{v}\cdot\boldsymbol\phi\right)dx-\int_{\mathbb{D}}\left(\mathbf{w}\cdot\nabla\right)\mathbf{v}\cdot\boldsymbol\phi\,dx .
\]
Substituting this expression into the definition and collecting the two identical contributions yields \eqref{eq-trilinear-identity}. If $\operatorname{div}\mathbf{w}=0$, the second term of \eqref{eq-trilinear-identity} disappears and H\"older's inequality gives \eqref{eq-trilinear-bound}.
\end{proof}
\noindent
Since $\boldsymbol\phi$ enters \eqref{eq-trilinear-bound} only through $\|\boldsymbol\phi\|_{\L^{2}}$, the choice $\boldsymbol\phi=E^{n}_{\u}$ with $(2,2)$--Young yields $\|E^{n}_{\u}\|^{2}_{\L^{2}}$ (absorbed by Gronwall) and, after summation, the temporal factor $(\Delta t)^{2}$.\\
\noindent
The second lemma records the two standard properties of $a_{0}$.

\begin{lemma}[Monotonicity and continuity of $a_{0}$]\label{lem-a0}
Let $p\ge2$ and let $a_{0}(\mathbf{v})=\frac{\nu}{\rho}|\nabla\mathbf{v}|^{p-2}\nabla\mathbf{v}$. Then for all $\boldsymbol\xi,\boldsymbol\eta\in\R^{3\times3}$,
\begin{align}
\left(|\boldsymbol\xi|^{p-2}\boldsymbol\xi-|\boldsymbol\eta|^{p-2}\boldsymbol\eta\right):\left(\boldsymbol\xi-\boldsymbol\eta\right)&\ge2^{2-p}\left|\boldsymbol\xi-\boldsymbol\eta\right|^{p},\label{eq-monotone}\\
\left||\boldsymbol\xi|^{p-2}\boldsymbol\xi-|\boldsymbol\eta|^{p-2}\boldsymbol\eta\right|&\le C_{1}\left|\boldsymbol\xi-\boldsymbol\eta\right|\left(|\boldsymbol\xi|+|\boldsymbol\eta|\right)^{p-2}.\label{eq-lipschitz}
\end{align}
\end{lemma}
\noindent
Here \eqref{eq-monotone} supplies the dissipation on the left of the error equation and \eqref{eq-lipschitz} bounds the projection defect on the right; both are classical \cite{BarrettLiu1994,DieningEbmeyerRuzicka2007}.

\subsection{The error equations}\label{subsec-err-equations}

Evaluating \eqref{1st weak form}--\eqref{2nd weeak equation} at $t=t_{n}$ with the discrete test functions $\v_{h}\in\mathcal{X}_{0h}$ and $\C_{h}\in\mathcal{Y}_{h}$, subtracting \eqref{eq-FEM-1}--\eqref{eq-FEM-2}, and inserting the projections, we obtain the error equations
\begin{align}
\left(d_{t}E^{n}_{\u},\v_{h}\right)+\frac{\nu}{\rho}\left(a_{0}(\mathbf{U}^{n})-a_{0}(\mathcal{P}_{h}\u^{n}),\nabla\v_{h}\right)
&=\left(M^{n}_{h,1},\v_{h}\right)+\mathcal{G}^{n}(\v_{h})+\mathcal{N}^{n}(\v_{h})+\mathcal{K}^{n}(\v_{h}),\label{eq-error-system-1}\\
\left(d_{t}E^{n}_{\B},\C_{h}\right)+\frac{1}{\mu\sigma}\left(\nabla\times E^{n}_{\B},\nabla\times\C_{h}\right)
&=\left(M^{n}_{h,2},\C_{h}\right)+\mathcal{I}^{n}(\C_{h}),\label{eq-error-system-2}
\end{align}
where the consistency terms are
\[
M^{n}_{h,1}:=\left(\partial_{t}\u^{n}-d_{t}\u^{n}\right)+d_{t}\eta^{n}_{\u},
\qquad
M^{n}_{h,2}:=\left(\partial_{t}\B^{n}-d_{t}\B^{n}\right)+d_{t}\eta^{n}_{\B},
\]
the $p$-Laplace projection defect is
\[
\mathcal{G}^{n}(\v_{h}):=\frac{\nu}{\rho}\left(a_{0}(\u^{n})-a_{0}(\mathcal{P}_{h}\u^{n}),\nabla\v_{h}\right),
\]
the convective defect is
\[
\mathcal{N}^{n}(\v_{h}):=a_{1}\!\left(\u^{n},\u^{n},\v_{h}\right)-a_{1}\!\left(\mathbf{U}^{n-1},\mathbf{U}^{n},\v_{h}\right),
\]
and the electromagnetic defects are
\begin{align*}
\mathcal{K}^{n}(\v_{h})&:=\frac{1}{\mu\rho}\left[\left(\B^{n}\times(\nabla\times\B^{n}),\v_{h}\right)-\left(\mathfrak{B}^{n-1}\times(\nabla\times\mathfrak{B}^{n}),\v_{h}\right)\right],\\
\mathcal{I}^{n}(\C_{h})&:=\left(\mathbf{U}^{n}\times\mathfrak{B}^{n-1},\nabla\times\C_{h}\right)-\left(\u^{n}\times\B^{n},\nabla\times\C_{h}\right).
\end{align*}
The magnetic diffusion defect $\frac{1}{\mu\sigma}\left(\nabla\times\eta^{n}_{\B},\nabla\times\C_{h}\right)$ is absent from \eqref{eq-error-system-2} because it vanishes identically by the defining property \eqref{eq-Fh-def} of $\mathcal{F}_{h}$.\\
\noindent
We now choose $\v_{h}=E^{n}_{\u}$ in \eqref{eq-error-system-1} and $\C_{h}=\frac{1}{\mu\rho}E^{n}_{\B}$ in \eqref{eq-error-system-2}. Using the elementary identity $2\,a\,(a-b)=a^{2}-b^{2}+(a-b)^{2}$ for the discrete time derivatives and Lemma \ref{lem-a0}\,\eqref{eq-monotone} for the dissipation, the sum of the two resulting relations gives
\begin{align}\label{eq-error-master}
\nonumber
&\frac{1}{2}d_{t}\|E^{n}_{\u}\|^{2}_{\L^{2}}+\frac{1}{2\mu\rho}d_{t}\|E^{n}_{\B}\|^{2}_{\L^{2}}
+\frac{\Delta t}{2}\|d_{t}E^{n}_{\u}\|^{2}_{\L^{2}}+\frac{\Delta t}{2\mu\rho}\|d_{t}E^{n}_{\B}\|^{2}_{\L^{2}}\\
&\qquad+\frac{\nu}{\rho}2^{2-p}\|\nabla E^{n}_{\u}\|^{p}_{\L^{p}}
+\frac{1}{\mu^{2}\sigma\rho}\|\nabla\times E^{n}_{\B}\|^{2}_{\L^{2}}
\le\ \underbrace{\left(M^{n}_{h,1},E^{n}_{\u}\right)+\frac{1}{\mu\rho}\left(M^{n}_{h,2},E^{n}_{\B}\right)}_{\text{Section \ref{subsec-consistency}}}\\
\nonumber
&\qquad\qquad+\underbrace{\mathcal{G}^{n}(E^{n}_{\u})}_{\text{Section \ref{subsec-plaplace-defect}}}
+\underbrace{\mathcal{N}^{n}(E^{n}_{\u})}_{\text{Section \ref{subsec-convective}}}
+\underbrace{\mathcal{K}^{n}(E^{n}_{\u})+\frac{1}{\mu\rho}\mathcal{I}^{n}(E^{n}_{\B})}_{\text{Section \ref{subsec-coupling}}} .
\end{align}
\subsection{Strategy of the proof}\label{subsec-strategy}

The left of \eqref{eq-error-master} provides the quantities we may spend against: the two discrete time derivatives and the viscous and magnetic dissipations. Each term on the right is bounded by a fraction of these plus a Gronwall term ($\|E^{n}_{\u}\|^{2}_{\L^{2}}$ or $\|E^{n-1}_{\B}\|^{2}_{\L^{2}}$) or data (powers of $\Delta t,h$). How the absorption is done is governed by two facts.

{(i) The viscous dissipation is of $p$-th power type:} A term of the form $\|\nabla E^{n}_{\u}\|_{\L^{p}}$ times a gradient-free prefactor can be absorbed into $\|\nabla E^{n}_{\u}\|^{p}_{\L^{p}}$ only by $\left(p,p'\right)$--Young, which leaves the prefactor to the power $p'\le2$; a factor $(\Delta t)^{1/2}$ then becomes $(\Delta t)^{p'/2}$ rather than $\Delta t$. We use this step once only, for the $p$-Laplace defect $\mathcal{G}^{n}$ in Lemma \ref{lem-plaplace-defect}. As $\mathcal{G}^{n}$ carries no power of $\Delta t$, the loss falls entirely on the spatial exponent and gives the $h^{lp/(p-1)}$ of Theorem \ref{Main theorem of error}.

{(ii) Everything else is routed through $\mathbf{L}^{2}$:} The consistency defects $M^{n}_{h,1},M^{n}_{h,2}$ and the convective defect $\mathcal{N}^{n}$ are estimated so that the error enters only through $\|E^{n}_{\u}\|_{\L^{2}}$, not its gradient; for $\mathcal{N}^{n}$ this is what Lemma \ref{lem-trilinear-L2} provides. The pair $(2,2)$ then squares each $(\Delta t)^{1/2}$, and summation supplies a second factor, giving $(\Delta t)^{2}$. The coupling defects $\mathcal{K}^{n},\mathcal{I}^{n}$ pair partly with $\nabla\times E^{n}_{\B}$, but the magnetic dissipation is quadratic, so $(2,2)$ again suffices without loss.\\
\noindent
Concretely, the four estimates run as follows. The consistency defects $M^{n}_{h,1},M^{n}_{h,2}$ (Lemma \ref{lem-consistency}) and the convective defect $\mathcal{N}^{n}$ (Lemma \ref{lem-convective}) are absorbed into $\|E_{\u}\|^{2}_{\L^{2}}$ at cost $(\Delta t)^{2}$ and $h^{2l}$ (with an additional $h^{2(l+1)}$ from consistency); the $p$-Laplace defect $\mathcal{G}^{n}$ (Lemma \ref{lem-plaplace-defect}) is absorbed into $\|\nabla E_{\u}\|^{p}_{\L^{p}}$ at cost $h^{lp/(p-1)}$; and the coupling defects $\mathcal{K}^{n},\mathcal{I}^{n}$ (Lemma \ref{lem-coupling}) are absorbed into $\|\nabla\times E_{\B}\|^{2}_{\L^{2}}$ and $\|E_{\B}\|^{2}_{\L^{2}}$ at cost $(\Delta t)^{2}$ and $h^{2l}$. Since $p'\le2$ gives $\frac{lp}{p-1}\le2l\le2(l+1)$, the $p$-Laplace defect sets the binding spatial exponent while the others set the binding temporal one. We now carry out the four estimates.

\subsection{The consistency terms}\label{subsec-consistency}

\begin{lemma}[Consistency]\label{lem-consistency}
Under Assumption \ref{Assumption} there holds, for $n=1,\dots,N$,
\begin{align*}
\left|\left(M^{n}_{h,1},E^{n}_{\u}\right)\right|+\frac{1}{\mu\rho}\left|\left(M^{n}_{h,2},E^{n}_{\B}\right)\right|
&\le\frac18\|E^{n}_{\u}\|^{2}_{\L^{2}}+\frac{1}{8\mu\rho}\|E^{n}_{\B}\|^{2}_{\L^{2}}\\
&\quad+C\Delta t\left\{\|\partial_{tt}\u\|^{2}_{L^{2}(I_{n};\L^{2})}+\|\partial_{tt}\B\|^{2}_{L^{2}(I_{n};\L^{2})}\right\}\\
&\quad+\frac{C}{\Delta t}\left\{h^{2(l+1)}\|\partial_{t}\u\|^{2}_{L^{2}(I_{n};\mathbf{W}^{l+1,p})}+h^{2l}\|\partial_{t}\B\|^{2}_{L^{2}(I_{n};\mathbf{W}^{l,2})}\right\}.
\end{align*}
\end{lemma}

\begin{proof}
We treat the two contributions to $M^{n}_{h,1}$ separately. For the temporal truncation error, Taylor's theorem with integral remainder gives
\[
\partial_{t}\u^{n}-d_{t}\u^{n}=\frac{1}{\Delta t}\int_{t_{n-1}}^{t_{n}}(s-t_{n-1})\,\partial_{tt}\u(s)\,ds ,
\]
whence, by the triangle and Cauchy--Schwarz inequalities,
\[
\left\|\partial_{t}\u^{n}-d_{t}\u^{n}\right\|_{\L^{2}}
\le\int_{t_{n-1}}^{t_{n}}\left\|\partial_{tt}\u(s)\right\|_{\L^{2}}ds
\le(\Delta t)^{1/2}\left\|\partial_{tt}\u\right\|_{L^{2}(I_{n};\L^{2})} .
\]
For the projection contribution we use the integral representation of the difference quotient,
\[
d_{t}\eta^{n}_{\u}=\frac{1}{\Delta t}\int_{t_{n-1}}^{t_{n}}\partial_{t}\left(\u-\mathcal{P}_{h}\u\right)(s)\,ds ,
\]
which is legitimate because $\mathcal{P}_{h}$ is independent of $t$ and $\u_{t}\in L^{2}(0,T;\mathbf{W}^{s+1,p})$ by Assumption \ref{Assumption}. Cauchy--Schwarz in time together with \eqref{eq-error approximation} and \eqref{eq-Lp-L2} then yields
\[
\left\|d_{t}\eta^{n}_{\u}\right\|_{\L^{2}}
\le(\Delta t)^{-1/2}\left\|\partial_{t}\eta_{\u}\right\|_{L^{2}(I_{n};\L^{2})}
\le C_{e}h^{l+1}(\Delta t)^{-1/2}\left\|\partial_{t}\u\right\|_{L^{2}(I_{n};\mathbf{W}^{l+1,p})} .
\]
This integral form is what removes the negative power of $\Delta t$ that a pointwise estimate of $d_{t}\eta^{n}_{\u}$ would incur, making the bound unconditional. Combining the two displays and applying Young's inequality with the conjugate pair $(2,2)$,
\[
\left|\left(M^{n}_{h,1},E^{n}_{\u}\right)\right|
\le\frac18\|E^{n}_{\u}\|^{2}_{\L^{2}}
+C\Delta t\|\partial_{tt}\u\|^{2}_{L^{2}(I_{n};\L^{2})}
+\frac{C}{\Delta t}h^{2(l+1)}\|\partial_{t}\u\|^{2}_{L^{2}(I_{n};\mathbf{W}^{l+1,p})} .
\]
The estimate for $M^{n}_{h,2}$ is identical, with \eqref{eq-error approximation-B} in place of \eqref{eq-error approximation}, and produces $h^{2l}$ instead of $h^{2(l+1)}$.
\end{proof}

\subsection{The convective defect}\label{subsec-convective}

\begin{lemma}[Convective defect]\label{lem-convective}
Under Assumptions \ref{Assumption} and \ref{Assumption-divfree} there holds
\[
\left|\mathcal{N}^{n}(E^{n}_{\u})\right|
\le\frac38\|E^{n}_{\u}\|^{2}_{\L^{2}}+C\|E^{n-1}_{\u}\|^{2}_{\L^{2}}
+C\left\{\Delta t\|\partial_{t}\u\|^{2}_{L^{2}(I_{n};\L^{2})}+h^{2l}\right\},
\]
where $C$ depends on $C_{e},C_{r}$ and on $\|\u\|_{L^{\infty}(0,T;\mathbf{W}^{s+1,p})}$, but not on $\Delta t$ or $h$.
\end{lemma}

\begin{proof}
{Step 1: decomposition.} Since $\mathbf{U}^{n}=\mathcal{P}_{h}\u^{n}+E^{n}_{\u}$ and $a_{1}(\mathbf{w},\boldsymbol\phi,\boldsymbol\phi)=0$ for every $\mathbf{w}$, we have
\[
a_{1}\!\left(\mathbf{U}^{n-1},\mathbf{U}^{n},E^{n}_{\u}\right)
=a_{1}\!\left(\mathbf{U}^{n-1},\mathcal{P}_{h}\u^{n},E^{n}_{\u}\right)
+\underbrace{a_{1}\!\left(\mathbf{U}^{n-1},E^{n}_{\u},E^{n}_{\u}\right)}_{=0}.
\]
Writing $\mathbf{U}^{n-1}=\mathcal{P}_{h}\u^{n-1}+E^{n-1}_{\u}$ and using linearity in the first two arguments,
\begin{equation}\label{eq-N-decomp}
\mathcal{N}^{n}(E^{n}_{\u})=\Theta_{1}+\Theta_{2}-\Theta_{3},
\end{equation}
where
\[
\Theta_{1}:=a_{1}\!\left(\u^{n}-\mathcal{P}_{h}\u^{n-1},\u^{n},E^{n}_{\u}\right),\qquad
\Theta_{2}:=a_{1}\!\left(\mathcal{P}_{h}\u^{n-1},\eta^{n}_{\u},E^{n}_{\u}\right),\qquad
\Theta_{3}:=a_{1}\!\left(E^{n-1}_{\u},\mathcal{P}_{h}\u^{n},E^{n}_{\u}\right).
\]
Indeed, $a_{1}(\u^{n},\u^{n},\cdot)-a_{1}(\mathcal{P}_{h}\u^{n-1},\mathcal{P}_{h}\u^{n},\cdot)=\Theta_{1}+\Theta_{2}$ by adding and subtracting $a_{1}(\mathcal{P}_{h}\u^{n-1},\u^{n},\cdot)$. Compared with the more common seven-term splitting, the decomposition \eqref{eq-N-decomp} has the advantage that the two contributions which vanish by skew-symmetry have already been removed, so that no term is retained whose first argument is identically zero.\\
\noindent
{Step 2: the term $\Theta_{1}$.} By Assumption \ref{Assumption-divfree} and because $\operatorname{div}\u^{n}=0$, the first argument of $\Theta_{1}$ is divergence free, so Lemma \ref{lem-trilinear-L2} applies with $\mathbf{v}=\u^{n}\in\mathbf{W}^{1,\infty}$:
\[
\left|\Theta_{1}\right|\le\left\|\u^{n}-\mathcal{P}_{h}\u^{n-1}\right\|_{\L^{2}}\left\|\nabla\u^{n}\right\|_{\L^{\infty}}\left\|E^{n}_{\u}\right\|_{\L^{2}} .
\]
Adding and subtracting $\u^{n-1}$ and using \eqref{eq-error approximation} together with $\|\u^{n}-\u^{n-1}\|_{\L^{2}}\le(\Delta t)^{1/2}\|\partial_{t}\u\|_{L^{2}(I_{n};\L^{2})}$,
\[
\left\|\u^{n}-\mathcal{P}_{h}\u^{n-1}\right\|_{\L^{2}}
\le(\Delta t)^{1/2}\left\|\partial_{t}\u\right\|_{L^{2}(I_{n};\L^{2})}+C_{e}h^{l+1}\left\|\u^{n-1}\right\|_{l+1,p} .
\]
Young's inequality with the pair $(2,2)$ gives
\[
\left|\Theta_{1}\right|\le\frac18\|E^{n}_{\u}\|^{2}_{\L^{2}}+C\Delta t\|\partial_{t}\u\|^{2}_{L^{2}(I_{n};\L^{2})}+Ch^{2(l+1)} .
\]
\noindent
{Step 3: the term $\Theta_{2}$.} Here the first argument $\mathcal{P}_{h}\u^{n-1}$ is divergence free by Assumption \ref{Assumption-divfree}, and the middle argument is the small quantity $\eta^{n}_{\u}$. Applying \eqref{eq-trilinear-identity} and estimating the single surviving integral by H\"older's inequality with exponents $(\infty,2,2)$,
\[
\left|\Theta_{2}\right|
\le\left\|\mathcal{P}_{h}\u^{n-1}\right\|_{\L^{\infty}}\left\|\nabla\eta^{n}_{\u}\right\|_{\L^{2}}\left\|E^{n}_{\u}\right\|_{\L^{2}}
\le C_{r}\,C_{e}h^{l}\left\|\u^{n}\right\|_{l+1,p}\left\|E^{n}_{\u}\right\|_{\L^{2}},
\]
where we used \eqref{eq-Lp-L2} and the gradient part of \eqref{eq-error approximation}, which carries $h^{l}$ and not $h^{l+1}$. Young's inequality with the pair $(2,2)$ yields
\[
\left|\Theta_{2}\right|\le\frac18\|E^{n}_{\u}\|^{2}_{\L^{2}}+Ch^{2l} .
\]
\noindent
{Step 4: the term $\Theta_{3}$.} The first argument $E^{n-1}_{\u}$ is a difference of two elements of $\mathcal{X}_{0h}$ and is therefore divergence free by Assumption \ref{Assumption-divfree}. Lemma \ref{lem-trilinear-L2} with $\mathbf{v}=\mathcal{P}_{h}\u^{n}$, whose gradient is uniformly bounded by \eqref{eq-Cr-bounds}, gives
\[
\left|\Theta_{3}\right|
\le\left\|E^{n-1}_{\u}\right\|_{\L^{2}}\left\|\nabla\mathcal{P}_{h}\u^{n}\right\|_{\L^{\infty}}\left\|E^{n}_{\u}\right\|_{\L^{2}}
\le C_{r}\left\|E^{n-1}_{\u}\right\|_{\L^{2}}\left\|E^{n}_{\u}\right\|_{\L^{2}},
\]
and Young's inequality with the pair $(2,2)$ produces
\[
\left|\Theta_{3}\right|\le\frac18\|E^{n}_{\u}\|^{2}_{\L^{2}}+C\left\|E^{n-1}_{\u}\right\|^{2}_{\L^{2}} .
\]
Both terms on the right are of Gronwall type. This is the step that uses Assumption \ref{Assumption-divfree}: without it, \eqref{eq-trilinear-identity} would leave the extra term $\tfrac12\int_{\mathbb{D}}(\operatorname{div}E^{n-1}_{\u})(\mathcal{P}_{h}\u^{n}\cdot E^{n}_{\u})$, quadratic in $\nabla E$ and hence not absorbable into the $p$-th power dissipation (Remark \ref{rem-divfree}).\\
\noindent         
{Step 5.} Collecting Steps 2--4 in \eqref{eq-N-decomp} and using $h^{2(l+1)}\le h^{2l}$ for $h\le1$ completes the proof.
\end{proof}

\subsection{The \texorpdfstring{$p$}{p}-Laplace projection defect}\label{subsec-plaplace-defect}

The following lemma isolates the one place in the entire analysis where Young's inequality with the conjugate pair $\left(p,p'\right)$ is unavoidable. The reason is structural: unlike the magnetic diffusion, whose defect vanishes identically by \eqref{eq-Fh-def}, the nonlinear operator $a_{0}$ admits no exact discrete orthogonality, and the resulting defect is paired with $\nabla E^{n}_{\u}$, which can be controlled only through the $p$-th power dissipation.

\begin{lemma}[$p$-Laplace projection defect]\label{lem-plaplace-defect}
Let $\alpha>0$. Under Assumption \ref{Assumption} there holds
\[
\left|\mathcal{G}^{n}(E^{n}_{\u})\right|
\le\frac{\nu}{\rho}\frac{\alpha^{p}}{p}\left\|\nabla E^{n}_{\u}\right\|^{p}_{\L^{p}}
+C_{\alpha}\,h^{\frac{lp}{p-1}}\left\|\u^{n}\right\|^{\frac{p}{p-1}}_{l+1,p}
\left\|\,|\nabla\u^{n}|+|\nabla\mathcal{P}_{h}\u^{n}|\,\right\|^{\frac{(p-2)p}{p-1}}_{\L^{p}} .
\]
\end{lemma}

\begin{proof}
By \eqref{eq-lipschitz},
\[
\left|\mathcal{G}^{n}(E^{n}_{\u})\right|
\le C_{1}\frac{\nu}{\rho}\int_{\mathbb{D}}
\left|\nabla\eta^{n}_{\u}\right|
\left(\left|\nabla\u^{n}\right|+\left|\nabla\mathcal{P}_{h}\u^{n}\right|\right)^{p-2}
\left|\nabla E^{n}_{\u}\right|dx .
\]
We apply H\"older's inequality with the three exponents $p$, $p$ and $\frac{p}{p-2}$, whose reciprocals sum to one:
\[
\left|\mathcal{G}^{n}(E^{n}_{\u})\right|
\le C_{1}\frac{\nu}{\rho}
\left\|\nabla E^{n}_{\u}\right\|_{\L^{p}}
\left\|\nabla\eta^{n}_{\u}\right\|_{\L^{p}}
\left\|\,|\nabla\u^{n}|+|\nabla\mathcal{P}_{h}\u^{n}|\,\right\|^{p-2}_{\L^{p}} .
\]
The last factor is bounded uniformly in $h$, because $\|\nabla\mathcal{P}_{h}\u^{n}\|_{\L^{p}}\le\|\nabla\u^{n}\|_{\L^{p}}+\|\nabla\eta^{n}_{\u}\|_{\L^{p}}$ together with $\u\in L^{\infty}(0,T;\mathbf{W}^{s+1,p})$. Young's inequality with the conjugate pair $\left(p,p'\right)$, in the scaled form $ab\le\frac{\alpha^{p}a^{p}}{p}+\frac{p-1}{p}\alpha^{-p'}b^{p'}$, now gives
\[
\left|\mathcal{G}^{n}(E^{n}_{\u})\right|
\le\frac{\nu}{\rho}\frac{\alpha^{p}}{p}\left\|\nabla E^{n}_{\u}\right\|^{p}_{\L^{p}}
+C_{\alpha}\left\|\nabla\eta^{n}_{\u}\right\|^{p'}_{\L^{p}}
\left\|\,|\nabla\u^{n}|+|\nabla\mathcal{P}_{h}\u^{n}|\,\right\|^{(p-2)p'}_{\L^{p}} .
\]
It remains to insert the gradient part of \eqref{eq-error approximation}, namely $\|\nabla\eta^{n}_{\u}\|_{\L^{p}}\le C_{e}h^{l}\|\u^{n}\|_{l+1,p}$, and to observe that raising $h^{l}$ to the power $p'=\frac{p}{p-1}$ produces $h^{\frac{lp}{p-1}}$.
\end{proof}

\subsection{The electromagnetic coupling}\label{subsec-coupling}

\begin{lemma}[Coupling terms]\label{lem-coupling}
Under Assumption \ref{Assumption} there exists $c_{1}\in(0,1)$, depending only on $C_{r}$ and the physical parameters, such that
\begin{align*}
\left|\mathcal{K}^{n}(E^{n}_{\u})+\frac{1}{\mu\rho}\mathcal{I}^{n}(E^{n}_{\B})\right|
&\le\frac{c_{1}}{\mu^{2}\sigma\rho}\left\|\nabla\times E^{n}_{\B}\right\|^{2}_{\L^{2}}
+\frac{c_{1}}{\mu^{2}\sigma\rho}\left\|\nabla\times E^{n-1}_{\B}\right\|^{2}_{\L^{2}}\\
&\quad+C\left\{\left\|E^{n}_{\u}\right\|^{2}_{\L^{2}}+\left\|E^{n-1}_{\B}\right\|^{2}_{\L^{2}}\right\}
+C\left\{\Delta t\left\|\partial_{t}\B\right\|^{2}_{L^{2}(I_{n};\L^{2})}+h^{2l}\right\}.
\end{align*}
\end{lemma}

\begin{proof}
The two defects are combined before being estimated, so that the leading-order interaction cancels. Adding and subtracting the intermediate quantities $\mathfrak{B}^{n-1}\times(\nabla\times\B^{n})$ and $\u^{n}\times\mathfrak{B}^{n-1}$ and using the scalar triple product identity $(\mathbf{a}\times\mathbf{b})\cdot\mathbf{c}=(\mathbf{c}\times\mathbf{a})\cdot\mathbf{b}$, exactly as in the proof of Lemma \ref{lem-uniqueness}, one obtains
\begin{align*}
\mathcal{K}^{n}(E^{n}_{\u})+\frac{1}{\mu\rho}\mathcal{I}^{n}(E^{n}_{\B})
&=\frac{1}{\mu\rho}\Big[\left(\left(\B^{n}-\mathfrak{B}^{n-1}\right)\times(\nabla\times\B^{n}),E^{n}_{\u}\right)
+\left(\u^{n}\times\left(\mathfrak{B}^{n-1}-\B^{n}\right),\nabla\times E^{n}_{\B}\right)\\
&\qquad\quad+\left(\mathfrak{B}^{n-1}\times(\nabla\times\eta^{n}_{\B}),E^{n}_{\u}\right)
-\left(\eta^{n}_{\u}\times\mathfrak{B}^{n-1},\nabla\times E^{n}_{\B}\right)\Big] .
\end{align*}
The four contributions are estimated in turn. Writing
\[
\B^{n}-\mathfrak{B}^{n-1}=\left(\B^{n}-\B^{n-1}\right)+\eta^{n-1}_{\B}-E^{n-1}_{\B},
\]
and invoking Assumption \ref{Assumption} together with \eqref{eq-error approximation-B} gives
\begin{equation}\label{eq-Bdiff}
\left\|\B^{n}-\mathfrak{B}^{n-1}\right\|_{\L^{2}}
\le(\Delta t)^{1/2}\left\|\partial_{t}\B\right\|_{L^{2}(I_{n};\L^{2})}+C_{e}h^{l}+\left\|E^{n-1}_{\B}\right\|_{\L^{2}} .
\end{equation}
The first term is bounded by $C\|\nabla\times\B^{n}\|_{\L^{\infty}}\|\B^{n}-\mathfrak{B}^{n-1}\|_{\L^{2}}\|E^{n}_{\u}\|_{\L^{2}}$, and Young's inequality with the pair $(2,2)$ together with \eqref{eq-Bdiff} contributes $C\|E^{n}_{\u}\|^{2}_{\L^{2}}$, $C\|E^{n-1}_{\B}\|^{2}_{\L^{2}}$, $C\Delta t\|\partial_{t}\B\|^{2}_{L^{2}(I_{n};\L^{2})}$ and $Ch^{2l}$. The second term is bounded by $C\|\u^{n}\|_{\L^{\infty}}\|\B^{n}-\mathfrak{B}^{n-1}\|_{\L^{2}}\|\nabla\times E^{n}_{\B}\|_{\L^{2}}$; here the quadratic magnetic dissipation is available, so Young's inequality with the pair $(2,2)$ absorbs a fraction $\frac{c_{1}}{2\mu^{2}\sigma\rho}\|\nabla\times E^{n}_{\B}\|^{2}_{\L^{2}}$ and leaves the same remainders. For the third term we write $\mathfrak{B}^{n-1}=\mathcal{F}_{h}\B^{n-1}+E^{n-1}_{\B}$ and estimate
\[
\left|\left(\mathcal{F}_{h}\B^{n-1}\times(\nabla\times\eta^{n}_{\B}),E^{n}_{\u}\right)\right|
\le\left\|\mathcal{F}_{h}\B^{n-1}\right\|_{\L^{\infty}}\left\|\nabla\times\eta^{n}_{\B}\right\|_{\L^{2}}\left\|E^{n}_{\u}\right\|_{\L^{2}}
\le C_{r}C_{e}h^{l}\left\|E^{n}_{\u}\right\|_{\L^{2}},
\]
while the remaining piece is handled by H\"older's inequality with exponents $(6,3,2)$, the embedding $\|E^{n-1}_{\B}\|_{\L^{6}}\le C\|\nabla\times E^{n-1}_{\B}\|_{\L^{2}}$ valid for discretely divergence-free fields, and \eqref{eq-Cr-bounds}:
\[
\left|\left(E^{n-1}_{\B}\times(\nabla\times\eta^{n}_{\B}),E^{n}_{\u}\right)\right|
\le C\left\|E^{n-1}_{\B}\right\|_{\L^{6}}\left\|\nabla\times\eta^{n}_{\B}\right\|_{\L^{3}}\left\|E^{n}_{\u}\right\|_{\L^{2}}
\le CC_{r}\left\|\nabla\times E^{n-1}_{\B}\right\|_{\L^{2}}\left\|E^{n}_{\u}\right\|_{\L^{2}},
\]
which after Young's inequality with the pair $(2,2)$ yields $\frac{c_{1}}{\mu^{2}\sigma\rho}\|\nabla\times E^{n-1}_{\B}\|^{2}_{\L^{2}}+C\|E^{n}_{\u}\|^{2}_{\L^{2}}$. The fourth term is estimated in the same way, using $\|\eta^{n}_{\u}\|_{\L^{\infty}}\le C_{r}$ and $\|\eta^{n}_{\u}\|_{\L^{2}}\le C_{e}h^{l+1}$ from \eqref{eq-Cr-bounds} and \eqref{eq-error approximation}, and absorbing into $\|\nabla\times E^{n}_{\B}\|^{2}_{\L^{2}}$. Collecting the four estimates and using $h^{2(l+1)}\le h^{2l}$ gives the assertion.
\end{proof}
\noindent
Observe that the coupling terms never require the $\left(p,p'\right)$--Young inequality: those paired with $\nabla\times E_{\B}$ are absorbed into the quadratic magnetic dissipation, and those paired with $E_{\u}$ are absorbed into the Gronwall term. This is why the electromagnetic coupling, despite being the most intricate part of the system, costs nothing in either rate.

\subsection{The main error estimate}\label{subsec-main-error}

\begin{theorem}[A priori error estimate]\label{Main theorem of error}
Let Assumptions \ref{Assumption} and \ref{Assumption-divfree} hold, let $p\ge2$, and let $l=\min\{k,s\}$. Let $(\u^{n},\B^{n})$ denote the solution of \eqref{1st weak form}--\eqref{2nd weeak equation} at $t=t_{n}$ and let $(\mathbf{U}^{n},\mathfrak{B}^{n})$ be the fully discrete solution of \eqref{eq-FEM-1}--\eqref{eq-FEM-2}, which exists and is unique for every $\Delta t>0$ and $h>0$ by Theorem \ref{thm-existence} and Lemma \ref{lem-uniqueness}. Assume that the initial approximations satisfy
\[
\|\u^{0}-\mathbf{U}^{0}\|_{\L^{2}}+\|\B^{0}-\mathfrak{B}^{0}\|_{\L^{2}}\le C^{*}h^{l}.
\]
Then there exists $\Delta t_{0}>0$, required only for the application of the discrete Gronwall lemma and not for the solvability of the discrete system, such that for all $\Delta t\le\Delta t_{0}$ and all $1\le m\le N$,
\begin{align}
\left\|\u^{m}-\mathbf{U}^{m}\right\|^{2}_{\L^{2}}+(\Delta t)\sum_{n=1}^{m}\frac{\nu}{\rho}\left\|\nabla(\u^{n}-\mathbf{U}^{n})\right\|^{p}_{\L^{p}}
&\le C^{*}\left\{(\Delta t)^{2}+h^{\frac{lp}{p-1}}\right\},\label{eq-main-u}\\
\frac{1}{\mu\rho}\left\|\B^{m}-\mathfrak{B}^{m}\right\|^{2}_{\L^{2}}+(\Delta t)\sum_{n=1}^{m}\frac{1}{\mu^{2}\rho\sigma}\left\|\nabla\times(\B^{n}-\mathfrak{B}^{n})\right\|^{2}_{\L^{2}}
&\le C^{*}\left\{(\Delta t)^{2}+h^{\frac{lp}{p-1}}\right\},\label{eq-main-B}
\end{align}
where $C^{*}>0$ is independent of $\Delta t$ and $h$.
\end{theorem}

\begin{proof}
{Step 1: assembling the estimates.} We insert Lemmas \ref{lem-consistency}, \ref{lem-convective}, \ref{lem-plaplace-defect} and \ref{lem-coupling} into the master inequality \eqref{eq-error-master}. Choosing $\alpha>0$ so small that $\frac{\alpha^{p}}{p}\le\frac{2^{1-p}}{2}$, and $c_{1}<\frac12$, the dissipative terms produced on the right-hand side are absorbed by the corresponding terms on the left. Multiplying by $2$ and rearranging, we arrive at
\begin{equation}\label{eq-error-gronwall}
d_{t}\|E^{n}_{\u}\|^{2}_{\L^{2}}+\frac{1}{\mu\rho}d_{t}\|E^{n}_{\B}\|^{2}_{\L^{2}}
+\kappa_{1}\|\nabla E^{n}_{\u}\|^{p}_{\L^{p}}+\kappa_{2}\|\nabla\times E^{n}_{\B}\|^{2}_{\L^{2}}
\le c\left\{\|E^{n}_{\u}\|^{2}_{\L^{2}}+\|E^{n-1}_{\u}\|^{2}_{\L^{2}}+\|E^{n-1}_{\B}\|^{2}_{\L^{2}}\right\}+\zeta_{n},
\end{equation}
with $\kappa_{1},\kappa_{2}>0$ and with the accumulated remainder
\begin{align}\label{eq-zeta_n}
\nonumber
\zeta_{n}&:=C\Delta t\left\{\|\partial_{tt}\u\|^{2}_{L^{2}(I_{n};\L^{2})}+\|\partial_{tt}\B\|^{2}_{L^{2}(I_{n};\L^{2})}+\|\partial_{t}\u\|^{2}_{L^{2}(I_{n};\L^{2})}+\|\partial_{t}\B\|^{2}_{L^{2}(I_{n};\L^{2})}\right\}\\
&\quad+\frac{C}{\Delta t}\left\{h^{2(l+1)}\|\partial_{t}\u\|^{2}_{L^{2}(I_{n};\mathbf{W}^{l+1,p})}+h^{2l}\|\partial_{t}\B\|^{2}_{L^{2}(I_{n};\mathbf{W}^{l,2})}\right\}
+Ch^{2l}+Ch^{\frac{lp}{p-1}} .
\end{align}
Every term of \eqref{eq-error-gronwall} on the right is quadratic in the errors; in particular no term of the form $\|E\|^{p'}_{\L^{2}}$ with $p'<2$ occurs. This is what makes the discrete Gronwall lemma applicable, and it is a direct consequence of having estimated the convective and coupling terms through Lemma \ref{lem-trilinear-L2} rather than through the $\left(p,p'\right)$--Young inequality.\\
\noindent
{Step 2: summation of the remainder.} We multiply \eqref{eq-zeta_n} by $\Delta t$ and sum over $n=1,\dots,m$. For the terms carrying an explicit factor $\Delta t$ we use
\[
\Delta t\sum_{n=1}^{m}\Delta t\,\|\varphi\|^{2}_{L^{2}(I_{n};\cdot)}\le(\Delta t)^{2}\|\varphi\|^{2}_{L^{2}(0,T;\cdot)},
\]
whereas for the projection consistency terms, which carry a factor $(\Delta t)^{-1}$, the same additivity of the $L^{2}$ norm in time gives
\[
\Delta t\sum_{n=1}^{m}\frac{1}{\Delta t}\|\varphi\|^{2}_{L^{2}(I_{n};\cdot)}\le\|\varphi\|^{2}_{L^{2}(0,T;\cdot)} .
\]
Finally $\Delta t\sum_{n=1}^{m}h^{\beta}\le Th^{\beta}$ for any $\beta>0$. Hence, using Assumption \ref{Assumption},
\[
\Delta t\sum_{n=1}^{m}\zeta_{n}\le C^{*}\left\{(\Delta t)^{2}+h^{2(l+1)}+h^{2l}+h^{\frac{lp}{p-1}}\right\}.
\]
Since $p\ge2$ we have $p'=\frac{p}{p-1}\in(1,2]$, so that $\frac{lp}{p-1}=lp'\le2l\le2(l+1)$ and, for $h\le1$, all remaining powers of $h$ are dominated by $h^{\frac{lp}{p-1}}$. Consequently
\begin{equation}\label{eq-zeta-summed}
\Delta t\sum_{n=1}^{m}\zeta_{n}\le C^{*}\left\{(\Delta t)^{2}+h^{\frac{lp}{p-1}}\right\}.
\end{equation}
\noindent
{Step 3: Gronwall.} Summing \eqref{eq-error-gronwall} from $n=1$ to $m$, multiplying by $\Delta t$ and using \eqref{eq-zeta-summed},
\[
\begin{aligned}
&\|E^{m}_{\u}\|^{2}_{\L^{2}}+\frac{1}{\mu\rho}\|E^{m}_{\B}\|^{2}_{\L^{2}}
+\Delta t\sum_{n=1}^{m}\left\{\kappa_{1}\|\nabla E^{n}_{\u}\|^{p}_{\L^{p}}+\kappa_{2}\|\nabla\times E^{n}_{\B}\|^{2}_{\L^{2}}\right\}\\
&\qquad\le c\,\Delta t\sum_{n=1}^{m}\left\{\|E^{n}_{\u}\|^{2}_{\L^{2}}+\|E^{n}_{\B}\|^{2}_{\L^{2}}\right\}+C^{*}\left\{(\Delta t)^{2}+h^{\frac{lp}{p-1}}\right\},
\end{aligned}
\]
where the initial contribution has been absorbed using the hypothesis on $(\mathbf{U}^{0},\mathfrak{B}^{0})$ and $2l\ge\frac{lp}{p-1}$. For $\Delta t\le\Delta t_{0}:=\frac{1}{2c}$ the discrete Gronwall lemma \cite[Lemma 4.9]{ding2022convergence} applies and yields
\[
\|E^{m}_{\u}\|^{2}_{\L^{2}}+\frac{1}{\mu\rho}\|E^{m}_{\B}\|^{2}_{\L^{2}}
+\Delta t\sum_{n=1}^{m}\left\{\kappa_{1}\|\nabla E^{n}_{\u}\|^{p}_{\L^{p}}+\kappa_{2}\|\nabla\times E^{n}_{\B}\|^{2}_{\L^{2}}\right\}
\le C^{*}\left\{(\Delta t)^{2}+h^{\frac{lp}{p-1}}\right\}.
\]
\noindent
{Step 4: from the discrete error to the total error.} The triangle inequality gives $\u^{n}-\mathbf{U}^{n}=\eta^{n}_{\u}-E^{n}_{\u}$, and by \eqref{eq-error approximation}, $\|\eta^{n}_{\u}\|^{2}_{\L^{2}}\le C h^{2(l+1)}$ and $\Delta t\sum_{n}\|\nabla\eta^{n}_{\u}\|^{p}_{\L^{p}}\le CTh^{lp}$. Since $lp\ge\frac{lp}{p-1}$ and $2(l+1)\ge\frac{lp}{p-1}$ for $p\ge2$, both projection contributions are dominated by $h^{\frac{lp}{p-1}}$, which gives \eqref{eq-main-u}. The same argument with \eqref{eq-error approximation-B} gives \eqref{eq-main-B}.
\end{proof}

\subsection{Discussion of the convergence rates}\label{subsec-err-discussion}

The estimate of Theorem \ref{Main theorem of error} is deliberately asymmetric: the temporal exponent is the classical $2$, whereas the spatial exponent involves the conjugate exponent $p'$. The following remarks explain why, and delimit what can and cannot be improved.

\begin{remark}\label{rem-exponents}
The right-hand side of \eqref{eq-main-u}--\eqref{eq-main-B} aggregates several powers of $\Delta t$ and $h$, and since $\Delta t,h\to0$ the binding contribution is the one carrying the smallest exponent.

All temporal contributions come from terms handled by Lemma \ref{lem-trilinear-L2} or the quadratic magnetic dissipation, hence by $(2,2)$--Young: each $(\Delta t)^{1/2}$ is squared to $\Delta t$, and the outer summation supplies a second factor, giving $(\Delta t)^{2}$. This order is optimal, since the backward Euler quotient and the lagged convection and coupling are all first-order consistent.

The spatial contributions are $h^{2(l+1)}$ and $h^{2l}$ from the projection remainders, and $h^{lp'}$ from the $p$-Laplace defect of Lemma \ref{lem-plaplace-defect}. Since $p'\le2$,
\[
lp'\le2l\le2(l+1),
\]
so $h^{lp'}=h^{\frac{lp}{p-1}}$ dominates. For $p=2$ one has $p'=2$ and Theorem \ref{Main theorem of error} reduces to the classical bound $C^{*}\{(\Delta t)^{2}+h^{2l}\}$, in agreement with the known results for Newtonian MHD systems \cite{ding2022convergence,Prohl2008}.
\end{remark}

\begin{remark}[The natural energy norm, and the quasi-norm benchmark]\label{rem-quasinorm}
Extracting the $p$-th root in \eqref{eq-main-u} and using $p'/p=\frac{1}{p-1}$ gives
\[
\left((\Delta t)\sum_{n=1}^{m}\left\|\nabla(\u^{n}-\mathbf{U}^{n})\right\|^{p}_{\L^{p}}\right)^{1/p}
\le C\left((\Delta t)^{\frac{2}{p}}+h^{\frac{l}{p-1}}\right).
\]
For $p=2$ this is the optimal rate $O(\Delta t+h^{l})$, whereas for $p>2$ the spatial rate in the natural energy norm degrades by the factor $p-1$ relative to the approximation-theoretic optimum $O(h^{l})$. This loss is not a defect of the argument above but the well-documented sub-optimality of $\mathbf{L}^{p}$-based error norms for the $p$-Laplacian. It is removed by measuring the error in the quasi-norm of Barrett and Liu \cite{BarrettLiu1994}, or equivalently in the natural distance $F(\nabla\u)=|\nabla\u|^{\frac{p-2}{2}}\nabla\u$ of \cite{DieningEbmeyerRuzicka2007,BerselliRuzicka2022}, for which the optimal order $\|F(\nabla\u)-F(\nabla\mathbf{U})\|^{2}_{\L^{2}}=O(h^{2l})$ is attained. The quasi-norm is also the natural framework in which the dissipation becomes quadratic with respect to a weighted measure, which as explained in Remark \ref{rem-divfree} is exactly what would be needed to dispense with Assumption \ref{Assumption-divfree}. We do not pursue this here.
\end{remark}
}

\section{Numerical Experiments and Verification of the Theoretical Results}
\label{sec-numerics}
In this section we report a series of numerical experiments that illustrate the qualitative behaviour of the scheme and verify the theoretical results of Sections \ref{sec-wellposed}--\ref{ lemma-error bounds}. All computations are carried out on the unit cube $\mathbb{D}=(0,1)^{3}$ with a uniform tetrahedral mesh, using continuous $\mathbb{P}_{2}$ elements for the velocity, continuous $\mathbb{P}_{1}$ elements for the pressure, and second-order N\'ed\'elec edge elements for the magnetic field. This choice is inf-sup stable for the velocity--pressure pair and $\mathbf{H}(\operatorname{\mathbf{curl}})$-conforming for the magnetic field, in agreement with the hypotheses of Sections \ref{sec-fem}--\ref{ lemma-error bounds}. In each experiment the source terms $\mathbf{g},\mathbf{f}$ and the initial data are prescribed so that a chosen smooth field solves the continuous $p$-MHD system exactly, and the discrete errors are measured against that field.

The four experiments are complementary. Example \ref{ex-temporal} illustrates the qualitative dynamics of the coupled velocity and magnetic fields; Example \ref{ex-stability} verifies the unconditional stability of Lemma \ref{bounds}; Example \ref{ex-time-rate} isolates and measures the relative temporal convergence rate; and Example \ref{ex-spacetime} confirms the full space-time convergence predicted by Theorem \ref{Main theorem of error}.

\begin{Example}[Temporal Profiles of the Velocity and Magnetic Fields]\label{ex-temporal}
This first experiment illustrates the temporal dynamics captured by the scheme. We track the computed velocity and magnetic fields on the plane $z=0.25$ over $t\in[0,1]$, for the parameters
\[
\nu=1,\qquad \sigma=1,\qquad \mu=1,
\]
with $p=3$, so that the fluid is shear-thickening. The exact velocity is prescribed as the curl of a scalar potential, which makes it pointwise divergence-free,
\[
\mathbf{u}=\nabla\times\left(0,0,\psi\right),
\qquad
\psi(x,y,z,t)=\sin^{2}(\pi x)\sin^{2}(\pi y)\sin^{2}(\pi z)\sin(\omega t+\pi x),
\]
that is,
\[
\mathbf{u}(x,y,z,t)=
\begin{pmatrix}
\pi \sin^{2}(\pi x)\,
\sin(2\pi y)\,
\sin^{2}(\pi z)\,
\sin(\omega t+\pi x)
\\[1.2ex]
\begin{aligned}
&-\pi \sin(2\pi x)\,\sin^{2}(\pi y)\,\sin^{2}(\pi z)\,\sin(\omega t+\pi x)\\
&\qquad-\pi \sin^{2}(\pi x)\,\sin^{2}(\pi y)\,\sin^{2}(\pi z)\,\cos(\omega t+\pi x)
\end{aligned}
\\[1.2ex]
0
\end{pmatrix},
\]
with $\omega=7$; the double zeros of the $\sin^{2}$ factors make $\u$ vanish on all six faces of $\mathbb{D}$. (The two terms of the second component come from the product rule applied to $-\partial_{x}\psi$ and form a single component.) The magnetic field is likewise taken as a curl, hence divergence-free,
$$
\mathbf{B}(x,y,z,t)=
\begin{pmatrix}
\pi \sin^{2}(\pi x)\sin(2\pi y)\sin^{2}(\pi z)
\sin\!\left(2\pi\bigl(x\cos(10t)+y\sin(10t)\bigr)\right)
\\
\qquad
+2\pi\sin(10t)\sin^{2}(\pi x)\sin^{2}(\pi y)\sin^{2}(\pi z)
\cos\!\left(2\pi\bigl(x\cos(10t)+y\sin(10t)\bigr)\right)
\\[1.5ex]
-\pi\sin(2\pi x)\sin^{2}(\pi y)\sin^{2}(\pi z)
\sin\!\left(2\pi\bigl(x\cos(10t)+y\sin(10t)\bigr)\right)
\\
\qquad
-2\pi\cos(10t)\sin^{2}(\pi x)\sin^{2}(\pi y)\sin^{2}(\pi z)
\cos\!\left(2\pi\bigl(x\cos(10t)+y\sin(10t)\bigr)\right)
\\[1.5ex]
0
\end{pmatrix}.
$$
The corresponding source terms are chosen so that the above functions satisfy the continuous $p$-MHD system exactly. Figures~\ref{fig:solutions_different_times} and \ref{fig:B_solutions_different_times} show the evolution of the computed velocity and magnetic field components, respectively, at the fixed $z=0.25$ over the time interval $[0,1]$. The computed fields reproduce the expected qualitative structure of the exact solution and its temporal evolution, illustrating the behaviour of the proposed finite element scheme for the coupled $p$-MHD system; a quantitative assessment of the convergence rates is provided in the convergence experiments below.
\begin{figure}[!htbp]
    \centering

    \begin{subfigure}[b]{0.3\textwidth}
        \centering
        \includegraphics[width=\textwidth]{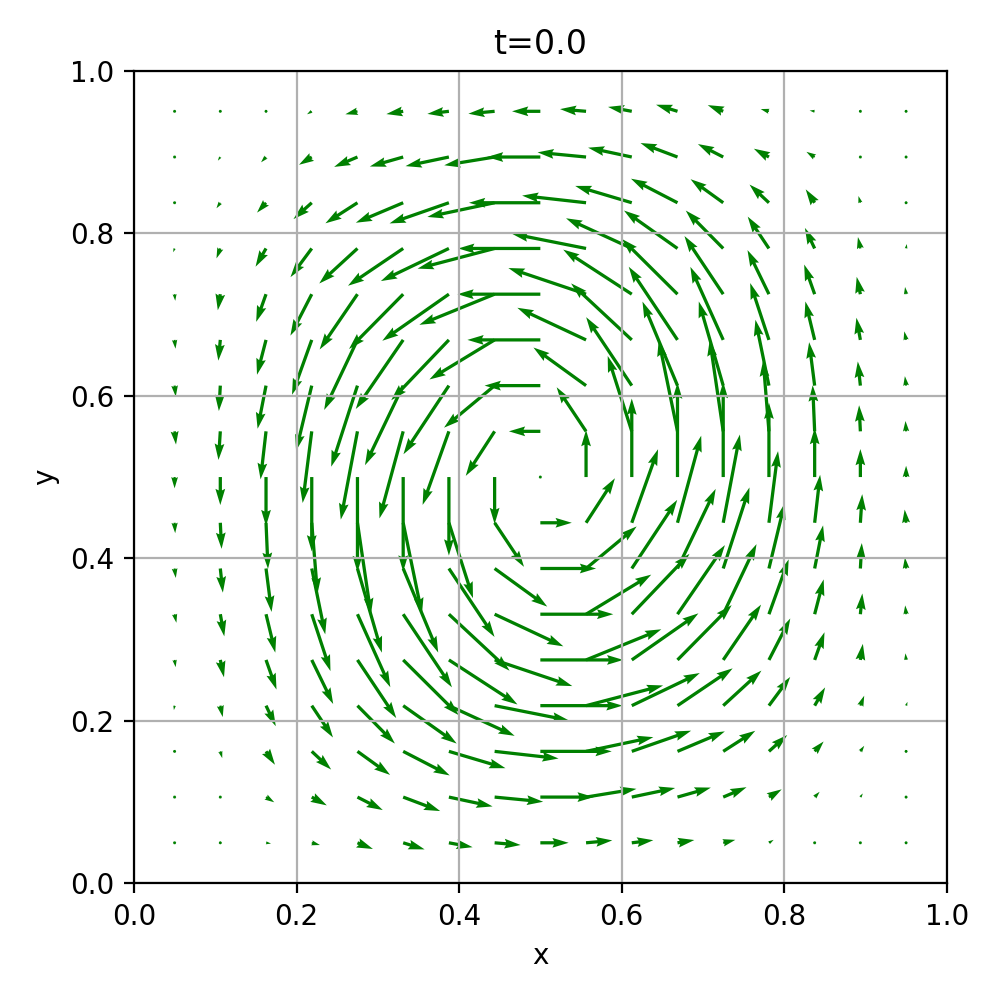}
        \caption{$t=0.0$}
        \label{fig:tu00}
    \end{subfigure}
    \hfill
    \begin{subfigure}[b]{0.3\textwidth}
        \centering
        \includegraphics[width=\textwidth]{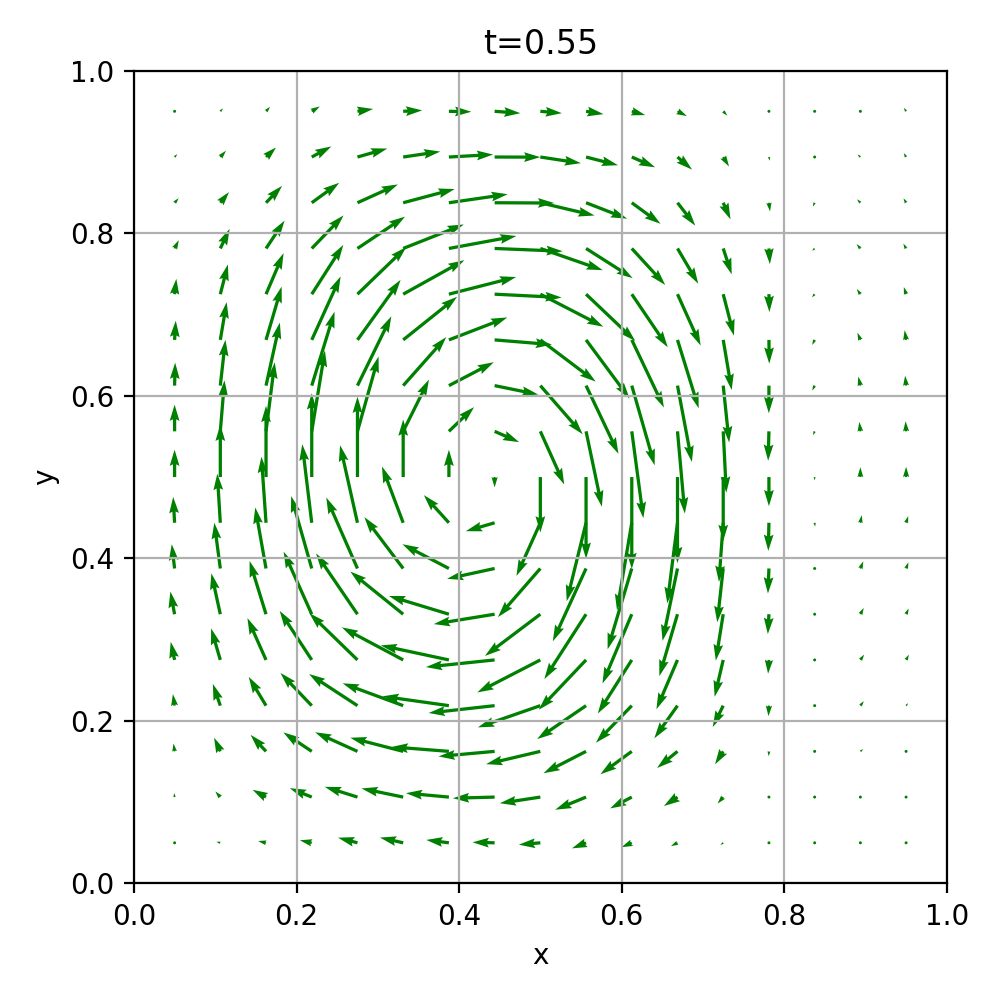}
        \caption{$t=0.55$}
        \label{fig:t055}
    \end{subfigure}
    \hfill
    \begin{subfigure}[b]{0.3\textwidth}
        \centering
        \includegraphics[width=\textwidth]{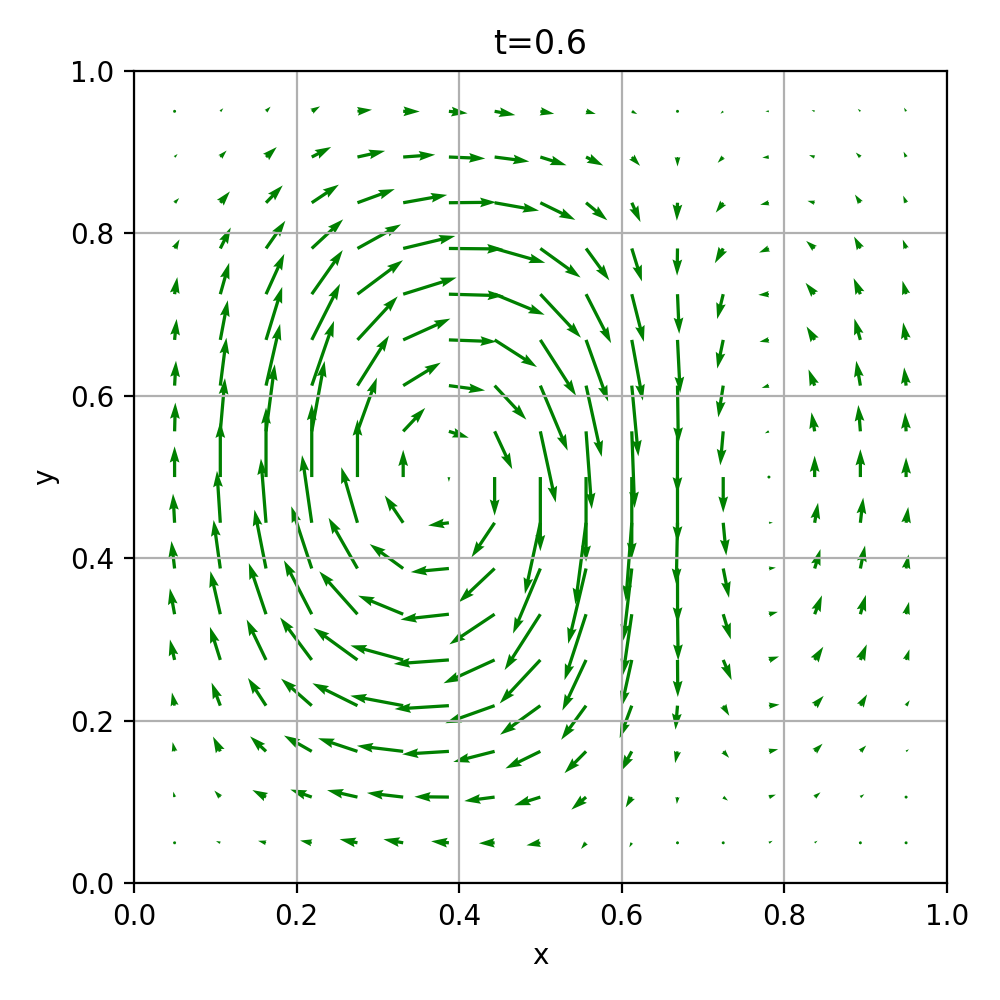}
        \caption{$t=0.6$}
        \label{fig:t06}
    \end{subfigure}

    \vspace{0.4cm}

    \begin{subfigure}[b]{0.3\textwidth}
        \centering
        \includegraphics[width=\textwidth]{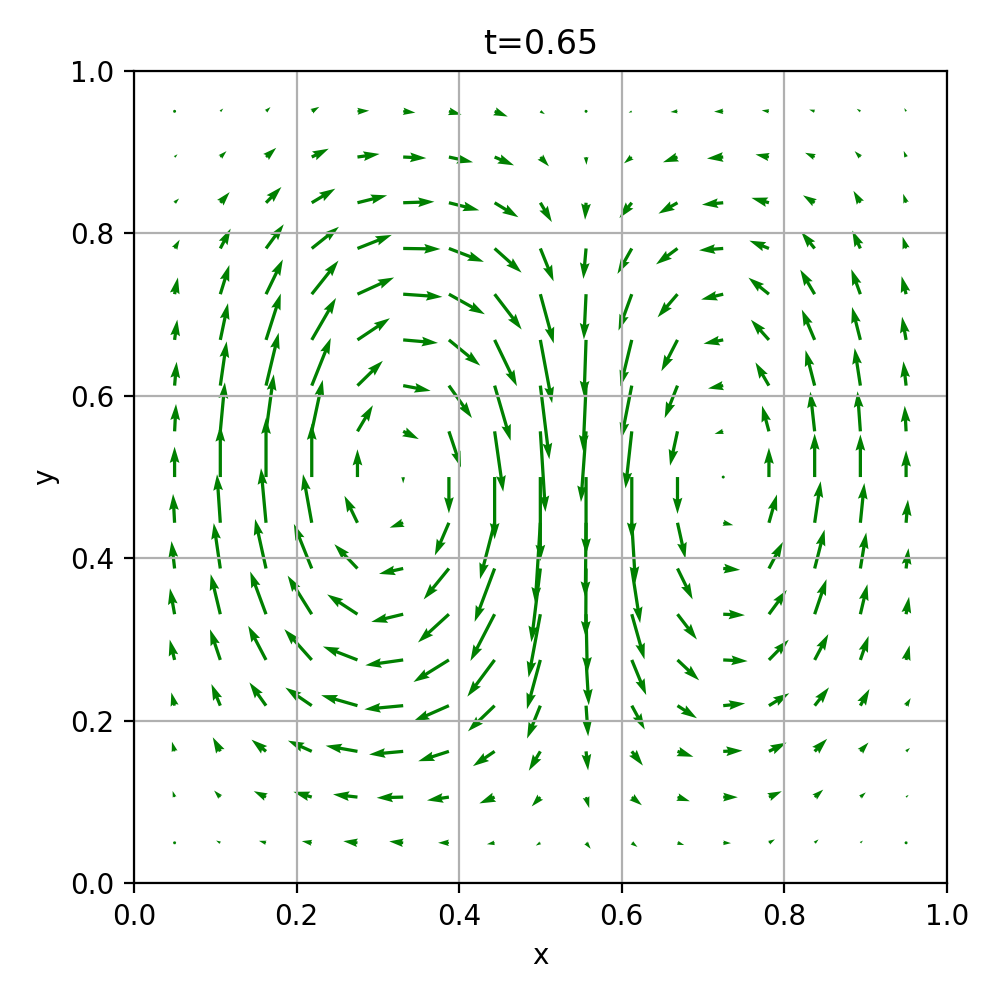}
        \caption{$t=0.65$}
        \label{fig:t065}
    \end{subfigure}
    \hfill
    \begin{subfigure}[b]{0.3\textwidth}
        \centering
        \includegraphics[width=\textwidth]{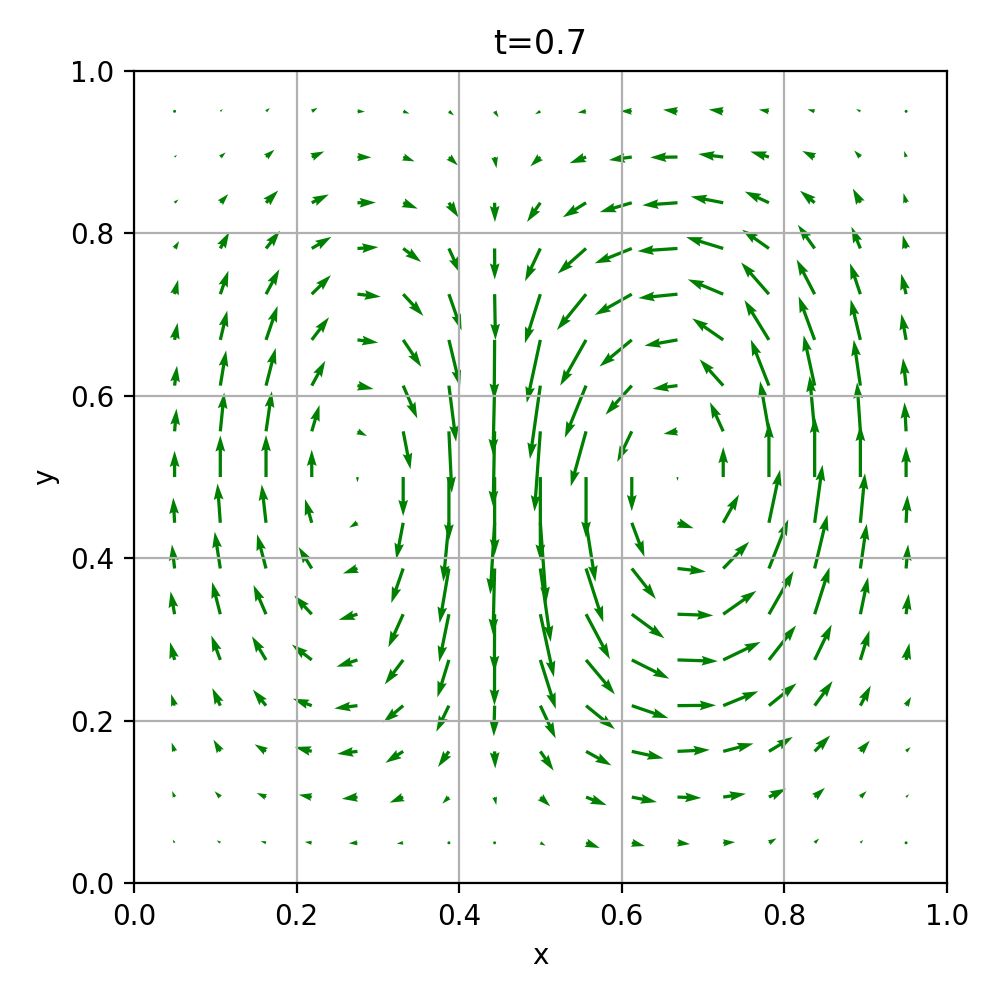}
        \caption{$t=0.7$}
        \label{fig:t07}
    \end{subfigure}
    \hfill
    \begin{subfigure}[b]{0.3\textwidth}
        \centering
        \includegraphics[width=\textwidth]{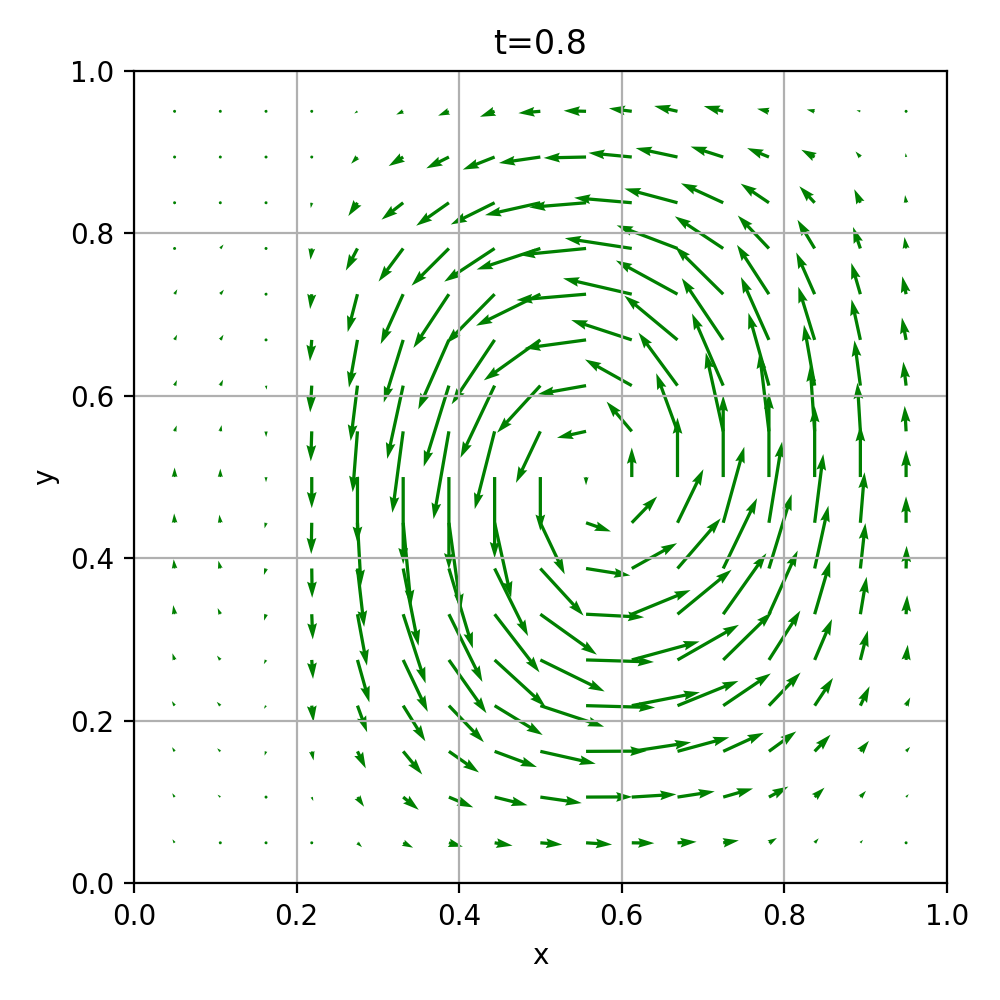}
        \caption{$t=0.8$}
        \label{fig:t08}
    \end{subfigure}
  \vspace{0.4cm}

    \begin{subfigure}[b]{0.3\textwidth}
        \centering
        \includegraphics[width=\textwidth]{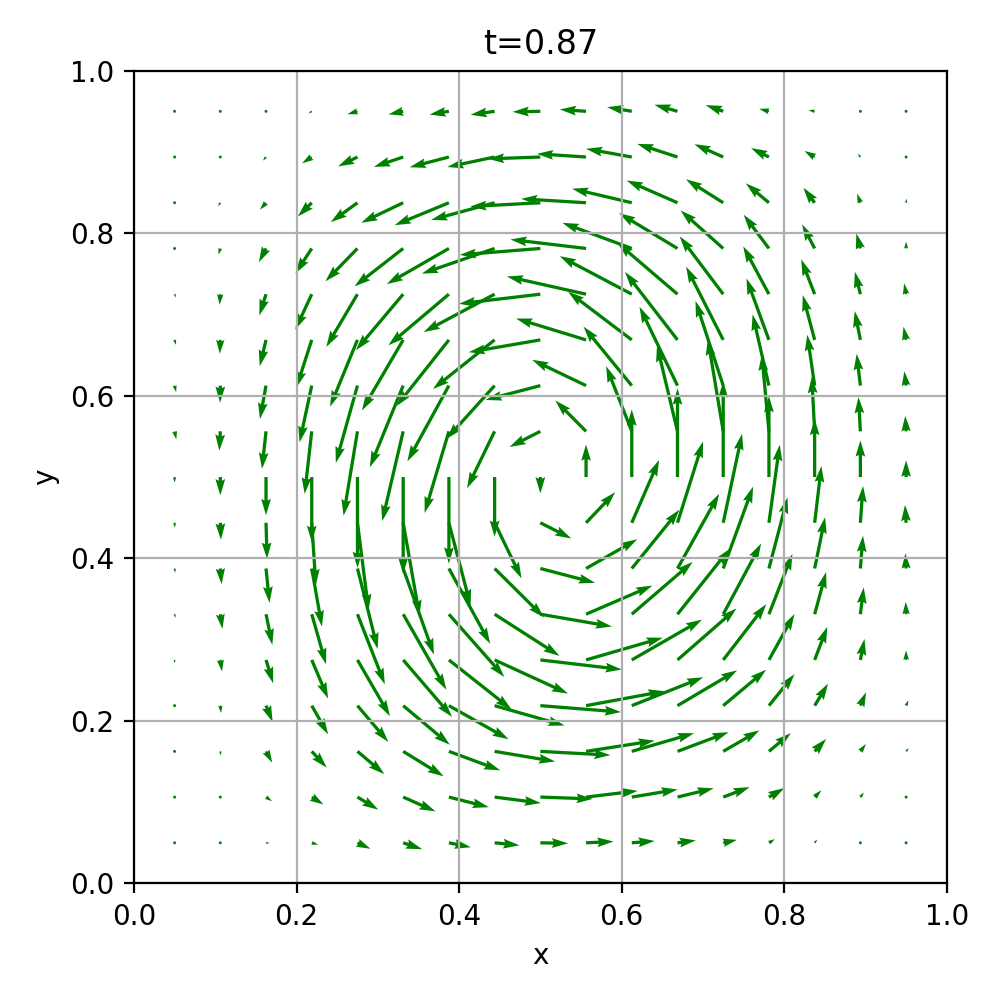}
        \caption{$t=0.87$}
        \label{fig:t087}
    \end{subfigure}
    \hfill
      \begin{subfigure}[b]{0.3\textwidth}
        \centering
        \includegraphics[width=\textwidth]{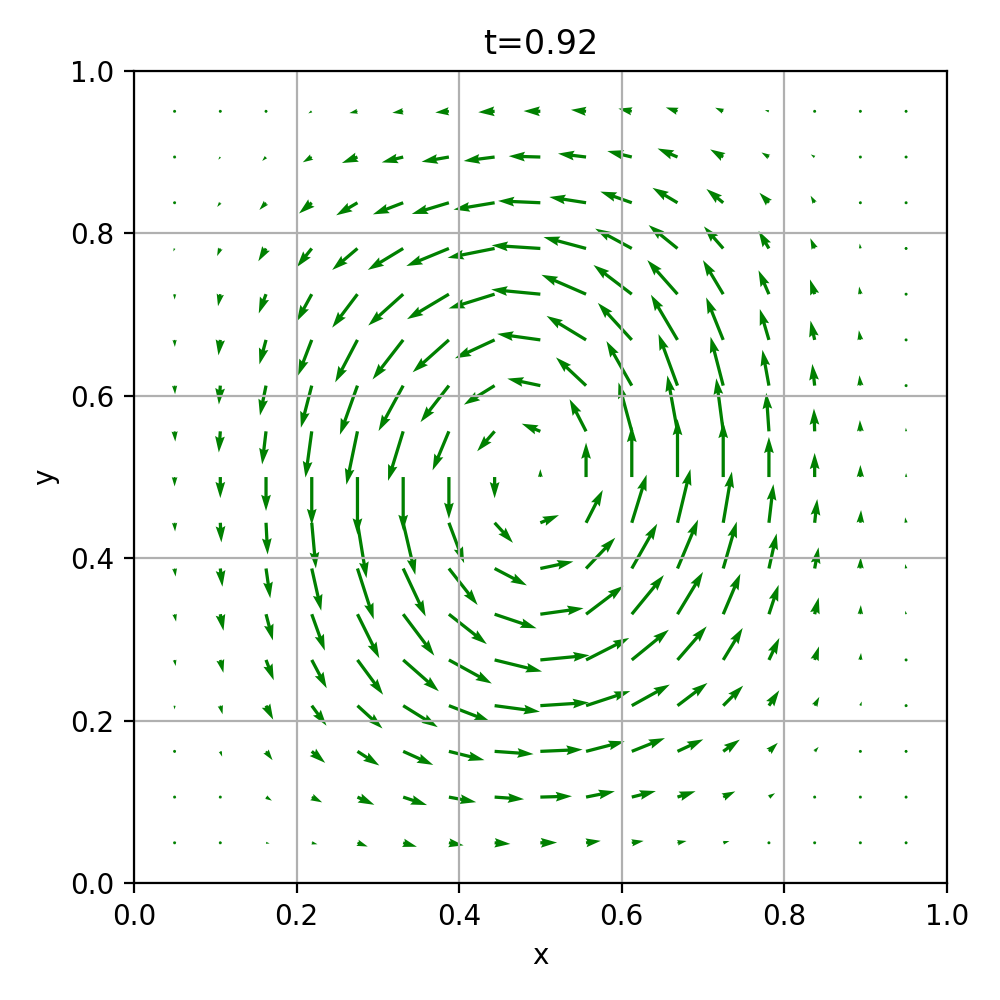}
        \caption{$t=0.92$}
        \label{fig:t092}
    \end{subfigure}
    \hfill
    \begin{subfigure}[b]{0.3\textwidth}
        \centering
        \includegraphics[width=\textwidth]{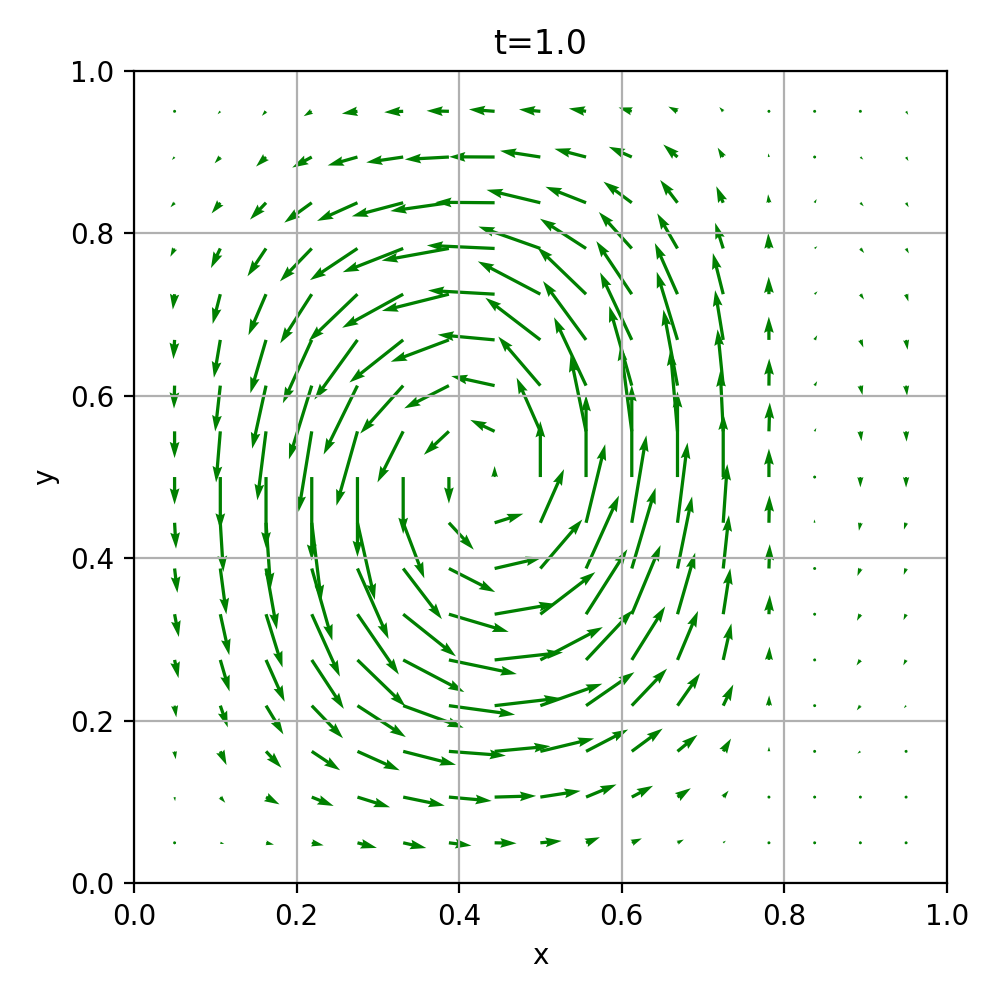}
        \caption{$t=1.0$}
        \label{fig:t10}
    \end{subfigure}
    \caption{Temporal evolution of the computed velocity field (shown in green) at the plane $z=0.25$ for different time instances.}
    \label{fig:solutions_different_times}
\end{figure}

\begin{figure}[!htbp]
    \centering

    \begin{subfigure}[b]{0.3\textwidth}
        \centering
        \includegraphics[width=\textwidth]{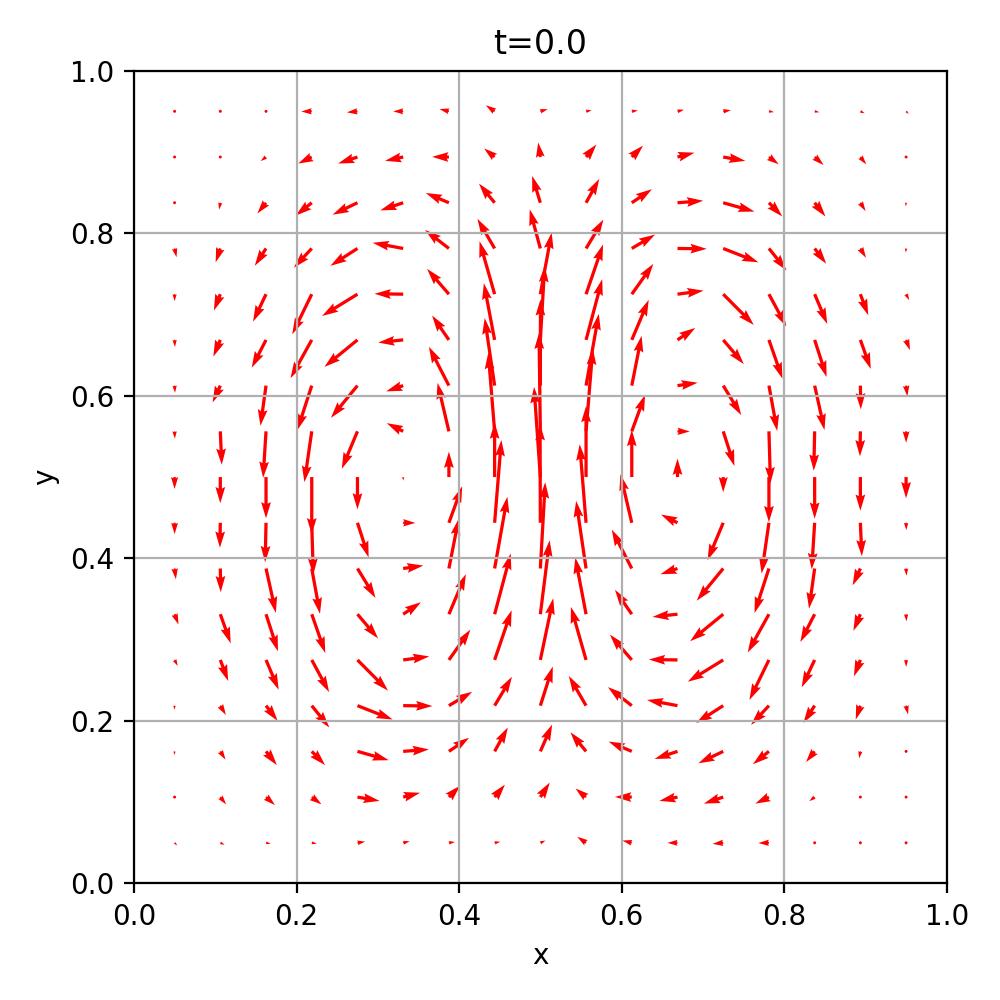}
        \caption{$t=0.0$}
        \label{fig:Bt00}
    \end{subfigure}
    \hfill
    \begin{subfigure}[b]{0.3\textwidth}
        \centering
        \includegraphics[width=\textwidth]{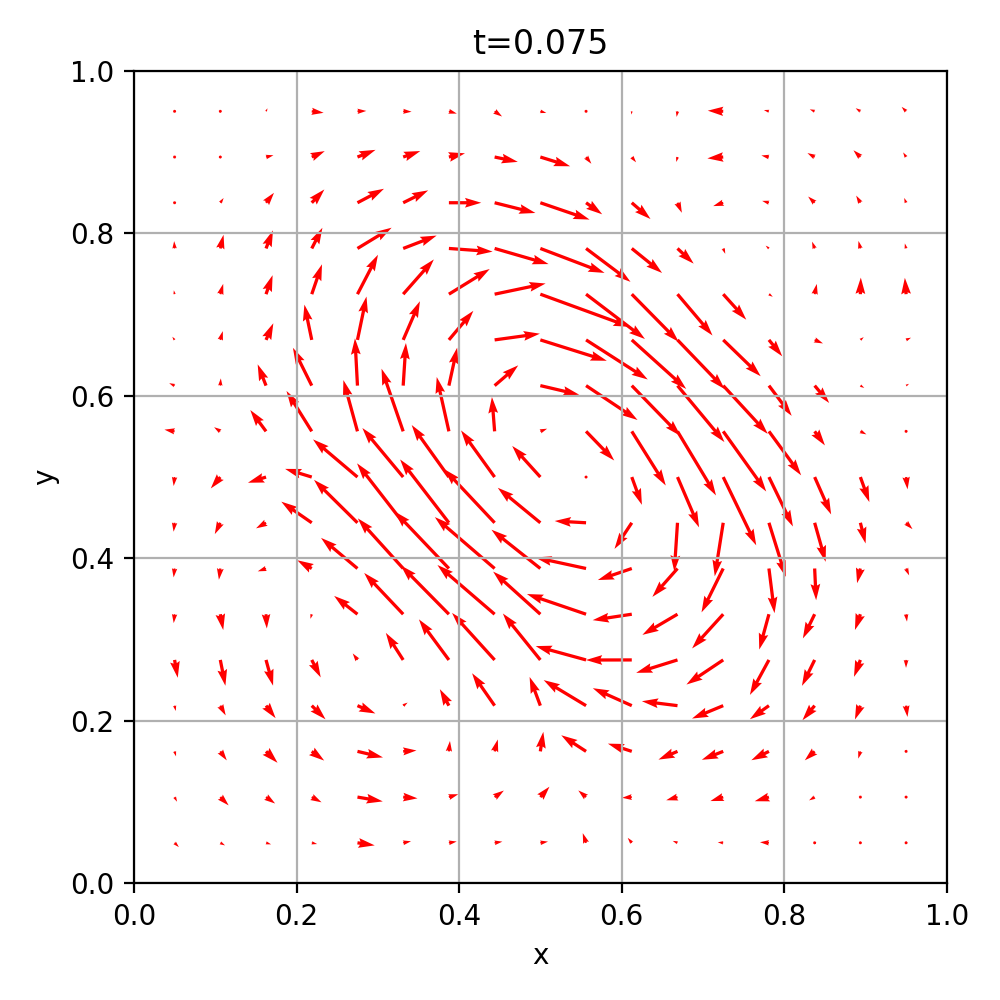}
        \caption{$t=0.075$}
        \label{fig:Bt0075}
    \end{subfigure}
    \hfill
    \begin{subfigure}[b]{0.3\textwidth}
        \centering
        \includegraphics[width=\textwidth]{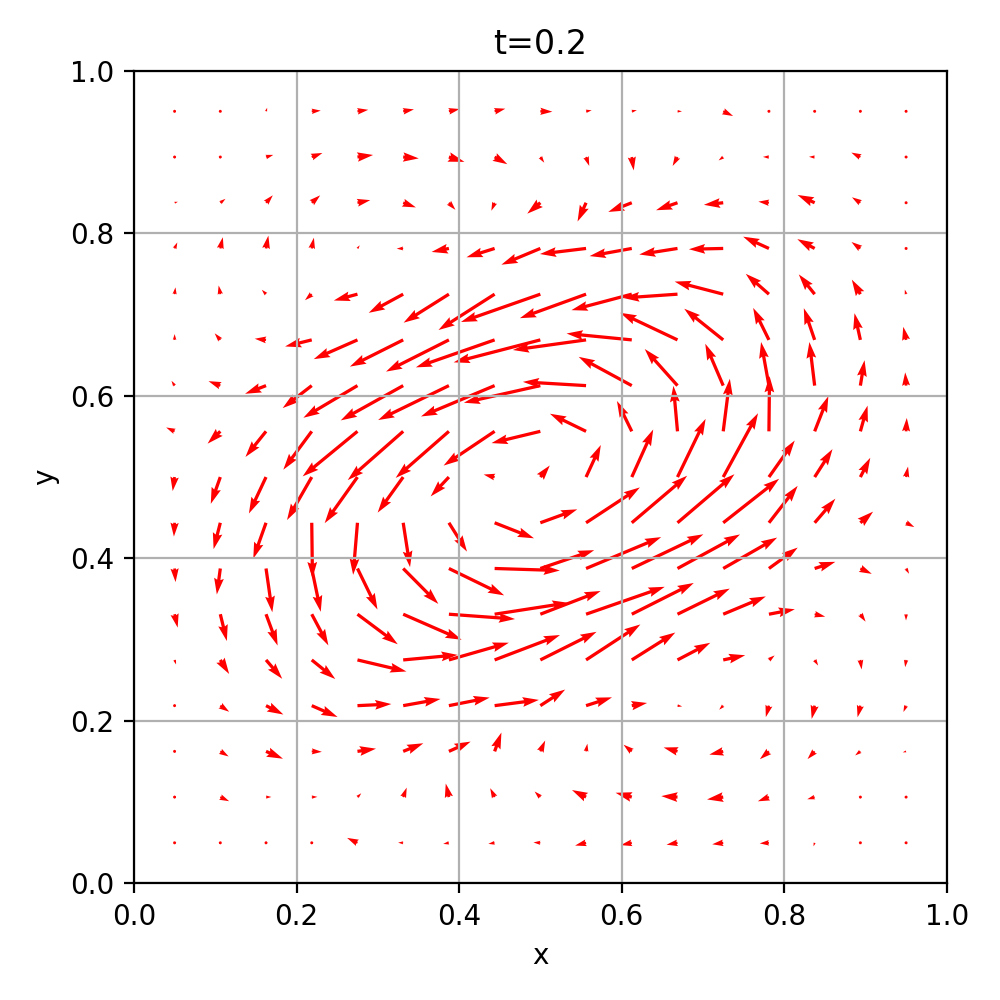}
        \caption{$t=0.2$}
        \label{fig:Bt02}
    \end{subfigure}

    \vspace{0.4cm}

    \begin{subfigure}[b]{0.3\textwidth}
        \centering
        \includegraphics[width=\textwidth]{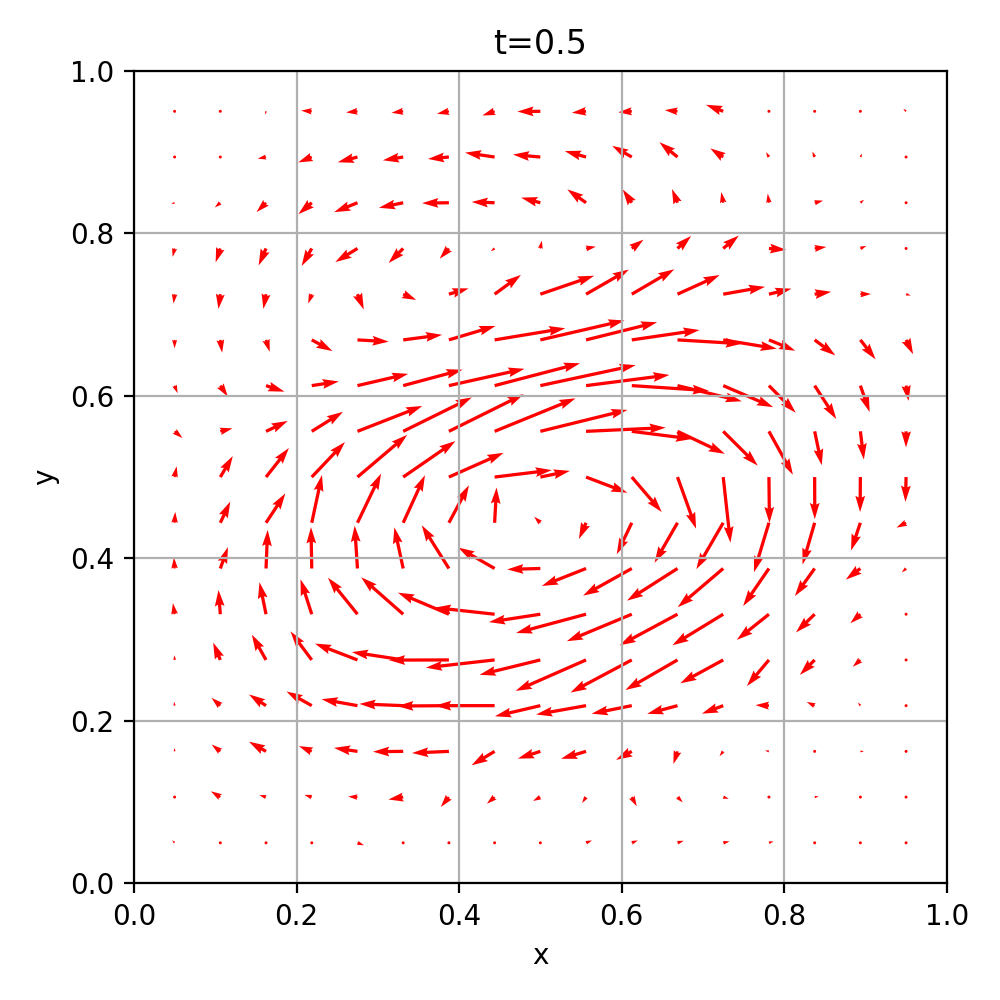}
        \caption{$t=0.5$}
        \label{fig:Bt05}
    \end{subfigure}
    \hfill
    \begin{subfigure}[b]{0.3\textwidth}
        \centering
        \includegraphics[width=\textwidth]{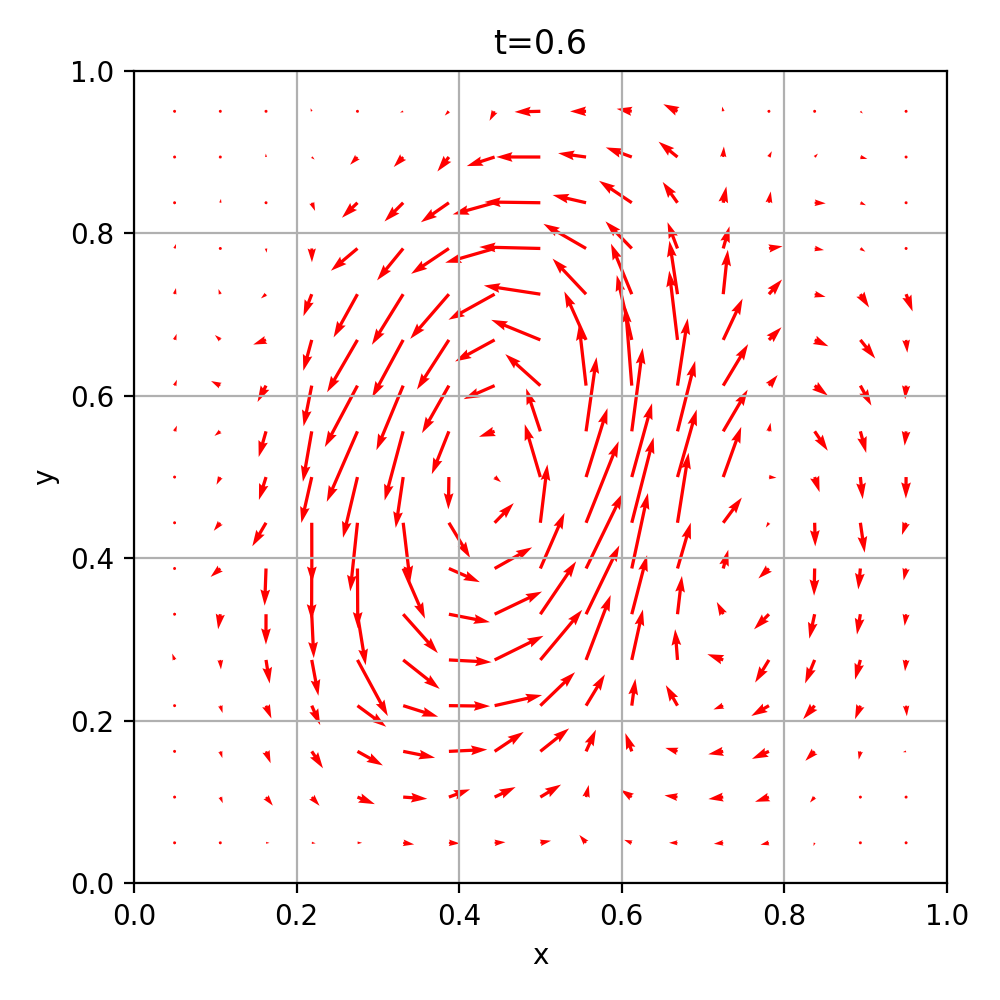}
        \caption{$t=0.6$}
        \label{fig:Bt06}
    \end{subfigure}
    \hfill
    \begin{subfigure}[b]{0.3\textwidth}
        \centering
        \includegraphics[width=\textwidth]{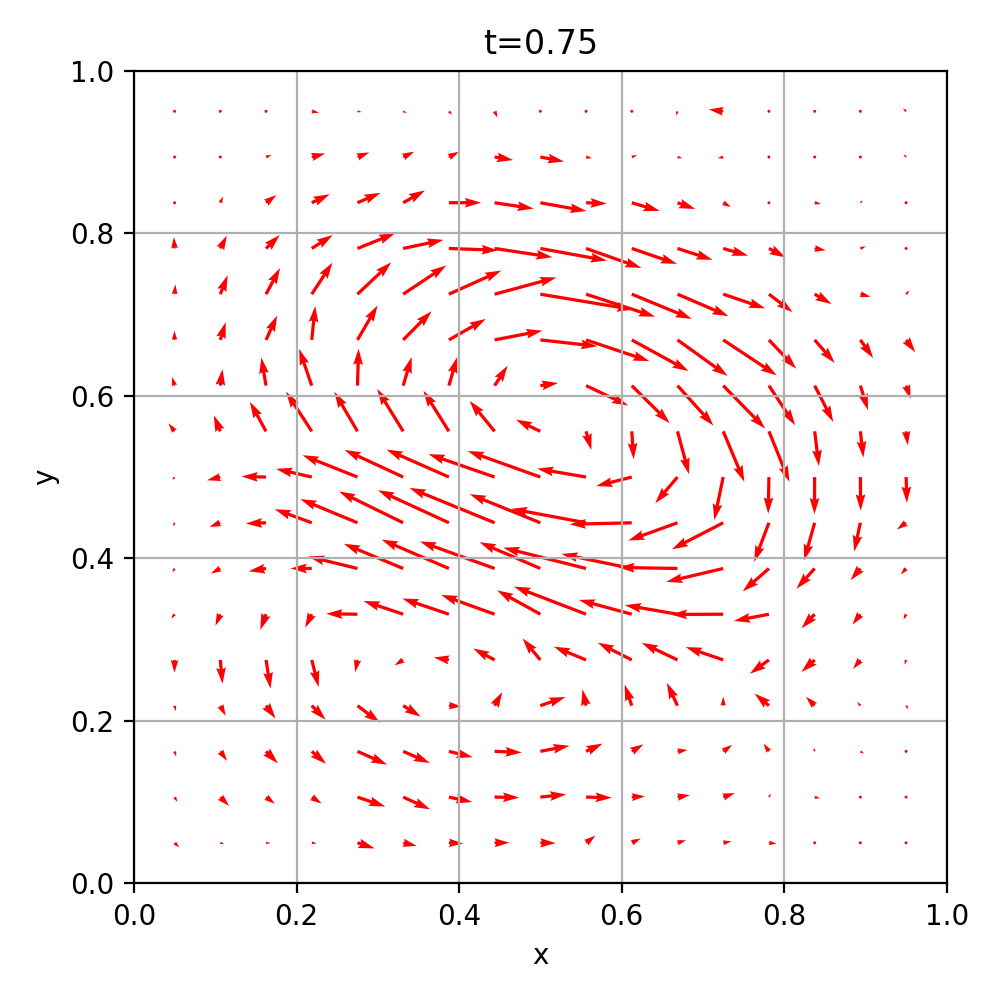}
        \caption{$t=0.75$}
        \label{fig:Bt075}
    \end{subfigure}
    \begin{subfigure}[b]{0.3\textwidth}
        \centering
        \includegraphics[width=\textwidth]{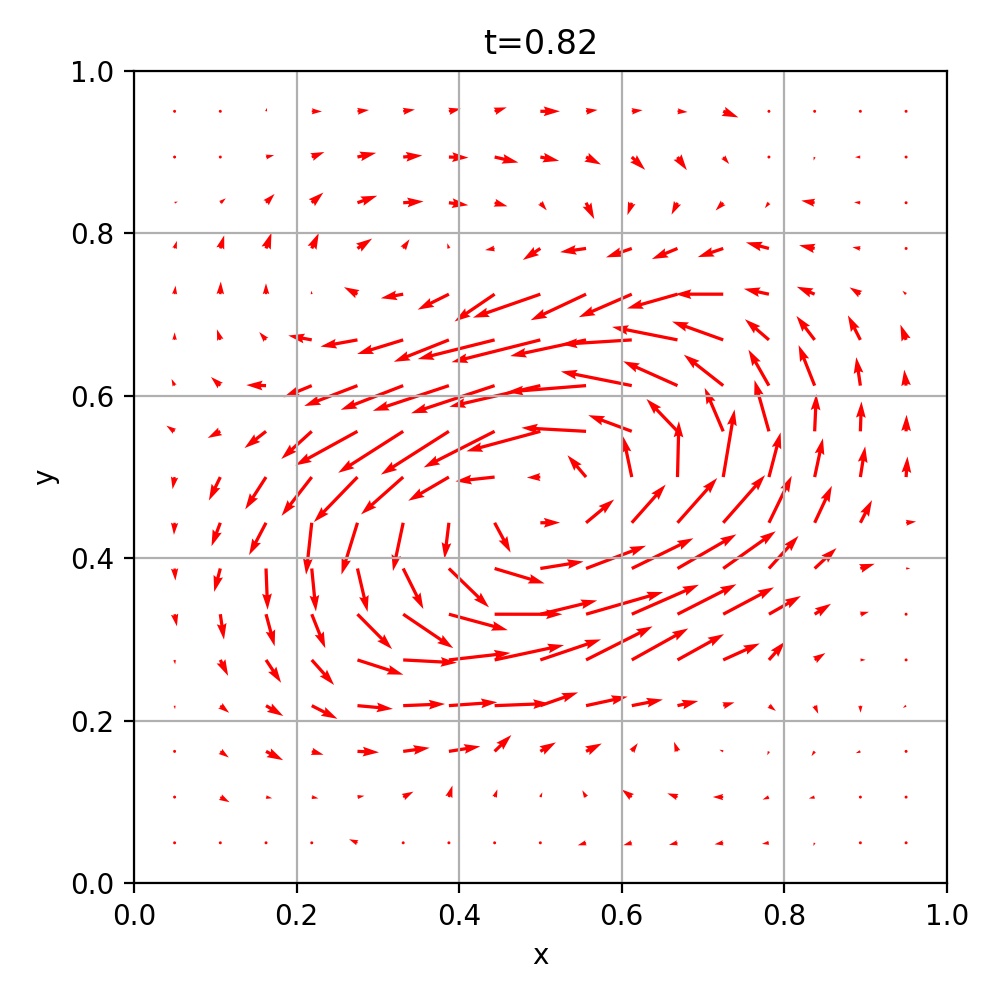}
        \caption{$t=0.82$}
        \label{fig:Bt082}
    \end{subfigure}
    \hfill
    \begin{subfigure}[b]{0.3\textwidth}
        \centering
        \includegraphics[width=\textwidth]{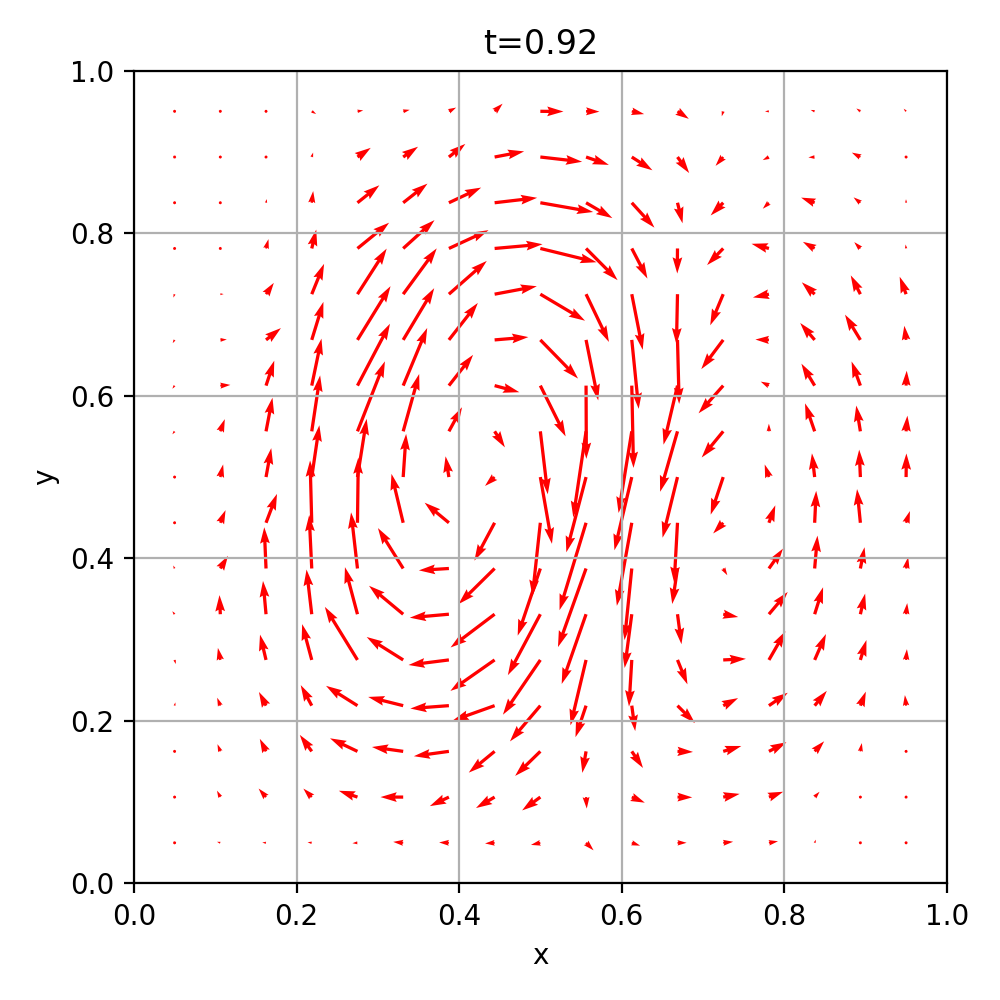}
        \caption{$t=0.92$}
        \label{fig:Bt092}
    \end{subfigure}
    \hfill
    \begin{subfigure}[b]{0.3\textwidth}
        \centering
        \includegraphics[width=\textwidth]{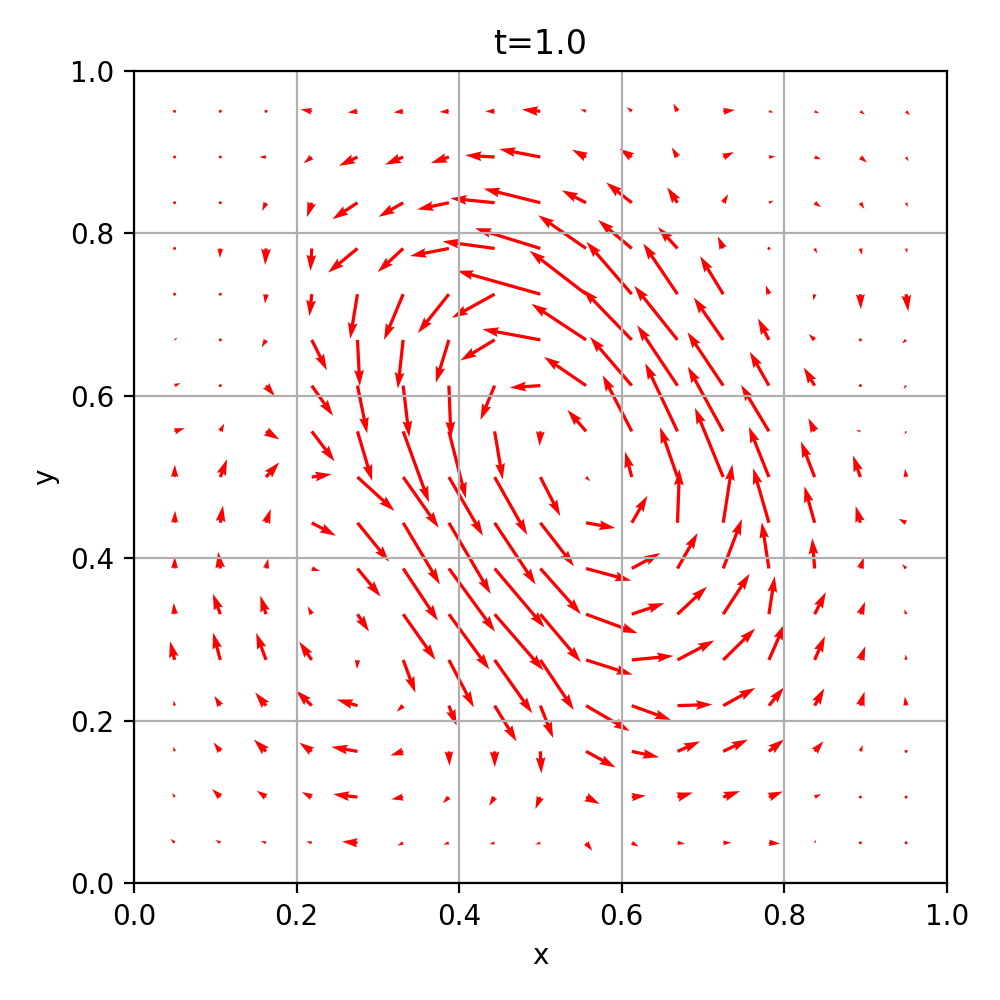}
        \caption{$t=1.0$}
        \label{fig:Bt10}
    \end{subfigure}

    \caption{Temporal evolution of the computed magnetic field (shown in red) at the plane $z=0.25$ for different time instances.}
    \label{fig:B_solutions_different_times}
\end{figure}
\end{Example}

\medskip

\begin{Example}[Verification of the Stability Estimates]\label{ex-stability}
This experiment verifies the unconditional stability of Lemma \ref{bounds} in a strongly nonlinear regime. We take the parameters
\[
\nu=1,\qquad \sigma=1,\qquad \mu=1,
\]
on the time interval $[0,1]$. The exact solution is chosen as
\[
\u=\nabla\times\left(0,0,\psi\right),
\qquad
\psi(x,y,z,t)=g(x)\,g(y)\,g(z)\,e^{-t},
\qquad
g(s)=s^{2}(s-1)^{2},
\]
that is,
\[
\u=\Bigl(2x^{2}(x-1)^{2}y(y-1)(2y-1)z^{2}(z-1)^{2}e^{-t},\
-2x(x-1)(2x-1)y^{2}(y-1)^{2}z^{2}(z-1)^{2}e^{-t},\ 0\Bigr),
\]
which is divergence-free and vanishes on $\partial\mathbb{D}$, with pressure
$$p=(2x-1)(2y-1)(1+t)$$
and magnetic field
\[
\B=
\left(
\sin(t)\sin(\pi x)\cos(\pi y),
\,
-\sin(t)\cos(\pi x)\sin(\pi y),
\,
0
\right).
\]
The corresponding source terms are chosen so that the above functions satisfy the continuous $p$-MHD system exactly.

We fix $h=0.433$, integrate up to $T=1$, and take the strongly shear-thickening exponent $p=3$. We monitor the individual norms that enter the discrete energy
which appears in the stability estimate of Lemma~\ref{bounds}.

Figure~\ref{fig:Stability norm vs Time} shows their evolution. All norms remain bounded and in fact decay monotonically, in agreement with the estimate \eqref{eq-bounds-1}: the decay reflects the factor $e^{-t}$ carried by the prescribed solution, while \eqref{eq-bounds-1} guarantees only boundedness and does not preclude it. The time step is not tied to the mesh size, so the experiment is consistent with the unconditional stability of Lemma \ref{bounds}, even at $p=3$.

\begin{figure}[!htbp]
    \centering
    \includegraphics[width=0.5\textwidth]{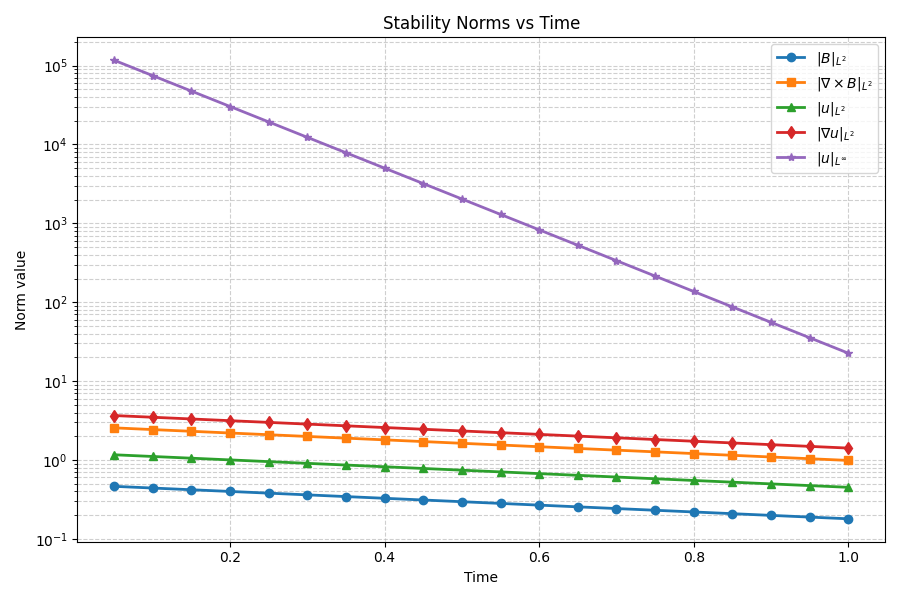}
    
    \caption{Evolution of the norms entering the discrete energy of Lemma~\ref{bounds} as functions of time, for $p=10$. All quantities remain bounded and decay monotonically, consistently with the unconditional stability estimate \eqref{eq-bounds-1}.}
    \label{fig:Stability norm vs Time}
\end{figure}

\end{Example}

\medskip

\begin{Example}[Verification of the Relative Temporal Convergence Rate]\label{ex-time-rate}
This experiment isolates the temporal error. Keeping the spatial mesh fixed and fine, we refine only the time step $\Delta t$ and measure the resulting convergence rate. We take the parameters
\[
\nu=1,\qquad \sigma=1,\qquad \mu=1,
\]
on the time interval $[0,T]$ with final time $T=1.6$. The exact velocity, pressure and magnetic field are
{
\[
\mathbf{u}
=
2e^{-t}
\begin{pmatrix}
\sin^{2}(\pi x)\,\sin(\pi y)\cos(\pi y)\,\sin^{2}(\pi z)
\\[1.2ex]
-\sin(\pi x)\cos(\pi x)\,\sin^{2}(\pi y)\,\sin^{2}(\pi z)
\\[1.2ex]
0
\end{pmatrix}, 
\]
with pressure
\[
p=e^{-t}\sin(\pi x)\sin(\pi y)\sin(\pi z),
\]
and magnetic field
\[
\B=
\left(
e^{-t}\sin(\pi x)\cos(\pi y)\cos(\pi z),
\,
-e^{-t}\cos(\pi x)\sin(\pi y)\cos(\pi z),
\,0
\right).
\]}

with source terms chosen so that these fields solve the continuous $p$-MHD system exactly. To expose the temporal error alone, the spatial mesh is fixed at the fine value $h=\sqrt{3}/16$, and $\Delta t$ is then successively halved. Table~\ref{tab:relative_time} reports the resulting relative errors at $T=1.6$ and the observed rates for the magnetic field, velocity and pressure. The rates approach one as $\Delta t\to0$, confirming the first-order temporal accuracy of the backward Euler discretisation predicted by Theorem~\ref{Main theorem of error}.

\begin{table}[!htbp]
\centering
\renewcommand{\arraystretch}{1.5}
\caption{Example~\ref{ex-time-rate}: relative errors at $T=1.6$ and observed temporal convergence rates under refinement of $\Delta t$ with the spatial mesh fixed at $h=\sqrt{3}/16$.}
\label{tab:relative_time}
\resizebox{\textwidth}{!}{
\begin{tabular}{|c|cc|cc|cc|cc|cc|}
\hline
$\Delta t$
& \multicolumn{2}{c|}{$\|\B-\B_h\|_{\L^2}/\|\B\|_{\L^2}$}
& \multicolumn{2}{c|}{$\|\nabla\times(\B-\B_h)\|_{\L^2}/\|\nabla\times \B\|_{\L^2}$}
& \multicolumn{2}{c|}{$\|\u-\u_h\|_{\L^2}/\|\u\|_{\L^2}$}
& \multicolumn{2}{c|}{$\|\u-\u_h\|_{\mathbf   H^1}/\|\u\|_{\mathbf H^1}$}
& \multicolumn{2}{c|}{$\|p-p_h\|_{L^2}/\|p\|_{L^2}$} \\
\cline{2-11}
& Error & Rate
& Error & Rate
& Error & Rate
& Error & Rate
& Error & Rate \\
\hline
0.80000
& $2.594418\times10^{-2}$ & --
& $3.701015\times10^{-2}$ & --
& $6.132137\times10^{-3}$ & --
& $1.444854\times10^{-2}$ & --
& $3.908815\times10^{-2}$ & -- \\

0.40000
& $1.129266\times10^{-2}$ & 1.2000
& $1.603131\times10^{-2}$ & 1.2070
& $2.639902\times10^{-3}$ & 1.2159
& $9.390869\times10^{-3}$ & 0.9216
& $1.664807\times10^{-2}$ & 1.2314 \\

0.20000
& $5.417687\times10^{-3}$ & 1.0596
& $7.538947\times10^{-3}$ & 1.0885
& $1.243325\times10^{-3}$ & 1.0863
& $8.160341\times10^{-3}$ & 0.7026
& $9.095220\times10^{-3}$ & 0.9722 \\
\hline
\end{tabular}
}
\end{table}
\end{Example}
%











\medskip

\begin{Example}[Verification of the Space--Time Convergence Rates]\label{ex-spacetime}
This experiment tests the full error estimate of Theorem~\ref{Main theorem of error} by refining the mesh size $h$ and the time step $\Delta t$ simultaneously. We take the parameters
\[
\nu=1,\qquad \sigma=1,\qquad \mu=1,
\]
on the time interval $[0,1]$, starting from the coarse mesh $h_0=0.866$. The exact fields are built from the scalar potential
\[
\psi(x,y,z,t)=e^{-t}g(x)g(y)g(z),
\qquad
g(s)=s^{3}(1-s)^{3},
\]
whose triple zeros ensure that the fields and their first derivatives vanish on $\partial\mathbb{D}$. The exact velocity field is
\[
\mathbf{u}(x,y,z,t)=
\begin{pmatrix}
e^{-t}g(x)g'(y)g(z)\\[2mm]
-e^{-t}g'(x)g(y)g(z)\\[2mm]
0
\end{pmatrix}.
\]
\noindent
The exact magnetic field is
\[
\mathbf{B}(x,y,z,t)=
\begin{pmatrix}
e^{-t}g(x)g'(y)g(z)\\[2mm]
-e^{-t}g'(x)g(y)g(z)\\[2mm]
0
\end{pmatrix}.
\]
\noindent
The pressure is chosen as
\[
p(x,y,z,t)=0.
\]

with forcing terms chosen so that these fields solve the continuous $p$-MHD system exactly. The mesh and time step are refined together along the sequence
\[
(h,\Delta t)=(0.866,\,0.25),\ (0.433,\,0.0625),\ (0.2165,\,0.015625),
\]
that is, with $\Delta t\sim h^{2}$. The errors at the final time and the resulting rates are collected in Tables~\ref{tab:space-time-B}--\ref{tab:space-time-pressure}: the magnetic field and its curl in Table~\ref{tab:space-time-B}, the velocity in $\mathbf{L}^{2}$ and in the natural energy quantity $E(\u)$ in Table~\ref{tab:space-time-u}, and the pressure together with $E(\u)$ at $p=5$ in Table~\ref{tab:space-time-pressure}. Table~\ref{Dofs} lists the corresponding degrees of freedom, which grow rapidly under refinement and dominate the computational cost.

The observed rates are approximately second order for every variable and norm. Two points are needed to read this correctly. First, since $\Delta t\sim h^{2}$, the first-order temporal error of Theorem~\ref{Main theorem of error} contributes $O(\Delta t)=O(h^{2})$; the experiment is thus consistent with the theory but cannot separate the temporal from the spatial order, for which a refinement in $\Delta t$ alone (as in Example~\ref{ex-time-rate}) is required. Second, with $\mathbb{P}_{2}$ velocity elements the spatial rate $O(h^{2})=O(h^{l})$ is the optimal approximation order and therefore exceeds the guaranteed rate $O(h^{l/(p-1)})$ of Remark~\ref{rem-quasinorm} by the factor $p-1$. This is expected: the guaranteed rate in the natural $\mathbf{L}^{p}$-based norm is known not to be sharp for the $p$-Laplacian, and the computed rates instead attain the optimal quasi-norm benchmark. The scheme therefore meets or exceeds the proven bounds throughout.

\vspace{0.1cm}
\begin{table}[!htbp]
\renewcommand{\arraystretch}{1.2} \begin{tabular}{|c|c|cc|cc|} \hline $h$ & $\Delta t$ & \multicolumn{2}{c|}{$\B$ in $\L^2$} & \multicolumn{2}{c|}{$E( \B)$} \\ \cline{3-6} & & Error & Rate & Error & Rate \\ \hline 0.866000 & 0.250000 & $3.377968\times10^{-3}$ & -- & $1.001575\times10^{-2}$ & -- \\ 0.433000 & 0.062500 & $8.920750\times10^{-4}$ & 1.9209 & $2.633440\times10^{-3}$ & 1.9272 \\ 0.216500 & 0.015625 & $2.263063\times10^{-4}$ & 1.9789 & $6.643903\times10^{-4}$ & 1.9868 \\ \hline \end{tabular}
\centering

\vspace{0.1cm}
\caption{ Errors and convergence rates for the magnetic field $\B$ and  $E(B)=\left(\Delta t\sum_{n=1}^{m}\frac{1}{\mu^{2}\rho\sigma}\left\|\nabla\times(\B^{n}-\mathfrak{B}^{n})\right\|^{2}_{\L^{2}}\right)^{\frac{1}{p}}$ for $p=2$ under simultaneous refinement of $h$ and $\Delta t$.}
\label{tab:space-time-B}
\end{table}

\begin{table}[!htbp]
\centering
\setlength{\tabcolsep}{4pt}
\renewcommand{\arraystretch}{1.2}
\begin{tabular}{|c|c|cc|cc|}
\hline
$h$ & $\Delta t$ &
\multicolumn{2}{c|}{$\mathbf{u}$ in $L^2$} &
\multicolumn{2}{c|}{$E(\u)$} \\
\cline{3-6}
& & Error & Rate & Error & Rate \\
\hline
0.866000 & 0.250000 &
$5.837016\times10^{-4}$ & -- &
$7.252391\times10^{-3}$ & -- \\
0.433000 & 0.062500 &
$9.660032\times10^{-5}$ & 2.5951 &
$1.816466\times10^{-3}$ & 1.9973 \\
0.216500 & 0.015625 &
$1.981077\times10^{-5}$ & 2.2857 &
$4.512861\times10^{-4}$ & 2.0090 \\
\hline
\end{tabular}

\vspace{0.1cm}
\caption{Errors and convergence rates for the velocity field $\u$ in $\L^2$ and $E(\u)=\left(\Delta t\sum_{n=1}^{m}\frac{\nu}{\rho}\|\nabla(\u^{n}-\mathbf{U}^{n})\|^{p}_{\L^p}\right)^{\frac{1}{p}}$ for $p=2$ under simultaneous refinement of $h$ and $\Delta t$.}
\label{tab:space-time-u}
\end{table}
\vspace{0.1cm}
\begin{table}[!htbp]
\renewcommand{\arraystretch}{1.2} \begin{tabular}{|c|c|cc|cc|} \hline $h$ & $\Delta t$ & \multicolumn{2}{c|}{$E(\u)$ for $p=5$} & \multicolumn{2}{c|}{$p$ in $L^2$} \\ \cline{3-6} & & Error & Rate & Error & Rate \\ \hline 0.866000 & 0.250000 & $ 2.418142e-05 $ & -- & $2.497127\times10^{-2}$ & -- \\ 0.433000 & 0.062500 & $1.293298e-05$ & 0.9028 & $5.820674\times10^{-3}$ & 2.1010 \\ 0.216500 & 0.015625 & $ 4.382566e-06 $ & 0.9612& $1.409400\times10^{-3}$ & 2.0461 \\ \hline \end{tabular}
\centering

\vspace{0.1cm}
\caption{Errors and convergence rates for the pressure $p$ and $E(\u)$ for $p=5$.}
\label{tab:space-time-pressure}
\end{table}

\begin{table}[!htbp]

\centering
\renewcommand{\arraystretch}{1.2}
\begin{tabular}{|c|c|c|}
\hline
$h$ & $\Delta t$ & Degrees of Freedom \\ \hline
0.866000 & 0.250000 & 838  \\ 
0.433000 & 0.062500 & 5248 \\ 
0.216500 & 0.015625 &  36892\\ 
\hline
\end{tabular}
\vspace{0.1cm}

\caption{ mesh size, time step and total degrees of freedom (DOFs) along the refinement sequence.}
\label{Dofs}
\centering

\end{table}

\end{Example}

\section{Concluding Remarks and Future Directions}\label{sec-conclusion}

\subsection{Concluding remarks}\label{subsec-concluding}

In this work we constructed and analysed a fully discrete mixed finite element scheme for the three-dimensional incompressible magnetohydrodynamics system with nonlinear $p$-Laplacian viscosity in the shear-thickening regime $p\ge2$. The scheme couples a backward Euler discretisation in time with an inf-sup stable velocity--pressure pair and N\'ed\'elec edge elements for the magnetic field, and it treats the convective and electromagnetic couplings semi-implicitly so that the only nonlinearity to be solved at each step is the monotone $p$-Laplace operator.

The analysis establishes four properties, all holding for every $\Delta t>0$ and $h>0$. First, each time step is well posed (Theorem \ref{thm-existence}, Lemma \ref{lem-uniqueness}); the unconditional uniqueness follows from the semi-implicit design, which confines the nonlinearity to the monotone diffusion term, whereas an implicit treatment of the convection or Lorentz coupling would require a smallness condition. Second, the scheme is unconditionally stable in the natural energy norm (Lemma \ref{bounds}). Third, the discrete solutions converge to a weak solution as $h,\Delta t\to0$ (Theorem \ref{Convergence of the Fully Discrete Scheme}): strong compactness of the velocity is obtained from a time-translation estimate, since the $p$-Laplace term precludes a direct time-derivative bound, while that of the magnetic field follows from a direct bound in a negative-order norm; the nonlinear limit is then identified by Minty's monotonicity argument. Fourth, under additional regularity we derived unconditional a priori error estimates (Theorem \ref{Main theorem of error}) for both fields, namely
\[
\|\u^{m}-\mathbf{U}^{m}\|^{2}_{\L^{2}}
+\Delta t\sum_{n=1}^{m}\|\nabla(\u^{n}-\mathbf{U}^{n})\|^{p}_{\L^{p}}
\le C^{*}\left\{(\Delta t)^{2}+h^{\frac{lp}{p-1}}\right\},
\]
for the velocity and, simultaneously,
\[
\|\B^{m}-\mathfrak{B}^{m}\|^{2}_{\L^{2}}
+\Delta t\sum_{n=1}^{m}\|\nabla\times(\B^{n}-\mathfrak{B}^{n})\|^{2}_{\L^{2}}
\le C^{*}\left\{(\Delta t)^{2}+h^{\frac{lp}{p-1}}\right\}
\]
for the magnetic field in its natural $\mathbf{H}(\operatorname{\mathbf{curl}})$ energy norm. Both are first order in time, which is optimal for backward Euler, and of order $h^{lp/(p-1)}$ in space, reducing to the classical rate $h^{2l}$ at $p=2$; the spatial order reflects the known sub-optimality of $\mathbf{L}^{p}$-based error norms for the $p$-Laplacian (Remark \ref{rem-quasinorm}). Although the induction equation is linear in $\B$, the magnetic rate is limited by the same conjugate exponent $p'=p/(p-1)$ as the velocity, since the two error equations are coupled and closed by a single Gronwall argument.

The experiments of Section \ref{sec-numerics} are consistent with this theory and, where it is not sharp, exceed it: the observed spatial rates attain the optimal $O(h^{l})$ of the quasi-norm benchmark rather than the guaranteed $O(h^{l/(p-1)})$, and the $p=3  $ stability test confirms boundedness with no restriction on the time step. All source codes are publicly available at
\begin{center}
\url{https://github.com/evanasarkar987-ux/MHD-p-Laplace}.
\end{center}

\subsection{Future directions}\label{subsec-future}

Several questions are left open by the present analysis.

\medskip
\noindent\emph{(i) Optimal-order estimates in the quasi-norm.} The gap between the guaranteed spatial rate $O(h^{l/(p-1)})$ and the optimal rate $O(h^{l})$ observed numerically is a feature of the $\mathbf{L}^{p}$-based norm, not of the scheme. Measuring the error instead in the quasi-norm of Barrett and Liu \cite{BarrettLiu1994}, or in the natural distance $F(\nabla\u)=|\nabla\u|^{(p-2)/2}\nabla\u$ of Diening et al.\ \cite{DieningEbmeyerRuzicka2007,BerselliRuzicka2022}, should recover the optimal order and close this gap. We expect this to bring the guaranteed convergence rates into agreement with those observed in the numerical experiments.

\medskip
\noindent\emph{(ii) Higher-order and adaptive time stepping.} The scheme is first order in time. Second-order time discretisations, such as BDF2 or a Crank--Nicolson treatment of the linear terms combined with a second-order extrapolation of the semi-implicit couplings, would improve the temporal rate to $(\Delta t)^{2}$ in the energy norm. The trade-off is that the clean unconditional uniqueness of Lemma \ref{lem-uniqueness} rests on the fully implicit, monotone treatment of the diffusion term, and preserving it under a higher-order or extrapolated scheme requires care. Adaptive time-step selection driven by a computable estimator is a related direction.

\medskip
\noindent\emph{(iii) The shear-thinning regime $1<p<2$.} The present analysis is confined to $p\ge2$, where the $p$-Laplacian is monotone with the strong-monotonicity constant $2^{2-p}$ used throughout. The shear-thinning case $1<p<2$ is genuinely different: the operator is singular rather than degenerate, the monotonicity and the relevant Sobolev embeddings take a different form, and both the stability and the convergence arguments would have to be reworked. This regime is physically important for many non-Newtonian conducting fluids.


\appendix
\section{Auxiliary Lemmas}

\renewcommand{\thetheorem}{A.\arabic{theorem}}
\setcounter{theorem}{0}

\vspace{0.1cm}
\noindent
To facilitate the convergence of the solution to the fully discrete scheme, we first review the following embedding results (refer to \cite[Proposition 3.7]{GiraultRaviart1986} or \cite{AmroucheBernardiDaugeGirault1998}).

 \begin{lemma}\label{Embedding lemma}
  Let $\mathbb{D}$ be a Lipschitz polyhedron, then there exists a parameter $\delta_{1}=\delta_{1}(\mathbb{D})>0$ such that the embedding $\mathcal{D}(\mathbb{D})=\mathcal{Y}_{0}\cap \mathbf{H}(\operatorname{div};\mathbb{D}) \hookrightarrow \L^{\left(3+\delta_{1}\right)}(\mathbb{D})$ is compact and the space $\mathcal{D}(\mathbb{D})$ is equipped with the following norm:
       $$\|\v\|_{\mathcal{D}(\mathbb{D})}=\left(\|\operatorname{\mathbf{curl}}\v\|_{\mathbf{L^2(\mathbb{D})}}^{2}+\|\operatorname{div}\v\|^2_{\L^2(\mathbb{D})}\right)^\frac{1}{2} \quad \forall \v\in \mathcal{D}(\mathbb{D}). $$
    \end{lemma}
    \noindent
    We also need to recall the Aubin-Lions' compactness result for Bochner spaces (see \cite[Lemma 2.8]{GerbeauLeBrisLelievre2006}), which is crucial for our subsequent convergence analysis of the solution to the numerical scheme.
    \begin{lemma}\label{Aubin Lions}
        Let $\mathbb B$ be a Banach space, $\mathbb{B}_{0}$ and $\mathbb{B}_{1}$ be two reflexive Banach Spaces. Assume $\mathbb{B}_{0} \hookrightarrow\hookrightarrow \mathbb{B}$ with compact injection, $\mathbb B\subset \mathbb{B}_{1}$ with continuous injection(embedding). Then the space
        $$
        \left\{w|w\in L^{p_{0}}(0,T;\mathbb{B}_{0}),\frac{\partial w}{\partial t}\in L^{p_{1}}(0,T;\mathbb{B}_{1})\right\}\hookrightarrow\hookrightarrow L^{p_{0}}(0,T;\mathbb{B})
        $$
        with $1<p_{0}<+\infty$, $1<p_{1}<+\infty$.
        \end{lemma}
\noindent
 The following compactness property for discrete divergence-free vector fields generalizes the result in \cite[Theorem 4.9]{Hiptmair2002}. 
   \begin{lemma}\label{Lemma 1}
                
         Let $\left\{\C_{h}\right\} \subset \mathcal{Y}_{0h}$ be a sequence of fields, which is uniformly bounded in $\mathcal{Y}$. Then there exists a subsequence $\left\{\mathbf{w}_{h^{\prime}}\right\}_{h^{\prime}}$ converging weakly in $\mathcal{Y}$, and strongly in $\L^{2}(\mathbb{D})$ to a solenoidal function $\C \in \mathcal{K}$.
    \end{lemma}
    \noindent
    With the help of Hodge mapping and Lemma \ref{Embedding lemma}, we have the following estimate.
 \begin{lemma}\label{Z=curl}
            Under the conditions of Lemma \ref{Embedding lemma}, for some constant $C>0$, there holds
$$\|\mathcal{L}(\B)\|_{0,3+\delta_{1}} \leq C \|\nabla \times \B\|_{\mathbf{L}^2}.$$
        \end{lemma}
        \noindent
        From the Fortin criterion, the following discrete inf-sup condition is established (see \cite[Chapter 2]{Hiptmair2002}).
        \begin{lemma}[inf-sup condition]
 For any $q_{h} \in Q_{h}$, there exists $\v_{h} \in \mathcal{X}_{h}$ that satisfying the discrete inf-sup condition
$$\inf \limits_{0 \neq q_{h} \in Q_{h}} \sup\limits_{0 \neq v_{h} \in X_{h}} \frac{\left(\operatorname{div}\v_{h},q_{h}\right)}{\left\|\v_{h}\right\|_{1,2}\left\|q_{h}\right\|_{\mathbf{L}^2}}\geq \beta_{1}$$
where, $\beta_{1}$ is a generic positive constant depending on the domain $\mathbb{D}.$
        \end{lemma}
   \noindent Now we recall the following inverse estimate from \cite[Theorem 3.2.6]{Ciarlet1978}. On a quasi-uniform mesh there holds
        \begin{equation}\label{eq-inverse estimate}
\left\|\v_{h}\right\|_{m,q}\le C_{inv}h^{l-m+3(1/q-1/p)}\left\|\v_{h}\right\|_{l,p}, \end{equation}\quad 
\noindent for all $\v_{h}\in \mathcal{X}_{h}$, where $C_{inv}$ is a generic constant independent of mesh size $h$, $l$ and $m$ are two real numbers with $0\le l\le m\le 1$, $p$ and $q$ are two real numbers with $1 \le p \le q \le \infty.$\\
  The following result (see \cite{Deimling1985}) will be used in the existence proof .
 
 \begin{lemma}\label{lemma-fixed point}
     Let $f\in C(\mathbb{R}^n)$ be such that $\frac{(f(x),x)}{|x|}\rightarrow \infty$ as $|x|\rightarrow \infty$. Then $f(\mathbb{R}^n)=\mathbb{R}^n.$ 
\end{lemma}
\noindent
The following bounds will be useful for error analysis.
   \begin{lemma}\label{eq-error-estimate bounds}
       Let $\u,\B$ be the unique solution of the \eqref{1st weak form} and \eqref{2nd weeak equation}, then the following estimates are established
       \begin{align}
           \|\u-\mathcal{P}_{h}\u\|_{\L^{\infty}}+\|\nabla(\u-\mathcal{P}_{h}\u)\|_{\L^{3}}+\|\nabla\times(\B-\mathcal{F}_{h}\B)\|_{\L^3}\le C_r.
       \end{align}
   \end{lemma}
   \begin{proof}
      For $p=2$ this is \cite[Lemma 5.4]{ding2022convergence}. For $p>2$ the same argument applies verbatim once the $\mathbf{H}^{1}$-projection estimates used there are replaced by the $\mathbf{W}^{1,p}$ estimates \eqref{eq-error approximation}; the ingredients required are the quasi-uniformity of $\mathcal{T}_{h}$, the $\mathbf{L}^{\infty}$-stability of $\mathcal{P}_{h}$ on quasi-uniform meshes, and the regularity of Assumption \ref{Assumption}.
   \end{proof}

\section*{Conflict of Interest}
The authors declare that they have no conflict of interest.

\section*{Funding}
The second author gratefully acknowledges the support of the DST--INSPIRE Faculty Research Grant (Registration No.\ IFA21-MA158). This work was partially supported by the Australian Government through the Australian Research Council’s
Discovery Projects funding scheme (grant number DP220100937).

\section*{Data Availability}
The datasets generated and analysed during the current study are available from the corresponding author on reasonable request.

\vspace{0.4cm}
During the preparation of this work, the authors used ChatGPT for language correction and editing, and Codex for reviewing and correcting programming code. The authors reviewed and edited the outputs as needed and take full responsibility for the content of the published article.
\bibliography{sample}

\begin{thebibliography}{40}
\newcommand{\enquote}[1]{``#1''}
\providecommand{\natexlab}[1]{#1}
\providecommand{\url}[1]{\texttt{#1}}
\providecommand{\urlprefix}{URL }
\expandafter\ifx\csname urlstyle\endcsname\relax
  \providecommand{\doi}[1]{\discretionary{}{}{}https://doi.org/#1}\else
  \providecommand{\doi}[1]{\discretionary{}{}{}\urlstyle{rm}\url{https://doi.org/#1}}\fi

\bibitem[{Amrouche et~al.(1998)Amrouche, Bernardi, Dauge, and Girault}]{AmroucheBernardiDaugeGirault1998}
Amrouche, C., Bernardi, C., Dauge, M., and Girault, V., \enquote{Vector Potentials in Three-Dimensional Non-Smooth Domains,} \emph{Mathematical Methods in the Applied Sciences}, Vol.~21, 1998, pp. 823--864.

\bibitem[{Banas and Prohl(2010)}]{BanasProhl2010}
Banas, L., and Prohl, A., \enquote{Convergent finite element discretization of the multi-fluid nonstationary incompressible magnetohydrodynamics equations,} \emph{Math. Comp.}, Vol.~79, 2010, pp. 1957--1999.

\bibitem[{Barrett and Liu(1994)}]{BarrettLiu1994}
Barrett, J.~W., and Liu, W.~B., \enquote{Finite element approximation of the parabolic $p$-Laplacian,} \emph{SIAM Journal on Numerical Analysis}, Vol.~31, No.~2, 1994, pp. 413--428.

\bibitem[{Berselli et~al.(2015)Berselli, Diening, and Ruzicka}]{BerselliDieningRuzicka2015}
Berselli, L.~C., Diening, L., and Ruzicka, M., \enquote{Optimal error estimate for semi-implicit space-time discretization for the equations describing incompressible generalized Newtonian fluids,} \emph{IMA Journal of Numerical Analysis}, Vol.~35, No.~2, 2015, pp. 680--697.

\bibitem[{Berselli and Ruzika(2022)}]{BerselliRuzicka2022}
Berselli, L.~C., and Ruzika, M., \enquote{Space‐time discretization for nonlinear parabolic systems with $p$‐structure,} \emph{IMA Journal of Numerical Analysis}, Vol.~42, No.~1, 2022, pp. 260--299.

\bibitem[{Breit et~al.(2021)Breit, Diening, Storn, and Wichmann}]{BreitDieningStornWichmann2021}
Breit, D., Diening, L., Storn, J., and Wichmann, J., \enquote{The parabolic $p$‐Laplacian with fractional differentiability,} \emph{IMA Journal of Numerical Analysis}, Vol.~41, No.~3, 2021, pp. 2110--2138.

\bibitem[{Cai et~al.(2021)Cai, Li, and Li}]{CaiLiLi2021}
Cai, W., Li, B., and Li, Y., \enquote{Error analysis of a fully discrete finite element method for variable density incompressible flows in two dimensions,} \emph{ESAIM Math. Model. Numer. Anal.}, Vol.~55, 2021, pp. S103--S147.

\bibitem[{Cao and Wu(2011)}]{CaoWu2011}
Cao, C., and Wu, J., \enquote{Global Regularity for the Two-Dimensional MHD Equations with Mixed Partial Dissipation and Magnetic Diffusion,} \emph{Advances in Mathematics}, Vol. 226, 2011, pp. 1803--1822.

\bibitem[{Chen et~al.(2008)Chen, Miao, and Zhang}]{ChenMiaoZhang2008}
Chen, Q., Miao, C., and Zhang, Z., \enquote{On the Regularity Criterion of Weak Solutions for the 3D Viscous Magnetohydrodynamics Equations,} \emph{Communications in Mathematical Physics}, Vol. 284, 2008, pp. 919--930.

\bibitem[{Ciarlet(1978)}]{Ciarlet1978}
Ciarlet, P.~G., \emph{The Finite Element Method for Elliptic Problems}, Studies in Mathematics and its Applications, Vol.~4, North-Holland Publishing Co., Amsterdam--New York--Oxford, 1978.

\bibitem[{Costabel and Dauge(1998)}]{CostabelDauge1998}
Costabel, M., and Dauge, M., \enquote{Singularities of Maxwell's equations on polyhedral domains,} \emph{Analysis, Numerics and Applications of Differential and Integral Equations (Stuttgart, 1996)}, Pitman Research Notes in Mathematics Series, Vol. 379, Longman, Harlow, 1998, pp. 69--76.

\bibitem[{Costabel and Dauge(2000)}]{CostabelDauge2000}
Costabel, M., and Dauge, M., \enquote{Singularities of electromagnetic fields in polyhedral domains,} \emph{Archive for Rational Mechanics and Analysis}, Vol. 151, 2000, pp. 221--276.

\bibitem[{Davidson(2001)}]{Davidson2001}
Davidson, P.~A., \emph{An Introduction to Magnetohydrodynamics}, Cambridge Texts in Applied Mathematics, Cambridge University Press, Cambridge, 2001.

\bibitem[{Deimling(1985)}]{Deimling1985}
Deimling, K., \emph{Nonlinear Functional Analysis}, Springer-Verlag, Berlin, 1985.

\bibitem[{Diening et~al.(2007)Diening, Ebmeyer, and Ruzika}]{DieningEbmeyerRuzicka2007}
Diening, L., Ebmeyer, C., and Ruzika, M., \enquote{Optimal convergence for the implicit space--time discretization of parabolic systems with $p$-structure,} \emph{SIAM Journal on Numerical Analysis}, Vol.~45, No.~2, 2007, pp. 457--472.

\bibitem[{Ding and Li(2025)}]{DING2025116470}
Ding, Q., and Li, M., \enquote{Convergence analysis of finite element method for incompressible magnetohydrodynamics system with variable density,} \emph{Journal of Computational and Applied Mathematics}, Vol. 462, 2025, p. 116470.

\bibitem[{Ding et~al.(2022)Ding, Long, and Mao}]{ding2022convergence}
Ding, Q., Long, X., and Mao, S., \enquote{Convergence analysis of a fully discrete finite element method for thermally coupled incompressible MHD problems with temperature-dependent coefficients,} \emph{ESAIM: Mathematical Modelling and Numerical Analysis}, Vol.~56, No.~3, 2022, pp. 969--1005.

\bibitem[{Droniou et~al.(2022)Droniou, Goldys, and Le}]{MR4410739}
Droniou, J., Goldys, B., and Le, K.-N., \enquote{Design and convergence analysis of numerical methods for stochastic evolution equations with {L}eray-{L}ions operator,} \emph{IMA J. Numer. Anal.}, Vol.~42, No.~2, 2022, pp. 1143--1179.

\bibitem[{Eckstein and Ruzicka(2018)}]{EcksteinRuzicka2018}
Eckstein, S., and Ruzicka, M., \enquote{On the full space--time discretization of the generalized Navier--Stokes equations: the Dirichlet case,} \emph{SIAM Journal on Numerical Analysis}, Vol.~56, No.~4, 2018, pp. 2234--2261.

\bibitem[{Emmrich and Wróblewska-Kamińska(2013)}]{EmmrichWroblewska2013}
Emmrich, E., and Wróblewska-Kamińska, A., \enquote{Convergence of a full discretization of quasi-linear parabolic equations in isotropic and anisotropic Orlicz spaces,} \emph{SIAM Journal on Numerical Analysis}, Vol.~51, No.~3, 2013, pp. 1163--1184.

\bibitem[{Gerbeau et~al.(2006)Gerbeau, Bris, and Leli{\`e}vre}]{GerbeauLeBrisLelievre2006}
Gerbeau, J.-F., Bris, C.~L., and Leli{\`e}vre, T., \emph{Mathematical Methods for the Magnetohydrodynamics of Liquid Metals}, Numerical Mathematics and Scientific Computation, Oxford University Press, Oxford, 2006.

\bibitem[{Girault and Raviart(1986)}]{GiraultRaviart1986}
Girault, V., and Raviart, P.-A., \emph{Finite Element Methods for Navier--Stokes Equations: Theory and Algorithms}, Springer Series in Computational Mathematics, Vol.~5, Springer-Verlag, Berlin, 1986.

\bibitem[{Guermond and Quartapelle(2000)}]{GuermondQuartapelle2000}
Guermond, J.-L., and Quartapelle, L., \enquote{A projection {FEM} for variable density incompressible flows,} \emph{J. Comput. Phys.}, Vol. 165, 2000, pp. 167--188.

\bibitem[{He and Wang(2008)}]{HeWang2008}
He, C., and Wang, Y., \enquote{Remark on the Regularity for Weak Solutions to the Magnetohydrodynamic Equations,} \emph{Mathematical Methods in the Applied Sciences}, Vol.~31, 2008, pp. 1667--1684.

\bibitem[{Hiptmair(2002)}]{Hiptmair2002}
Hiptmair, R., \enquote{Finite Elements in Computational Electromagnetism,} \emph{Acta Numerica}, Vol.~11, 2002, pp. 237--339.

\bibitem[{Li et~al.(2022)Li, Qiu, and Yang}]{LiQiuYang2022}
Li, B., Qiu, W., and Yang, Z., \enquote{A convergent post-processed discontinuous {Galerkin} method for incompressible flow with variable density,} \emph{J. Sci. Comput.}, Vol.~91, No.~2, 2022, p.~28.

\bibitem[{Li and An(2021)}]{LiAn2021}
Li, Y., and An, R., \enquote{Temporal error analysis of Euler semi-implicit scheme for the magnetohydrodynamics equations with variable density,} \emph{Appl. Numer. Math.}, Vol. 166, 2021, pp. 146--167.

\bibitem[{Lions(1978)}]{Lions1978}
Lions, J.-L., \enquote{On some questions in boundary value problems of mathematical physics,} \emph{Contemporary Developments in Continuum Mechanics and Partial Differential Equations}, North-Holland Mathematics Studies, Vol.~30, North-Holland, Amsterdam-New York, 1978, pp. 284--346.

\bibitem[{Liu and Walkington(2007)}]{LiuWalkington2007}
Liu, C., and Walkington, N.~J., \enquote{Convergence of numerical approximations of the incompressible Navier-Stokes equations with variable density and viscosity,} \emph{SIAM Journal on Numerical Analysis}, Vol.~45, No.~3, 2007, pp. 1287--1304.

\bibitem[{Meir(1995)}]{Meir1995}
Meir, A.~J., \enquote{Thermally coupled, stationary, incompressible MHD flow; existence, uniqueness, and finite element approximation,} \emph{Numerical Methods for Partial Differential Equations}, Vol.~11, 1995, pp. 311--337.

\bibitem[{Monk(2003)}]{Monk2003}
Monk, P., \emph{Finite Element Methods for Maxwell's Equations}, Numerical Mathematics and Scientific Computation, Oxford University Press, New York, 2003.

\bibitem[{Moreau(1990)}]{Moreau1990}
Moreau, R., \emph{Magnetohydrodynamics}, Fluid Mechanics and its Applications, Vol.~3, Kluwer Academic Publishers Group, Dordrecht, 1990.
\newblock Translated from the French by A. F. Wright.

\bibitem[{Prohl(2008)}]{Prohl2008}
Prohl, A., \enquote{Convergent finite element discretizations of the nonstationary incompressible magnetohydrodynamics system,} \emph{M2AN Mathematical Modelling and Numerical Analysis}, Vol.~42, No.~6, 2008, pp. 1065--1087.

\bibitem[{Qiu(2020)}]{Qiu2020}
Qiu, H., \enquote{Error analysis of Euler semi-implicit scheme for the nonstationary magneto-hydrodynamics problem with temperature dependent parameters,} \emph{Journal of Scientific Computing}, Vol.~85, 2020, pp. 1--26.

\bibitem[{Ravindran(2019)}]{Ravindran2019}
Ravindran, S.~S., \enquote{Partitioned time-stepping scheme for an MHD system with temperature-dependent coefficients,} \emph{IMA Journal of Numerical Analysis}, Vol.~39, 2019, pp. 1860--1887.

\bibitem[{Sermange and Temam(1983)}]{SermangeTemam1983}
Sermange, M., and Temam, R., \enquote{Some Mathematical Questions Related to the MHD Equations,} \emph{Communications on Pure and Applied Mathematics}, Vol.~36, 1983, pp. 635--664.

\bibitem[{Tabata and Tagami(2005)}]{TabataTagami2005}
Tabata, M., and Tagami, D., \enquote{Error estimates of finite element methods for nonstationary thermal convection problems with temperature-dependent coefficients,} \emph{Numerische Mathematik}, Vol. 100, 2005, pp. 351--372.

\bibitem[{Walkington(2005)}]{Walkington2005}
Walkington, N.~J., \enquote{Convergence of the discontinuous Galerkin method for discontinuous solutions,} \emph{SIAM Journal on Numerical Analysis}, Vol.~42, No.~5, 2005, pp. 1801--1817.

\bibitem[{Wu(2004)}]{Wu2004}
Wu, J., \enquote{Regularity Results for Weak Solutions of the 3D MHD Equations,} \emph{Discrete and Continuous Dynamical Systems - Series A}, Vol.~10, 2004, pp. 543--556.

\bibitem[{Zhou(2006)}]{Zhou2006}
Zhou, Y., \enquote{Regularity Criteria for the 3D MHD Equations in Terms of Pressure,} \emph{International Journal of Non-Linear Mechanics}, Vol.~41, 2006, pp. 1174--1180.

\end{thebibliography}

\end{document}